\documentclass[12pt,reqno]{amsart}
\usepackage{amsmath, amsthm, amssymb}

\usepackage{comment}
\usepackage{color}
\usepackage{fancyhdr}

\usepackage[margin=1.25in]{geometry}

\usepackage{pdfsync}

\newtheorem{thm}{Theorem}[section]
\newtheorem{prop}[thm]{Proposition}
\newtheorem{lemma}[thm]{Lemma}

\newtheorem{condition}{Condition}

\newtheorem{dfn}[thm]{Definition}

\renewcommand{\b}[1]{\mathbf{#1}}
\renewcommand{\u}[1]{\underline{#1}}

\newcommand{\ep}{\varepsilon}

\newcommand{\con}{\equiv}

\newcommand{\ndiv}{\nmid}
\newcommand{\modd}[1]{\; ( \mathrm{mod} \; #1)}
\newcommand{\bstack}[2]{#1 \atop #2}

\newcommand{\maps}{\rightarrow}
\newcommand{\intersect}{\cap}

\newcommand{\Union}{\bigcup}

\newcommand{\diag}{\text{diag}\,}

\newcommand{\al}{\alpha}
\newcommand{\be}{\beta}
\newcommand{\gam}{\gamma}

\newcommand{\del}{\delta}
\newcommand{\Del}{\Delta}

\newcommand{\Sig}{\Sigma}
\newcommand{\sig}{\sigma}
\newcommand{\lam}{\lambda}
\newcommand{\Lam}{\Lambda}

\newcommand{\A}{\mathcal{A}}

\newcommand{\Ecal}{\mathcal{E}}
\newcommand{\Fcal}{\mathcal{F}}

\newcommand{\Ical}{\mathcal{I}}

\newcommand{\Rcal}{\mathcal{R}}

\newcommand{\Abf}{\mathbf{A}}

\newcommand{\bbf}{{\bf b}}

\newcommand{\hbf}{{\bf h}}

\newcommand{\lbf}{{\bf l}}

\newcommand{\rbf}{{\bf r}}

\newcommand{\ubf}{{\bf u}}

\newcommand{\wbf}{{\bf w}}
\newcommand{\x}{{\bf x}}
\newcommand{\xbf}{{\bf x}}
\newcommand{\ybf}{{\bf y}}

\newcommand{\unu}{\u{\nu}}

\newcommand{\F}{\mathbb{F}}

\newcommand{\N}{\mathbb{N}}

\newcommand{\Q}{\mathbb{Q}}
\newcommand{\R}{\mathbb{R}}

\newcommand{\Z}{\mathbb{Z}}

\newcommand{\Mf}{\mathfrak{M}}

\newcommand{\mf}{\mathfrak{m}}
\newcommand{\Sf}{\mathfrak{S}}

\newcommand{\rank}{\mathrm{rank}\,}

\newcommand{\Disc}{\mathrm{Disc}}

\newcommand{\beq}{\begin{equation}}
\newcommand{\eeq}{\end{equation}}

\numberwithin{equation}{section}

	\newcommand{\xtra}[1]{}
	
\definecolor{pink}{rgb}{1,.2,.6}
\definecolor{orange}{rgb}{0.7,0.3,0}
\definecolor{blue}{rgb}{.2,.6,.75}
\definecolor{green}{rgb}{.4,.7,.4}
\definecolor{purple}{RGB}{127,0,255}

\newtheorem{proposition}[thm]{Proposition}

\theoremstyle{remark}

\numberwithin{equation}{section}

\newcommand{\mmod}[1]{\,\,\text{mod}\,\,#1}

\newcommand{\lambf}{\boldsymbol{\lambda}}
\def\bfa{{\mathbf a}}
\def\bfb{{\mathbf b}}
\def\bfc{{\mathbf c}}
 \def\bfe{{\mathbf e}}

\def\bfh{{\mathbf h}}

\def\bfl{{\mathbf l}}

\def\bfr{{\mathbf r}}

\def\bfu{{\mathbf u}}
\def\bfv{{\mathbf v}}
\def\bfw{{\mathbf w}}
\def\bfx{{\mathbf x}}
\def\bfy{{\mathbf y}}
\def\bfz{{\mathbf z}}

\def\calJ{{\mathcal J}}
\def\calL{{\mathcal L}}

\def\calN{{\mathcal N}}

\def\calQ{{\mathcal Q}}
 
\def\calS{{\mathcal S}}

\def\calX{{\mathcal X}}
\def\calY{{\mathcal Y}}

\def\A{{\mathbb A}}
\def\F{{\mathbb F}}\def\N{{\mathbb N}}\def\P{{\mathbb P}}
\def\R{{\mathbb R}}
\def\Z{{\mathbb Z}}\def\Q{{\mathbb Q}}

\def\grS{{\mathfrak S}}

\def\alp{{\alpha}} 
\def\alptil{{\widetilde{\alp}}}
\def\bet{{\beta}}  
\def\gam{{\gamma}}

\def\del{{\delta}}
 \def\Del{{\Delta}}
\def\zet{{\zeta}}  

\def\tet{{\theta}}  
\def\kap{{\kappa}}
\def\lam{{\lambda}} 
\def\bflam{{\boldsymbol \lam}}
\def\Lam{{\Lambda}} 
\def\Lamm{{\Lambda}}

\def\sig{{\sigma}} \def\Sig{{\Sigma}}

\def\ome{{\omega}} \def\Ome{{\Omega}}
\def\d{{\partial}}
\def\eps{\varepsilon}

\def\d{{\,{\rm d}}}

\def\rank{{\rm rank}}

\def\dim{{\rm dim}}

\def\bQ{{\underline{Q}}}
\def\ba{{\underline{a}}}

\def\bb{{\underline{b}}}
\def\bn{{\underline{n}}}

\def\bphi{{\underline{\phi}}}
\def\btet{{\underline{\tet}}}
\def\bmu{{\underline{\mu}}}
\def\bnu{{\underline{\nu}}}
\def\blam{{\underline{\lam}}}

\DeclareMathOperator{\codim}{codim}

\begin{document}
\title[Systems of three quadratic forms]{Representations of integers by systems of three quadratic forms}

\author[Pierce]{Lillian B. Pierce}
\address{Mathematics Department, Duke University, 120 Science Drive, Durham, NC 27708 USA}
\email{pierce@math.duke.edu}

\author[Schindler]{Damaris Schindler}
\address{Hausdorff Center for Mathematics, Endenicher Allee 60-62, 53115 Bonn, Germany. Present address: Institute for Advanced Study, Einstein Drive, Princeton NJ 08540 USA}
\email{damaris.schindler@hcm.uni-bonn.de}

\author[Wood]{Melanie Matchett Wood}
\address{Department of Mathematics, 480 Lincoln Dr., Madison, WI 53706 USA\\
and
American Institute of Mathematics\\600 East Brokaw Road\\
San Jose, CA 95112 USA}
\email{mmwood@math.wisc.edu}

\date{}

\maketitle


\begin{abstract}
It is classically known that the circle method produces an asymptotic for the number of representations of a tuple of integers $(n_1,\ldots,n_R)$  by a system of quadratic forms $Q_1,\ldots, Q_R$ in $k$ variables, as long as $k$ is sufficiently large with respect to $R$; reducing the required number of variables remains a significant open problem. In this work, we consider the case of 3 forms and  improve on the classical result by reducing the number of required variables to $k \geq 10$ for ``almost all'' tuples, under a nonsingularity assumption on the forms $Q_1,Q_2,Q_3$. 
To accomplish this, we develop a three-dimensional analogue of Kloosterman's circle method, in particular capitalizing on geometric properties of appropriate systems of three quadratic forms.

\end{abstract}

\section{Introduction}
The study of representing an integer by an integral quadratic form has a long history, and is today relatively well understood. More generally, one may consider a system of  integral quadratic forms $Q_1, \ldots, Q_R \in \Z[x_1,\ldots, x_k]$, and ask for the number of simultaneous representations of a fixed tuple of integers $\bn= (n_1,\ldots, n_R)$ by the system. More precisely,
we define a smoothly-weighted counting function by setting
\begin{equation*}
\Rcal_B(\bn)= \sum_{\substack{\bfx\in \Z^k\\ \bQ(\bfx)=\bn}} w\left(\frac{\bfx}{B}\right),
\end{equation*}
where  $w: \R^k\rightarrow \R$ is a smooth non-negative weight function of compact support, $B$ is a large parameter, and 
the notation $\bQ(\bfx)=\bn$ represents the system of equations 
\[ Q_i(\xbf) = n_i, \qquad 1 \leq i \leq R.\]
We  note that throughout we use the convention that a quadratic form $Q$ is said to be integral if it has an integral matrix (also denoted by $Q$); that is, the off-diagonal terms of the form have even coefficients. 
 
The expectation is that as long as $k \geq k(R)$ is sufficiently large and the system is not too singular, for $B$ sufficiently large the circle method will produce for each tuple $\u{n}$ an asymptotic of the shape
 \beq\label{R_asympt}
\Rcal_B(\u{n}) = C_{w, B}(\u{n})B^{k-2R} + o(B^{k-2R}).
 \eeq
Here $C_{w, B}(\u{n})$ is the product of the standard singular integral and singular series, which may be shown to be non-zero as long as $\u{n}$ lies in a suitable range depending on $B$, and satisfies appropriate local conditions. We note that it is natural to expect a main term of size $B^{k-2R}$, since there are $k$ choices for the variables $x_1,\ldots, x_k$, and $R$ constraint equations on the values $Q_i(\xbf)$, which are quantities of size at most $O(B^2)$.

Birch's work \cite{Bir62} provided long-standing  records for using the circle method to count solutions to systems of forms of any degree. Recently Myerson \cite{My16} has improved on Birch's theorem for nonsingular systems of quadratic forms; his result is much stronger than \cite{Bir62} in the case of four or more quadratic forms, but does not improve on the known bounds for two or three quadratic forms.
In the case of a nonsingular system of two quadratic forms, the main result of \cite{Bir62}, combined with the sharpened notion of the Birch singular locus in \cite{Die15} and \cite{Sch15}, produces an asymptotic for the counting function $\Rcal_B(\u{n})$, for any fixed $\u{n} =(n_1,n_2)$, as long as $k \geq 14$. Recent work of Munshi \cite{Mun13} using a ``nested'' version of the $\del$-circle method \cite{HB96} has reduced this to $k \geq 11$. 
For a nonsingular system of three quadratic forms, Birch's method (in the version \cite{Die15} and \cite{Sch15}) provides the benchmark that (\ref{R_asympt}) holds for $k \geq 27$.

Alternatively, one can aim for the weaker result of proving that the expected asymptotic (\ref{R_asympt})  for $\Rcal_B(\u{n})$ holds for ``almost all'' $\u{n}$. Birch's work can again be applied in this setting, showing in the case of a nonsingular system of two quadratic forms that $k \geq 8$ variables suffice, and in the case of three forms,  $k\geq 15$ variables suffice.
Recently, Heath-Brown and the first author \cite{HBPierce} took up the question of improving on this for systems of two quadratic forms by developing a so-called 2-dimensional version of Kloosterman's circle method; the version of the circle method in \cite{HBPierce} shows that $k \geq 5$ variables suffice to prove that (\ref{R_asympt}) holds for ``almost all'' pairs $\u{n}=(n_1,n_2)$.

In this paper we develop a version of the circle method that carries through Kloosterman's strategy in 3 dimensions, in order to treat systems $\u{Q} = \{Q_1,Q_2,Q_3\}$ of three quadratic forms satisfying a suitable nonsingularity condition.
The case of 3 quadratic forms is significantly more complicated geometrically than a system of 2 quadratic forms, and we will devote Section \ref{sec_geom} to a thorough exposition of the necessary geometric constructions. Here we briefly recall that it is standard to say that the system of equations 
\beq\label{3eqnsQ}
 Q_i(\xbf)=0 \quad \text{ for $1 \leq i \leq 3$}  
 \eeq
 is nonsingular if $Q_1,Q_2,Q_3$ satisfy the Jacobian criterion  for smoothness:  at any non-zero point $\xbf$ satisfying the equations (\ref{3eqnsQ}), the $3 \times k$ Jacobian matrix 
\[ 
 \left( \frac{\partial Q_i}{\partial x_j} (\xbf) \right)_{i,j}
\]
has full rank. 
As we will explain later in more detail, the Jacobian criterion alone is not sufficient for our approach. 
Instead, we define the form
\[ F_{\u{Q}}(x,y,z) =\det (xQ_1 + yQ_2 + zQ_3).\]
The key nonsingularity condition we will assume throughout is that the form $F_{\u{Q}}$ has nonvanishing discriminant as a function of $x,y,z$, or, equivalently that $F_{\u{Q}}$ defines a smooth plane curve as a subscheme of the projective plane.  (Unless $F_{\u{Q}}$ is a perfect power, this is also equivalent to $F_{\u{Q}}$ defining a smooth plane curve as a \emph{subvariety} of the projective plane.) 
We note that this assumption is not overly restrictive; indeed, for a generic choice of three quadratic forms $Q_1,Q_2,Q_3$, we will indeed have 
\beq\label{F_disc0}
\Disc(F_{\u{Q}}(x,y,z)) \neq 0;
\eeq
see Proposition \ref{prop4}. In particular, we provide a specific example of a system that satisfies (\ref{F_disc0}) in an appendix in Section \ref{sec_appendix}.

Under the assumption that $\Disc(F_{\u{Q}}(x,y,z)) \neq 0$, Birch's method only requires $k \geq 14$ for an ``almost all'' result (see the remark following Proposition \ref{prop1}). 
 Our main result yields a significant improvement of this: we show  that $k \geq 10$ variables suffice.

\begin{thm}\label{thm_asymp}
Let $Q_1,Q_2,Q_3 \in \Z[x_1,\ldots, x_k]$ be integral quadratic forms satisfying (\ref{F_disc0}).
Let $N=B^2$. If $k \geq 10$, for $B$ sufficiently large and for each $\varpi \in (0,1/56)$, there is an exceptional set $\Ecal_\varpi(N) \subset [-N,N]^3$ with
\[ |\Ecal_\varpi(N)| \ll_{\varpi, k, \u{Q}} N^{3-(1/56 - \varpi)},\]
 such that for any $\u{n} \in [-N,N]^3 \setminus \Ecal_\varpi(N)$,   the counting function $\Rcal_B(\u{n})$ admits the expected asymptotic 
\beq\label{R_asymp_thm}
\Rcal_B(\bn)=J_w (B^{-2}\bn)\grS(\bn)B^{k-6} + O_{\varpi, k, \u{Q}}(B^{k-6-\varpi}).
\eeq
Here  $J_w(B^{-2}\bn)$ and $\grS(\bn)$ denote the standard singular integral and singular series, which are defined in (\ref{sing_int_dfn}) and (\ref{sing_ser_dfn}), respectively.
\end{thm}
Here, as throughout, we use the convention that the notations $\ll_\lam$ and $O_\lam(\cdots)$ indicate an implied constant that may depend on the parameter $\lam$. (Since all implied constants may depend on the fixed weight function $w$ we have chosen, we will not typically notate this dependence.)

Theorem \ref{thm_asymp} is a corollary of the following more technical theorem, which is our main result.
To state the theorem, we make use of a polynomial $H_{\u{Q}}\in \Z[t_1,t_2,t_3]$, defined in Proposition~\ref{prop3.4}, that vanishes exactly when the three equations $\u{Q}(\bfx)=\u{t}$ fail the Jacobian criterion  for smoothness and codimension $3$ (i.e. when their intersection fails to be a smooth scheme of codimension $3$).

\begin{thm}\label{thm_meansquare}
Let $Q_1,Q_2,Q_3 \in \Z[x_1,\ldots, x_k]$ be integral quadratic forms satisfying (\ref{F_disc0}).
Then there exists a non-zero polynomial $H_{\u{Q}}\in \Z[t_1,t_2,t_3]$ that depends only on $k$ and the forms $Q_1,Q_2,Q_3$, and with degree at most $2(k+1)k(k-1)$, such that the following holds. 
For $k \geq 10$,
\begin{equation*}
\sum_{\substack{|\u{n}|_\infty \leq B^2\\ H_{\u{Q}}(\bn)\neq 0}} \left|\Rcal_B(\bn)-\calJ_w (B^{-2}\bn)\grS(\bn)B^{k-6}\right|^2 \ll_{k,\bQ,w} B^{2k-6-1/28},
\end{equation*}
where $\calJ_w(B^{-2}\bn)$ and $\grS(\bn)$ are the singular integral and singular series defined in (\ref{sing_int_dfn}) and (\ref{sing_ser_dfn}), respectively. Here $|\u{n}|_\infty = \max_{1 \leq i \leq 3} |n_i|$. 
\end{thm}

 One may deduce from Theorem \ref{thm_meansquare} that the Hasse principle holds on average for the representability of a tuple of integers $\bn$ by a system of forms $\bQ(\bfx)$; this requires knowledge of the size of the singular series $\grS (\bn)$ and singular integral $\calJ_w (B^{-2}\bn)$, which is now provided in the following theorem:
\begin{thm}\label{thm_SJ}
Let $Q_1,Q_2,Q_3 \in \Z[x_1,\ldots, x_k]$ be integral quadratic forms satisfying (\ref{F_disc0}),
and let $H_{\u{Q}}$ be the corresponding polynomial provided by Theorem \ref{thm_meansquare}. 
Then if $k>6$ and  $H_{\u{Q}}(\bn)\neq 0$, 
 \[\grS(\bn)\ll_{\eps,k,\bQ} |\u{n}|_\infty^{\eps}.\]
  for any $\eps >0$. Furthermore, there is a fixed prime $p_0$ and a positive real $\al >0$ depending only on $k$ and $\u{Q}$, such that the singular series $\grS(\bn)$ satisfies the lower bound
\begin{equation*}
\grS(\bn)\gg_{\eps ,k,\bQ} |\u{n}|_\infty^{-\eps}\prod_{p\leq p_0}|H_{\u{Q}}(\bn)|_p^{\alp},
\end{equation*}
 for any $\eps >0$. Here we write $|\cdot|_p$ for the standard $p$-adic metric on $\Q$. 
 As a consequence, under the above hypotheses, if the system $\u{Q}(\xbf) = \u{n}$ is solvable in every $p$-adic ring $\Z_p$, then $\grS(\u{n})$ is real and positive.
 
Second, let $w$ be a smooth weight of compact support and assume that $k>6$. Then the singular integral satisfies 
\[ J_w (\bmu)\ll_{w,\u{Q}} 1.\]
Furthermore there are positive constants $C$ and $\be$ depending only on $k$ and  $\u{Q}$, with the following property. If $w (\bfx) >0$ for $|\bfx|\leq C$, then we have 
\begin{equation}\label{JggH}
J_w (\bmu)\gg_w |H_{\u{Q}}(\bmu)|^{\be},
\end{equation}
for any $\bmu$ in the region $1/2\leq |\u{\mu}|_\infty \leq 1$, for which the system of equations $\bQ(\bfx)=\bmu$ has a solution $\bfx\in \R^k$.
\end{thm}
We note that in Theorem \ref{thm_SJ} we may in fact take
\[ \al = 3^{3+k}(2k-6), \qquad \be =3^{3+k}(2k-7).\]

A standard procedure allows one to deduce the following result from Theorems \ref{thm_meansquare} and \ref{thm_SJ}:

\begin{thm}\label{thm_exceptions}
If $k \geq 10$, there exists a positive constant $\varpi>0$,  depending only on $k$, such that the following holds.
Let $Q_1,Q_2,Q_3 \in \Z[x_1,\ldots, x_k]$ be integral quadratic forms satisfying (\ref{F_disc0}).
Let $E(N)$ be the set of $\bn\in \Z^3$ with $|\u{n}|_\infty \leq N$ such that the system $\bQ(\bfx)=\bn$ is locally solvable in $\R$ and in $\Z_p$ for every prime $p$, but has no integer solution in $\bfx\in \Z^k$. Then for $N$ sufficiently large,
\begin{equation}\label{E_upper_bound}
|E(N)|\ll_{\varpi,k,\u{Q}} N^{3-\varpi}.
\end{equation}
\end{thm}
Note that $\varpi$ is independent of the system $\u{Q}$. While we have not attempted to optimize the exponent $\varpi$, we note that we may presently take 
 \[  \varpi =  \frac{1}{2^5.7.3^{3+k}(4k-13)(k+1)k(k-1)+1}.\]

In particular, Theorem \ref{thm_exceptions} asserts that the Hasse principle holds for the representability of almost all tuples of integers $\bn$ by systems of three generic quadratic forms in at least 10 variables.
We note that due to the power gain in (\ref{E_upper_bound}),
 we may deduce from Theorem \ref{thm_exceptions} a result on the representability of primes:
\begin{thm}\label{thm_primes}
Let $Q_1,Q_2,Q_3 \in \Z[x_1,\ldots, x_k]$ be integral quadratic forms satisfying (\ref{F_disc0}).
 Suppose furthermore, that there is some $\bfx_0\in \R^k$ such that $ Q_i(\bfx_0)>0$ for all $1\leq i\leq 3$, and that for each prime $q$ there is some $\bfx_q\in (\Z_q)^k$ for which $q\nmid \prod_{i=1}^3 Q_i(\bfx_q)$. Then if $k \geq 10$, there are infinitely many tuples $(p_1,p_2,p_3)$ of primes that can be represented simultaneously by the system of forms $Q_1,Q_2,Q_3$.
\end{thm}

Further corollaries may be deduced as in \cite{HBPierce}. For example, one may prove asymptotic formulas for the number of representations of tuples $n_1,n_2,n_3$ by 
\begin{equation*}
n_i=Q_i'(\bfx)+Q_i''(\bfy),\quad 1\leq i\leq 3,
\end{equation*}
where $\bQ' (\bfx)$ and $\bQ'' (\bfx)$ are systems of quadratic forms in $k \geq 10$ variables satisfying the conditions in Theorem \ref{thm_meansquare}, and $\bfx, \bfy \in \Z^k$ are restricted to some box of side-length $B$. Moreover, one may choose $n_i=0$ for $1\leq i\leq 3$ in this setting. We omit the details, as they are very similar to those found in \cite{HBPierce}.

\subsection{The method of proof}
The key to our approach is the development of a 3-dimensional application of the circle method that allows us to extract cancellation in the style of Kloosterman's classical one-dimensional method \cite{Kloosterman}, which he developed for treating a single quadratic form in $4$ variables. In spirit, the structure of our approach is analogous to the work of \cite{HBPierce}, but at a technical level the 3-dimensional case requires a new treatment at many points of the argument:  the geometry of a system of 3 forms is fundamentally, not just cosmetically, distinct from that of a system of 2 forms. 
The immediate technical impacts of the more sophisticated (and inexplicit) geometric perspective we now adopt may be found in Propositions \ref{prop1}, \ref{prop4}, \ref{prop3.4}, Lemmas \ref{lem5.6}, \ref{lemL}, and  \ref{lemreduction}, and the appendix in Section \ref{sec_appendix}.

To aid the reader, we will now describe in more detail a sampling of the aspects of our work that necessitated new approaches; to situate these points we start with a brief sketch of the over-all strategy.
As is usual in applications of the circle method, we express the counting function $\Rcal_B(\u{n})$ as an integral over the $3$-dimensional unit cube, which we divide into an integral over so-called major arcs and minor arcs. The major arcs we are able to treat in a relatively standard manner, producing the main term in the asymptotic (\ref{R_asymp_thm}); for this work $k > 6$ variables would suffice. 
 We show that the minor arcs contribute a smaller error term by showing that they are small in a mean-square average sense. This method of using a mean-square argument to gain an improved bound on the minor arcs, and thus  produce an asymptotic for $\Rcal_B(\u{n})$ for almost all $\u{n}$ with fewer required variables, dates back to work of Hardy and Littlewood \cite{HL25}.
Our main effort goes into bounding the mean-square contribution of the minor arcs; we are led to consider a new system of forms given by 
\[ Q_i (\xbf) - Q_i(\ybf) = 0, \qquad \text{for $i = 1,2,3,$ with $\xbf, \ybf \in \Z^k$}.
\]
We then parametrize the relevant integral over the minor arcs by summing over (possibly overlapping) 3-dimensional boxes centered at rational tuples $(a_1/q, a_2/q,a_3/q)$. Our method generalizes Kloosterman's one-dimensional strategy by extracting cancellation between the contribution of boxes centered at  tuples $(a_1/q, a_2/q,a_3/q)$ and $(a_1'/q, a_2'/q,a_3'/q)$ with distinct numerators but identical denominators. 
A further distinctive feature of our analysis of the minor arcs---also present in \cite{HBPierce}---is that it avoids any analysis at the real place (such as an application of Weyl's inequality).

Let us now survey a few of the new features we encounter when considering a system of 3 forms.
We will first examine the nonsingularity condition (\ref{F_disc0}) that we impose in all of our theorems.
We observe that for any fixed tuple $\u{a} = (a_1,a_2,a_3) \in \R^3$, $F_{\u{Q}}(\u{a})$ vanishes precisely when the quadratic form given by the linear combination  
\[ a_1 Q_1(\bfx)+a_2 Q_2(\bfx)+a_3 Q_3(\bfx)
\]
 is singular. Thus imposing (\ref{F_disc0}) requires essentially that the locus of singular elements in the linear system generated by the quadratic forms $Q_1, Q_2, Q_3$ is a smooth projective curve, or more precisely that the projective scheme $F_{\u{Q}}(x,y,z)=0$ is nonsingular of dimension $1$. 
 
We recall from \cite{HBPierce} that for a system of two integral quadratic forms $Q_1,Q_2$, the analogous requirement that 
\beq\label{2forms_cond}
\Disc( \det(xQ_1 + yQ_2)) \neq 0
\eeq
is \emph{equivalent} to the requirement that the system
\begin{eqnarray}
Q_1(\xbf)&=&0 \nonumber\\ 
Q_2(\xbf)&=&0 \label{Q120}
\end{eqnarray}
satisfies the  Jacobian criterion to have smooth, codimension $2$ projective intersection. 
In addition, it was observed in \cite{HBPierce} that it is simple to construct forms $Q_1, Q_2$ satisfying (\ref{2forms_cond}) by choosing any two diagonal quadratic forms $Q_1 (\xbf)= \sum_i a_i x_i^2$ and $Q_2(\xbf) = \sum_i b_i x_i^2$ such that all the ratios $a_i/b_i$ are distinct.

Here we face  a striking difference between systems of two forms and systems of three forms: (\ref{F_disc0}) need not  hold for nonsingular systems of three diagonal forms. For example, the variety 
\begin{eqnarray*}
 x_3^2 + x_4^2 & =&0 \\
  x_1^2 + x_2^2+x_3^2 \,  \quad \quad &=&0\\
  x_1^2 + 2x_2^2+x_3^2  + x_4^2 & =& 0
\end{eqnarray*}
is projectively nonsingular, but the corresponding form
\[F_{\u{Q}}(x,y,z) = (y+z)(y+2z)(x+y+z)(x+z)\]
 is a product of four lines, and is hence singular at the intersection points, so that $\Disc (F_{\u{Q}}) =0.$
In fact, as we demonstrate in an appendix in Section \ref{sec_appendix}, all systems of three diagonal quadratic forms in at least $k \geq 2$ variables fail (\ref{F_disc0}). Thus our work in Theorem \ref{thm_meansquare} does not apply even to nonsingular systems of diagonal forms; however, we remark that in the case of 3 diagonal forms, one could replace our minor arc estimates by work of Cook \cite{Cook73} and already deduce the desired results for only $k\geq 7$ variables.\par

The function $F_{\u{Q}}(x,y,z)$ plays an important role in our work, just as the analogous function $\det(xQ_1 + yQ_2)$ played a role in \cite{HBPierce}, but even once we restrict our attention to systems $\u{Q}$ for which $F_{\u{Q}}$ has nonvanishing discriminant, we still lack several convenient properties that \cite{HBPierce} exploited for two forms. For example, in place of factorizing $F_{\u{Q}}(x,y,z)$ into linear factors, we must now utilize its nonsingularity (see Lemma \ref{lem15.1}, for example). 
In \cite{HBPierce}, to count solutions to $\det (xQ_1 + yQ_2)=0$ modulo prime powers, one could rely on work of Huxley.  Instead, we must now develop a multi-variable Hensel's lemma for varieties singular over $\F_p$ (see for example Lemma \ref{Hensel} as applied in Lemmas \ref{lemF} and \ref{lemZ}).

Second, we next recall that \cite{HBPierce} detected those pairs $(n_1,n_2)$ for which the system $Q_1(\xbf)=n_1, Q_2(\xbf)=n_2$ was nonsingular by simply requiring $\det (n_2Q_1 - n_1Q_2) \neq 0$. In the case of three forms, we must prove the existence of a  polynomial $H_{\u{Q}}(\u{n})$ which functions analogously; in our work, this polynomial is now inexplicit. As part of constructing this polynomial, we make an integral version of the classical construction of the discriminant of multiple forms over a field, since we need a single polynomial that works over $\overline{\Q}$ as well as modulo primes (see Section \ref{sec_appendix}).
Ultimately, we must then be able to estimate the cardinality of sublevel sets for $H_{\u{Q}}(\u{n})$, both over $\Z$ and $p$-adically.  Lemmas \ref{lemPoly2} and \ref{lemPoly4}  develop such estimates in a general context.

Third, in \cite{HBPierce} the nonsingularity of the system (\ref{Q120}) over an algebraically closed field $K$ of characteristic zero is equivalent to the condition that $Q_1,Q_2$ are simultaneously diagonalizable over $K$  (see Proposition 2.1 of \cite{HBPierce}). In contrast, a generic system of three quadratic forms is not simultaneously diagonalizable; to replace diagonalization we thus bring in new ideas, such as a version of the Nullstellensatz applied to the polynomial $H_{\u{Q}}$ mentioned above (see for example Lemma \ref{lem5.6}).

We have already noted several distinct differences between the geometry of systems of two quadratic forms versus three quadratic forms. It is worth asking  whether a similar method can produce an $R$-dimensional Kloosterman method that will prove an ``almost all'' asymptotic result for $\Rcal_B(\u{n})$ for a system of $R \geq 4$ quadratic forms. One immediately encounters a fundamental barrier: when $R \geq 4$, it can be shown that \emph{all} systems of integral quadratic forms $Q_1,\ldots, Q_R$ will satisfy
\[ \Disc(\det(x_1Q_1+ \cdots + x_R Q_R))=0.\]
Thus to generalize the present approach to systems of four or more quadratic forms, significant new ideas would be required.

\section{Notation}\label{sec_notation}

Throughout, we use the convention that boldface symbols, such as $\bfx$, $\bfy$, or $\lbf$  denote vectors of dimension $k$ or $2k$, while underlined symbols, such as $\bQ$, $\u{a}$, or $\u{n}$,  denote vectors of dimension three. We use $|\u{n}|_\infty$ to denote $\max_{1 \leq i \leq 3} |n_i|$, and we write $(\u{n},q)$ for $\gcd(n_1,n_2,n_3,q)$. We write $\bfx\cdot\bfy$ for the usual scalar product of two vectors $\bfx$ and $\bfy$, and similarly for $\u{a} \cdot \u{n}$. By $\u{a}\cdot \u{Q}(\xbf)$ we refer to the quadratic form given by the linear combination $a_1Q_1(\xbf) + a_2Q_2(\xbf)+a_3Q_3(\xbf)$. For a quadratic form $Q$, we use the norm $\Vert  Q \Vert = \sup_{|\xbf|=1} |Q(\xbf)|$.

We use the convention for quadratic forms in $k$ variables that we write $Q$ for both the quadratic form and its associated $k\times k$ matrix. We use the normalization for an integral quadratic form $Q$ that $Q(\bfx)=\bfx^tQ\bfx$ where $Q$ is a $k\times k$ matrix with integral entries; the off-diagonal terms of $Q$ thus have even coefficients.

We write $\overline{K}$ for some algebraic closure of a field $K$. We use the standard notation $e(x)$ for $e^{2\pi i x}$.
Given a tuple $\u{\phi} = (\phi_1,\phi_2,\phi_3)$ we let $\int_{\{ \u{\phi} \}}$ denote integration over the region 
\[ \prod_{i=1}^3 ([-2\phi_i,-\phi_i]\cup [\phi_i,2\phi_i]).\]

Throughout we apply the convention that implied constants may depend on the weight $w$, the dimension $k$, and the fixed system $\u{Q}$. 

\section{Geometric considerations}\label{sec_geom}
We first give precise formulations of three conditions one may impose on integral quadratic forms $Q_1,Q_2,Q_3$, and  specify the relationships between these conditions. In this section we will let $K$ represent an algebraically closed field of characteristic zero or odd characteristic, and we will then consider the following conditions as being over $K$. If $K$ has characteristic zero, the field of definition for the forms $Q_i$ will be taken to be $\Q$.

We define the $3 \times k$ Jacobian matrix 
\beq\label{Jac_dfn}
 J_{\u{Q}} (\xbf) = \left( \frac{\partial Q_i}{\partial x_j} (\xbf) \right)_{i,j}.
 \eeq
We now define three conditions on triples of quadratic forms; we will assume throughout this paper that we work in $k\geq 4$ variables.
\begin{condition}\label{cond1}
For any non-zero point $\xbf \in K^k$ lying on the variety
\beq\label{cond1_eqn}
Q_i(\xbf)=0 \quad \text{for $1\leq i\leq 3$},
\eeq
 we have
\beq\label{cond1_J}
 \rank(J_{\u{Q}}(\xbf)) =3.
 \eeq
\end{condition}
If the ideal $(Q_1(\xbf),Q_2(\xbf),Q_3(\xbf))$ is reduced, this 
is equivalent to the condition that the projective variety (\ref{cond1_eqn}) is  smooth/nonsingular of codimension $3$ over $K$ (in which case it is a complete intersection).
In general, Condition~\ref{cond1} is  equivalent to the condition that the projective scheme (\ref{cond1_eqn}) is  smooth of codimension $3$ over $K$.

Consider all homogeneous degree $k$ forms $G$ in the variables $x_1,x_2,x_3$ over the field $K$. 
There is a discriminant $\Disc$ of such forms, which is an irreducible polynomial in the coefficients of the form $G$, with integer coefficients, such that
  $\Disc(G)=0$ if and only if
there is a non-zero point $\u{a} \in K^3$, such that 
\[\frac{\partial G}{\partial x_1}(\u{a})=\frac{\partial G}{\partial
 x_2}(\u{a})= \frac{\partial G}{\partial x_3}(\u{a})=0.\]
We give the construction of the discriminant in the appendix in Section \ref{sec_appendix}.  Though the basic construction over a field is classical, we take care to see that there is a single discriminant we can use for all our fields at once.

We may now specify our second condition on a system $Q_1,Q_2,Q_3$ of quadratic forms.
\begin{condition}\label{cond2}
The discriminant of the determinant form
\beq\label{det_F_dfn}
 F_{\u{Q}}(x_1,x_2,x_3) = \det (x_1Q_1+x_2Q_2 + x_3Q_3)
 \eeq
is non-zero in $K$.
\end{condition}
This is equivalent to the condition that there does not exist a non-zero point $\u{a} \in K^3$ 
such that $\partial F_{\u{Q}}/\partial x_1(\u{a})=\partial F_{\u{Q}}/\partial x_2(\u{a})=\partial F_{\u{Q}}/\partial x_3(\u{a})=0$. If the ideal $(F_{\u{Q}}(\u{x}))$ is reduced  it is 
 equivalent to the condition that the projective variety given by $F_{\u{Q}}(\u{x})$ is nonsingular of codimension $1$.
 We will think of this as the statement that the locus of singular elements in the linear system generated by the quadratic forms $Q_1,Q_2,Q_3$ is nonsingular. 
Throughout the remainder of the paper, when we assume that $\u{Q}$ satisfies Condition \ref{cond2}, we refer to the condition over $\overline{\Q}$; of course, we note that visibly Condition \ref{cond2} holds over $K$ of characteristic zero if and only if it holds over $\overline{\Q}$.

We now state the third and final condition.
\begin{condition}\label{cond3}
For any $\lam_1,\lam_2,\lam_3 \in K$, not all zero, 
\beq\label{rank_drop0}
\rank\left(\lam_1 Q_1 + \lam_2 Q_2 + \lam_3 Q_3\right)\geq k-1.
\eeq
\end{condition}

We now state several key results about these conditions, deferring proofs to Section \ref{sec_equiv_conditions}. Our first result is as follows:
\begin{proposition}\label{prop1}
Let $Q_1, Q_2, Q_3$ be integral quadratic forms in $k\geq 4$ variables, and let $K$ be a fixed algebraically closed field of  characteristic zero or odd characteristic. Then Condition \ref{cond2} with respect to $K$ holds if and only if both Condition \ref{cond1} and Condition \ref{cond3} hold with respect to $K$. 
\end{proposition}

In particular, we see that once we assume a system $\u{Q}$ satisfies Condition \ref{cond2}, then any element in the linear system spanned by $\bQ$ has its rank drop from full rank $k$ by at most one. This is in analogy to the situation of two forms seen in \cite{HBPierce}. The assumption of a smoothness condition like Condition \ref{cond2} is crucial; as remarked in the introduction, Condition \ref{cond3} can be violated by a system of three diagonal forms, since Condition~\ref{cond1} can hold without Condition~\ref{cond2} holding.

Importantly, Condition \ref{cond2} holds for a generic choice of three quadratic forms:

\begin{proposition}\label{prop4}
Let $K$ be a fixed algebraically closed field of  characteristic zero or odd characteristic. Then Condition \ref{cond2} holds with respect to $K$ for a generic choice of integral quadratic forms $Q_1,Q_2,Q_3$.
\end{proposition}

Next, if we let $\Abf$ denote the collection of all coefficients of $Q_1,Q_2,Q_3$, we note that there is a polynomial in $\Abf$ and $\bn$ which identifies those forms $Q_1,Q_2,Q_3$ and tuples $\bn$ such that the affine system of equations $\bQ(\bfx)=\bn$ satisfies the Jacobian criterion for smoothness.

\begin{proposition}\label{prop3.4}
There exists a homogeneous form $H(\Abf,\bn)$ in $\Abf$ and $\bn$ with integral coefficients of total degree $2(k+1)k(k-1)$ such that the following holds  
for any  algebraically closed field  $K$ of  characteristic zero or odd characteristic.
When the system of quadratic forms given by $\Abf$ satisfies Condition \ref{cond2} over $K$,
then we have $H(\Abf,\u{n})=0$ in $K$ if and only if the system of forms with coefficients given by $\Abf$ is such that there is a point $\bfx\in K^k$ with 
\beq\label{Qn0}
\u{Q}(\bfx)= \u{n}
\eeq
and 
\[ \rank J_{\u{Q}}(\bfx)<3.\]
Moreover, for each $\Abf$ giving a system of forms satisfying Condition \ref{cond2} over $K$, we have that
$H(\Abf,\bn)$, as a polynomial in $\u{n}$ with coefficients in $K$, is homogeneous, not identically zero, and degree at least $1$.
\end{proposition}

Note that once we have fixed a choice for the coefficients $\Abf$ and thus determined a system $Q_1,Q_2,Q_3$, we will denote the corresponding polynomial by $H_{\u{Q}}(\u{n}) = H(\Abf,\u{n})$.

\subsection{Definition of good and bad primes}
Throughout our work, we will assume that we have chosen a system of forms $\u{Q}$ defined over $\Q$ that satisfy Condition \ref{cond2} over $\overline{\Q}$ (and hence over any algebraically closed field of characteristic zero). In our estimation of certain exponential sums we will also require a dichotomy of the primes into those  for which $\u{Q}$ satisfies Condition \ref{cond2} over $\overline{\F}_p$ and those for which it fails; this  motivates the definition of good and bad primes as follows.

\begin{dfn}
We say a prime $p$ is bad if $p|2\Disc(F_{\u{Q}})$. Otherwise we say that $p$ is good.
\end{dfn}
Under the assumption that $\u{Q}$ satisfies Condition \ref{cond2} over $\overline{\Q}$, we see that there are finitely many bad primes; this will be crucial throughout our work.

We record for later use two observations which are immediate consequences of this definition and an application of Proposition \ref{prop1} with respect to $\overline{\F}_p$.
\begin{lemma}
If $p$ is good, the projective variety over $\overline{\F}_p$ given by $\cap_{i=1}^3\{Q_i=0\}$ is nonsingular, that is, $\u{Q}$ satisfies Condition \ref{cond1} over $\overline{\F}_p$.
\end{lemma}

\begin{lemma}\label{geolem1}
Let $p$ be a good prime and $\bb\in\Z^3$ such that $(\bb,p)=1$. Then the rank of the matrix $\bb\cdot\bQ$ over $\overline{\F}_p$ is at least $k-1$.
\end{lemma}

Furthermore, we need to make a further distinction among good primes, depending on whether the reduction modulo $p$ of the affine variety $\bQ(\bfx)=\bn$ is smooth or not.

\begin{dfn}
For a fixed tuple $\bn$ we say that a good prime $p$ is Type I with respect to $\u{n}$ if  $p \ndiv H_{\u{Q}}(\bn)$; otherwise if $p$ is a good prime such that $p | H_{\u{Q}}(\u{n})$ it is said to be Type II with respect to $\u{n}$.
\end{dfn}

We remark that the notion of Type I and Type II depends on $\bn$, in contrast to the property of a prime being good or bad. 
Moreover, it is important to note that for all $\u{n}$ such that $H_{\u{Q}}(\u{n}) \neq 0$, there are finitely many Type II primes.

\subsection{Diagonalization}
While it is not typically possible to  diagonalize three quadratic forms simultaneously, the following result ensures that it is possible to simultaneously diagonalize two of the three forms, if their intersection is smooth and of codimension two.

We first recall the following observation  by Browning, Dietmann and Heath-Brown, an immediate consequence of Lemma 3.1 of  \cite{BroDieHBA13}:
\begin{lemma}\label{geolem0}
Assume that $Q_1, Q_2,Q_3$ satisfy Condition \ref{cond1} over some algebraically closed field $K$ of characteristic zero. Then there exist linearly independent vectors $\u{\lam}_1, \u{\lam}_2 \in K^3$ such that upon defining 
\begin{eqnarray*}
Q_1'&=& \u{\lam}_1 \cdot \u{Q} \\
Q_2' &=& \u{\lam}_2 \cdot \u{Q},
\end{eqnarray*}
we have that $Q'_1, Q'_2$ satisfy the Jacobian criterion for smoothness, i.e. \\
$\rank \left( \partial Q'_i/\partial x_j (\xbf) \right)_{i=1,2; 1\leq j\leq k}=2$ at any non-zero point $\xbf\in K^k$ such that $Q'_1(\xbf)=Q'_2(\xbf)=0$.
\end{lemma}

Note the Jacobian criterion above is equivalent to 
the intersection of $Q'_1=0$ and $Q'_2=0$ in $\P^{k-1}$ being a smooth scheme of codimension 2
(and in particular a complete intersection).
In our situation, we assume that $Q_1,Q_2, Q_3$ satisfy Condition \ref{cond2} and hence Condition \ref{cond1}, so that this lemma shows we may always choose the representative forms so that the intersection of two of them is smooth. 

We next recall a consequence of Proposition 2.1 in \cite{HBPierce}:
\begin{lemma}\label{geoprop7}
Let $Q_1$ and $Q_2$ be two quadratic forms in $k$ variables, and let  $K$ be an algebraically closed field of characteristic zero or of odd characteristic. Then $Q_1, Q_2$ satisfy the Jacobian criterion for smoothness, i.e. $\rank \left( \partial Q_i/\partial x_j (\xbf) \right)_{i=1,2; 1\leq j\leq k}=2$ at any non-zero point $\xbf\in K^k$ such that $Q_1(\xbf)=Q_2(\xbf)=0$, 
 if and only if $Q_1$ and $Q_2$ are simultaneously diagonalizable over $K$.
\end{lemma} 

This strategy will enable us, in Lemma \ref{HilfslemmaS}, to simultaneously diagonalize two of the three quadratic forms.

\subsection{Proof of geometric facts}\label{sec_equiv_conditions}
In this section we turn to the proof of the key geometric propositions we have stated.

\begin{proof}[Proof of Proposition~\ref{prop1}]
We will prove the equivalence of Condition \ref{cond2} with the simultaneous satisfaction of both Conditions \ref{cond1} and \ref{cond3} by contradiction. Throughout the proof we use the following notation: write $Q_i=(\bfb_1^{(i)},\ldots, \bfb_k^{(i)})$ with column vectors $\bfb_l^{(i)}$ and write $\ba\cdot\bQ= (\bfb_1,\ldots, \bfb_k)$ with column vectors $\bbf_l$, where we omit the dependence on the vector $\ba\in K^3$.
 Then note that we have
\begin{equation*}
\frac{\partial F_{\u{Q}}}{\partial a_i} (\ba)= \sum_{l=1}^k \det (\bfb_1,\ldots, \bfb_l^{(i)},\ldots, \bfb_k).
\end{equation*}
\par
We first assume that there is some $\ba\in K^3\setminus\{\u{0}\}$ with $\rank (\ba\cdot\bQ)\leq k-2$. Then one certainly has $F_{\u{Q}}(\ba)=0$ and the fomula for the derivatives of $F_{\u{Q}}(\ba)$ shows that $\frac{\partial F_{\u{Q}}}{\partial a_i}(\ba)=0$ for all $1\leq i\leq 3$. This is a contradiction to Condition \ref{cond2}, since $\ba$ would be a singular point on $F_{\u{Q}}(\ba)=0$.

Next assume that we are given some $\bfx\in K^k\setminus\{\boldsymbol{0}\}$ with the property that $Q_i(\bfx)=0$ for $1\leq i\leq 3$ and $\rank J_{\u{Q}}(\xbf)< 3$, so that there exists $\ba\in K^3\setminus\{\u{0}\}$ with the property that $(\ba\cdot\bQ)\bfx=\boldsymbol{0}$. Since $\ba\cdot\bQ$ has non-trivial kernel, we see that $F_{\u{Q}}(\u{a})=0$. Also, after a linear change of variables we may assume that $\bfx=\bfe_1$, the first unit vector. Then we see that $\bfb_1=\boldsymbol{0}$ and so for each $1 \leq i \leq 3$,
\begin{equation*}
\frac{\partial F_{\u{Q}}}{\partial a_i}(\ba)= \det (\bfb_1^{(i)},\bfb_2,\ldots, \bfb_k).
\end{equation*}
Note that $Q_i(\bfe_1)=0$ implies that $b_{1,1}^{(i)}=0$; also, by symmetry of the matrix $\ba\cdot\bQ$, the first row of this matrix is identically zero. Putting these facts together, we see that $\frac{\partial F_{\u{Q}}}{\partial a_i}(\ba)=0$ for $1 \leq i \leq 3$; together with the previous observation that $F_{\u{Q}}(\ba)=0$, this provides a contradiction to Condition \ref{cond2}.

Finally assume that Condition \ref{cond2} fails, so that there exists some $\ba\in K^3\setminus\{\u{0}\}$ with $F_{\u{Q}}(\ba)=0$ and $\frac{\partial F_{\u{Q}}}{\partial a_i}(\ba)=0$ for $1 \leq i \leq 3$. Either this implies that $\rank (\ba\cdot\bQ)\leq k-2$ or we have $\rank (\ba\cdot\bQ)=k-1$. In the first case we have already arrived at a contradiction to Condition \ref{cond3}, and hence we continue under the assumption that $\rank (\ba\cdot\bQ)=k-1$. For notational convenience assume that $\bfb_1,\ldots, \bfb_{k-1}$ are linearly independent and that we have a relation of the form
\begin{equation*}
\bfb_k= \sum_{l=1}^{k-1} c_l \bfb_l,
\end{equation*}
for some coefficients $c_l\in  K$. Let $\bfx_0= (\bfc,-1) \neq \boldsymbol{0}$ and observe that 
\beq\label{aQx0}
(\ba\cdot\bQ)\bfx_0=\boldsymbol{0},
\eeq
so that $\rank J_{\u{Q}}(\xbf_0) < 3$.
 We claim that in addition $Q_i(\bfx_0)=0$ for all $1\leq i\leq 3$. For this we fix $1 \leq i \leq 3$ and rewrite the condition $\frac{\partial F_{\u{Q}}}{\partial a_i}(\ba)=0$ as
\begin{equation*}
\begin{split}
0&= \det (\bfb_1,\ldots, \bfb_{k-1},\bfb_k^{(i)})+\sum_{l=1}^{k-1} \det (\bfb_1,\ldots, \bfb_l^{(i)},\ldots, \bbf_{k-1},\sum_{i=1}^{k-1}c_i\bfb_i)\\
&= \det (\bfb_1,\ldots, \bfb_{k-1},\bfb_k^{(i)}) +\sum_{l=1}^{k-1} \det (\bfb_1,\ldots, c_l \bfb_l^{(i)},\ldots, \bfb_{k-1},\bfb_l)\\
&= \det (\bfb_1,\ldots, \bfb_{k-1},\bfb_k^{(i)})- \sum_{l=1}^{k-1} \det (\bfb_1,\ldots, \bfb_{l},\ldots, \bfb_{k-1},c_l\bfb_l^{(i)})\\
&= \det (\bfb_1,\ldots, \bfb_{k-1}, \bfb_k^{(i)}- \sum_{l=1}^{k-1} c_l\bfb_l^{(i)}) \\
& =  \det (\bfb_1,\ldots, \bfb_{k-1}, -Q_i \bfx_0).
\end{split}
\end{equation*}
We conclude that the vector $Q_i \bfx_0$ is contained in the span of $\bfb_1,\ldots, \bfb_{k-1},$ and since $\bfx_0$ is orthogonal to $\bfb_l$ for all $1\leq l\leq k-1$ we see that $\bfx_0^t Q_i\bfx_0=0$. This holds for each $1\leq i\leq 3$, and thus provides a contradiction to Condition \ref{cond1}, and finally completes the proof of the proposition.
\end{proof}

\begin{proof}[Proof of Proposition~\ref{prop4}]
Clearly for a generic choice of $Q_1,Q_2,Q_3$, the forms will have linear span of rank $3$ in the space of quadratic forms.
In the $\P^{(k+1)k/2-1}$  of quadrics in $\P^{k-1}$, let $\Phi_i$ be the set of rank $\leq i.$
We have that the singular locus $(\Phi_{k-1})_{\operatorname{sing}}\subset \Phi_{k-2}$, and that $\codim \Phi_{k-2}=3$ \cite[Example 22.31]{Harris}.  So by Bertini's theorem \cite[Corollary 11]{Kleiman},
a generic two-dimensional plane in the $\P^{(k+1)k/2-1}$ of quadrics does not intersect  $(\Phi_{k-1})_{\operatorname{sing}}$. For such a two-dimensional plane $P$, we have $P\cap  \Phi_{k-1}= P\cap (\Phi_{k-1} \setminus (\Phi_{k-1})_{\operatorname{sing}})$.  Since $(\Phi_{k-1} \setminus (\Phi_{k-1})_{\operatorname{sing}})$ is smooth, again by Bertini's theorem  a generic two-dimensional plane  $P$ in the $\P^{(k+1)k/2-1}$ of quadrics
has $P\cap (\Phi_{k-1} \setminus (\Phi_{k-1})_{\operatorname{sing}})$ smooth of dimension $1$, and thus $P\cap  \Phi_{k-1}$ smooth.  If the plane $P$ is spanned by quadrics $Q_1, Q_2, Q_3$, then $F_{\u{Q}}(\u{x})=0$ describes exactly the scheme-theoretic intersection $P\cap  \Phi_{k-1}$, proving the proposition.
\end{proof}

\begin{proof}[Proof of Proposition~\ref{prop3.4}: Construction of the polynomial $H$]
First, let $\Delta$ be the discriminant constructed in the appendix in Section \ref{sec_appendix}
 that detects when $3$ quadratic forms in $k+1$ variables fail the Jacobian criterion (i.e. 
 determine a projective scheme that is not smooth of codimension $3$ in $\P^{k}$). 
  We have that $\Delta$ is homogeneous of degree $2(k+1)k(k-1)$ in the coefficients of the forms \cite[Th\'eor\`eme 1.3]{Benoist}.  
By setting all the coefficients of $x_ix_{k+1}$ for $i\ne k+1$ equal to $0$, we arrive at a polynomial
$H(\Abf,\u{n})$ that detects whether the projective scheme given by
$
\u{Q}(\bfx)= \u{n} x_{k+1}^2
$
is not smooth of codimension $3$.  When the system of quadratic forms given by $\Abf$ satisfies Condition \ref{cond1}, there are no obstructions to being smooth of codimension $3$ at infinity (when $x_{k+1}=0$).  Thus in this case, over an algebraically closed field $K$ not of characteristic $2$,  $H(\Abf,\u{n})=0$ if and only if there is a point $\bfx\in K^k$ with 
$\u{Q}(\bfx)= \u{n}$ and $\rank J_{\u{Q}}(\bfx)<3.$  

Now we will show that given $\Abf$, we have that $H(\Abf,\u{n})$ is non-zero as a polynomial in $\u{n}$.  For forms $\u{Q}$ satisfying Condition \ref{cond2}, there is a one-dimensional (projective) locus, call it $\mathcal{C}$, of $[\lambda_1:\lambda_2:\lambda_3]\in \P^2$ such that
$\sum_i \lambda_iQ_i$ is singular.  Let $I\subset \A^k \times \mathcal{C}$ be the incidence locus  of those $(\bfx,\u{\lambda})$ such that $\sum_i \lambda_iQ_i\bfx=0.$  Since for forms $\u{Q}$ satisfying Condition \ref{cond3} we have that $\sum_i \lambda_iQ_i$ is rank at least $k-1$, in the projection $I\rightarrow \mathcal{C}$ each $\u{\lambda}$ has a fiber of (affine) dimension $1$.  Thus $I$ has dimension $2$, and the image of the map $I\rightarrow\A^3$ sending $\bfx \mapsto \u{Q}(\bfx)$ has dimension at most $2$.  Thus for the generic $\u{n}$, there is no point $\bfx\in K^k$ with $\u{Q}(\bfx)=\u{n}$ and $\rank J_{\u{Q}}(\bfx)<3.$

Given $\Abf$, we see that $H(\Abf,\u{n})$ is homogeneous in the $\u{n}$.  We can check this over $\overline{\Q}$, in which case the projective scheme given by $
\u{Q}(\bfx)= \u{n} x_{k+1}^2
$  is isomorphic to that given by $
\u{Q}(\bfx)= \lambda \u{n} (x'_{k+1})^2
$
via the change of coordinates $x'_{k+1}=\lambda^{-1/2}x_{k+1}$.  Thus $H(\Abf,\u{n})$ has zero-locus invariant under scaling the $\u{n}$ by $\lambda$, and is thus homogeneous in the  $\u{n}$.

Suppose for the sake of contradiction that given some $\Abf$, we have that $H(\Abf,\u{n})$ is degree $0$.  That would mean over any field, the projective scheme given by $
{\u{Q}}(\bfx)= \u{n} x_{k+1}^2
$ is always smooth of codimension $3$. 
 Since there must be some $\u{a}$ over an algebraically closed field  such that $\u{a}\cdot \u{Q}$ is 
not full rank, there is some $\bfx$ such that $\u{a}\cdot \u{Q} \bfx=(1/2)[\partial \u{a}\cdot \u{Q}/\partial 
x_i (\bfx)]_{1\leq i\leq k}=0.$ For this $\bfx$, we choose $\u{n}$ such that 
 $
{\u{Q}}(\bfx)= \u{n}, 
$  and we have a singularity, giving a contradiction.

\end{proof}

\section{Preliminary lemmas}
We gather here certain results on counting solutions to polynomial equations modulo primes and prime powers, as well as counting elements in sublevel sets over $\Z$ and $p$-adically.

\subsection{Solutions modulo primes and prime powers}

The following crude estimate on the number of $\F_p$-rational points on a variety is due to Lang and Weil. For the convenience of the reader we state it here again.

\begin{lemma}[Lemma 1 of \cite{LangWeil}]\label{geolem4}
Let $p$ be a prime. Let $V \subset \A_{\F_p}^r$ be an affine variety given by $s$ homogeneous polynomials with integer coefficients of degree not exceeding $\rho$ and assume that $\dim (V)= n$. Then the number of $\F_p$-rational points on $V$ is bounded by $O(p^n)$, where the implied constant depends on $\rho$ and $r$.
\end{lemma}

We will also refer to a consequence of the Deligne bound, as formulated by Hooley \cite{Hoo91}:
\begin{lemma}\label{lemma_Hooley}
If $V$ is a projective complete intersection of dimension $n$ over the finite field $\F_p$, with singular locus of dimension $s$, then the number of $\F_p$-rational points of $V$ is equal to
\[(p^{n+1} -1)/(p-1) + O(p^{(n+s+1)/2}).\]
\end{lemma}

\subsection{A version of Hensel's Lemma}

We now let $r\geq 2$ and take $ f \in \Z[x_1,\ldots, x_r]$ to be a homogeneous polynomial. We write $D_f$ for its discriminant and assume that $D_f\neq 0$ in $\Q$. 
The goal of this section is to give upper bounds for the counting function
\begin{equation*}
N(p^\ell):= \# \{ \bfx \modd{p^\ell}: (\bfx,p)=1,\ f(\bfx) \equiv 0 \modd{p^\ell}\}.
\end{equation*}

The assumption that $D_f \neq 0$ implies in particular that the projective variety $V$ given by $f(\bfx)=0$ is smooth over $\Q_p$ for any prime $p$. As a consequence we will obtain the following lemma.
\begin{lemma}\label{claim1}
Let $p$ be a prime. Then there is a natural number $\alp$ depending on $f$ and $p$ such that the following holds: if  $\bfx\in (\Z/p^\alp\Z)^r$ with $p\nmid \bfx$ such that $f(\bfx)\equiv 0 \modd{p^\alp}$, then $p^\alp \nmid \nabla f(\bfx)$.
\end{lemma}
Before proving this, we note several immediate consequences.
Note that for almost all primes $p,$ namely for $p \ndiv D_f$, the reduction of $V$ modulo $p$ is smooth and then one can take $\alp=1$.
In this particularly nice case, we may apply Lemma \ref{lemma_Hooley} with $n=r-2$, $s=-1$ to conclude:
\begin{lemma}\label{lemma_Np_smooth}
For $p \ndiv D_f$,
\[ N(p) = p^{r-1}  + O(p^{\frac{r}{2}}).\]
\end{lemma}
More generally for prime powers, we have:
\begin{lemma}\label{Hensel}
There exists an absolute constant $C$ such that the following is true: if $p\ndiv D_f$, then for all $\ell \geq 1$,
\begin{equation*}
N(p^\ell)\leq C p^{\ell(r-1)}.
\end{equation*}
If $p|D_f$, then for all $\ell \geq 1$,
\begin{equation*}
N(p^\ell)\leq  \alp p^{2\alp} p^{\ell(r-1)},
\end{equation*}
where $\al$ is as provided by Lemma \ref{claim1}.
\end{lemma}

The following useful lemma is an immediate consequence of Lemma \ref{Hensel}:
\begin{lemma}\label{lemF}
Let $\u{Q}$ be a system of forms that satisfies Condition \ref{cond2} over $\overline{\Q}.$ For any prime $p$ and $u \geq 1$, there exists a constant $c_p$ such that
\begin{equation*}
\#\{\u{x} \modd{p^u}:\ (\u{x},p)=1,\ p^u| F_{\u{Q}}(\u{x})\}\leq c_p p^{2u},
\end{equation*}
where $c_p=c$ may be chosen independent of $p$ for all primes $p\nmid 2 \Disc(F_{\u{Q}})$.
\end{lemma}
Effectively, upon recalling that the bad primes are those that divide $2 \Disc(F_{\u{Q}})$, we see that this lemma provides a bound that is uniform for good primes.
We now turn to the proofs.

\begin{proof}[Proof of Lemma \ref{claim1}] 
We proceed by contradiction: assume that for any natural number $\alp$ there is some $\bfx\in (\Z/p^\alp \Z)^r$ with $p\nmid \bfx$, $f(\bfx)\equiv 0 \mmod p^\alp$ and $p^\alp|\nabla f(\bfx)$. Now consider the following directed system indexed by $\alp\in \N$, which is a subsystem of the sets $(\Z/p^\alp \Z)^r$ with the natural projections as transition maps. At level $\alp$ let $P(\alp)\subset (\Z/p^\alp\Z)^r$ be the set of $\bfx\in (\Z/p^\alp \Z)^r$ with $p\nmid \bfx$, $f(\bfx)\equiv 0 \mmod p^\alp$ and $p^\alp \mid \nabla f(\bfx)$. Let $\iota_{\alp}^\bet$ be the projection map $\iota_\alp^\bet: P(\bet)\rightarrow P(\alp)$. We define a subsystem $Q(\alp)$ of the directed system $P(\alp)$ by taking 
$$Q(\alp) := \{ \bfx\in (\Z/p^\alp \Z)^r: \forall \bet>\alp, \; \exists \bfy \in P(\bet) \mbox{ such that } \iota_\alp^\bet (\bfy)=\bfx\}.$$
The system $Q(\alp)$ is again a directed system with the natural projection maps $\iota_\alp^\bet$ as transition maps. Moreover, since $Q(\alp)= \cap_{\bet>\alp}\iota_{\alp}^\bet (P(\bet))$ and $(\Z/p^\alp\Z)^r$ is a finite set, the chain of sets $\cap_{\alp<\bet<\gam} \iota_{\alp}^\bet (P(\bet))$ must become stationary for $\gam$ increasing. Since each $P(\bet)$ is non-empty by assumption, we conclude that each set $Q(\alp)$ is non-empty, and the maps $\iota_\alp^\bet: Q(\bet)\rightarrow Q(\alp)$ are surjective. Hence the inverse limit $\varprojlim_\alp Q(\alp)$ is non-zero. This means that there is singular point $\bfx\in \Z_{p}^r$ on the projective hypersurface $f=0$, which is a contradiction to the assumption that $D_f\neq 0$.
\end{proof}

\begin{proof}[Proof of Lemma \ref{Hensel}]
We give an elementary proof, which is self-contained; one could also deduce parts of the argument from results in the literature, as for example Corollary 5.23 in \cite{Greenberg}. For notational convenience, we temporarily define the set
\begin{equation*}
\calN (p^\ell) =  \{ \bfx \modd{p^\ell}: p\nmid \bfx \mbox{ and } f(\bfx) \equiv 0 \modd{p^\ell}\},
\end{equation*}
with cardinality $N(p^\ell)= \# \calN (p^\ell)$. We further define
for each $\xbf \modd{p^\al}$ the set
\[ \calN(\bfx,p^\ell) =  \{ \ybf \modd{p^\ell}:  f(\ybf) \equiv 0 \modd{p^\ell} \mbox{ and } \ybf \con \xbf \modd{p^\al}\},
\]
with corresponding cardinality $N(\bfx,p^\ell)= \# \calN(\bfx,p^\ell)$. If $\ell \geq \alp$, then we can partition
\begin{equation}\label{star3}
\calN (p^\ell)= \bigcup_{\substack{ \bfx \modd{p^\al}\\ p\ndiv \bfx}} \calN (\bfx,p^\ell).
\end{equation}
In order to bound $\# \calN(\bfx,p^\ell)$, we fix an $\bfx \modd{p^\alp}$ with $p \ndiv \xbf$ and $f(\bfx)\equiv 0 \modd{p^\alp}$. Supposing $p^\bet \parallel \nabla f(\bfx)$ (note that $\bet <\alp$ by Lemma \ref{claim1}), then for any $\hbf \modd{p^\be}$,
\begin{equation}\label{star1}
f(\bfx + p^\alp \bfh) \equiv f(\bfx) \modd{p^{\alp+\bet}},
\end{equation}
since for any integer $u \geq 0$,
\begin{equation*}
f(\bfx+p^u\bfh)\equiv f(\bfx)+ p^u \nabla f(\bfx)\cdot \bfh \modd{p^{2u}}.
\end{equation*}
Hence the value of $f$ modulo $p^{\alp+\bet}$ only depends on $\bfx$ modulo $p^\alp$. In other words, equation (\ref{star1}) implies that
\begin{align*}
N(\bfx,p^{\alp+\bet})= \left\{ \begin{array}{cl} p^{r\bet} & \mbox{if } f(\bfx)\equiv 0\modd{p^{\alp+\bet}}\\ 0 & \mbox{if } f(\bfx)\not\con 0\modd{p^{\alp+\bet}}.\end{array}\right.
\end{align*}

Next for any $u \geq \al$ we again
consider the congruence
\begin{equation*}
f(\bfx+p^u \bfh) \equiv f(\bfx) + p^u \nabla f(\bfx)\cdot \bfh \modd{p^{u+\bet+1}}, 
\end{equation*}
still under the assumption $p^\be \| \nabla f(\xbf)$; without loss of generality, we may assume
\beq\label{beta_divis}
p^\bet \parallel \frac{\partial f}{\partial x_1}(\bfx).
\eeq
 Assume we are given a congruence class $\bfx$ modulo $p^u$ with $f(\bfx)\equiv 0 \modd{p^{u+\bet}}$; if we want to lift this to a congruence modulo $p^{u+\bet+1}$ we can choose $h_2,\ldots, h_r$ freely modulo $p$ and then by (\ref{beta_divis}) there exists a unique $h_1$ modulo $p$ such that $f(\bfx+p^u \bfh) \equiv 0 \modd{p^{u+\bet+1}}$. This implies that
\begin{equation*}
N(\bfx,p^{u+\bet+1})= p^{r-1} N(\bfx,p^{u+\bet}),
\end{equation*}
for any $u\geq \alp$. Inductively this gives
\begin{equation*}
N(\bfx,p^\ell)= p^{(\ell-\alp-\bet)(r-1)}N(\bfx,p^{\alp+\bet}),
\end{equation*}
if $\ell \geq \alp+\bet$. For $\xbf \modd{p^\al}$ such that $f(\bfx)\equiv 0 \modd{p^{\alp+\bet}}$, this implies that
\begin{equation*}
N(\bfx,p^\ell)= p^{(\ell-\alp-\bet)(r-1)}p^{r\bet}.
\end{equation*}
Together with equation (\ref{star3}) this implies that
\begin{equation}\label{star5}
N (p^\ell)= p^{(\ell-\alp)(r-1)}\sum_{\bet =0}^{\alp -1}p^\bet  M(\al,\be;p) 
\end{equation}
where
\[ M(\al,\be;p)=\#\{\bfx \modd{p^\alp}: p\nmid \bfx,\ p^\bet\parallel \nabla f(\bfx),\ \mbox{ and } f(\bfx)\equiv 0\modd{p^{\alp+\bet}}\}.\]
We now distinguish between primes that do or do not divide the discriminant $D_f$.
 First we consider the case of primes $p \ndiv D_f$, so that we can choose $\alp=1$ and $\bet=0$, obtaining
\begin{align*}
N (p^\ell)&= p^{(\ell-1)(r-1)}\# \{\bfx \modd{p}: p\nmid \bfx \mbox{ and } f(\bfx)\equiv 0\modd{p}\}\\
& = p^{(\ell-1)(r-1)}(N(p)-1)\\
& = p^{\ell(r-1)}\left(1+O(p^{-(r-2)/2})\right),
\end{align*}
where in the last line we applied Lemma \ref{lemma_Np_smooth}.\par
We come to the case of primes $p| D_f$. Using equation (\ref{star5}) we trivially estimate
\begin{align*}
N(p^\ell)\leq p^{(\ell-\alp)(r-1)} \sum_{\bet=0}^{\al-1} p^\bet p^{\alp r} \leq \alp p^{2\alp}p^{\ell(r-1)}.
\end{align*}
\end{proof}

\subsection{Discrete sublevel set estimates}
We recall that our main theorems are controlled by a certain polynomial $H_{\u{Q}}(\u{n})$ provided by Proposition \ref{prop3.4}.  We will apply the sublevel set estimates we derive in this section at several points in our work:  to count the number of exceptional tuples $\u{n}$ such that $H_{\u{Q}}(\u{n})$ vanishes as an integer, is small as a real number, or is divisible by a high power of a prime (see the proof of Theorem \ref{thm_asymp} in Section \ref{sec_final_proofs}, and Lemma \ref{lemma_E123} in the proof of Theorem \ref{thm_exceptions}, also in Section \ref{sec_final_proofs}).

We start with the one-dimensional case and then deduce a general lemma which bounds the number of small values attained by a given homogeneous polynomial at integer points of bounded height. We first recall Lemma 1 of \cite{BHBS}, which we state in the following form:
\begin{lemma}\label{lemPoly1}
Let $P(x)\in \Q[x]$ be a polynomial of degree $d$ with  leading coefficient $a_0$ and $M,N\geq 1$. Then for any $M \geq 1$,
\begin{equation*}
\#\{ x\in \Z: |P(x)|\leq M\} \ll M^{1/d},
\end{equation*}
with an implied constant only depending on $a_0$ and the degree $d$ of the polynomial $P$. 
\end{lemma}


From this we may immediately deduce the following general lemma:
\begin{lemma}\label{lemPoly2}
Let $P(x_1,\ldots, x_s)$ be a non-zero homogeneous polynomial with rational coefficients, of total degree $d$. Then for any real $M,N\geq 1$ we have the bound
\begin{equation*}
\# \{\bfx\in \Z^s: \max_i |x_i|\leq N,\ |P(\bfx)|\leq M\} \ll N^{s-1}M^{1/d}.
\end{equation*}
with an implied constant only depending on the coefficients and the total degree of the polynomial $P$.
\end{lemma}

\begin{proof}
We may assume that $N$ is sufficiently large with respect to $d$. First we claim that there is some integer tuple $\bfa\in \Z^s$ such that $P(\bfa)\neq 0$ and $a_1=1$.
Indeed, there are $\gg N^{s-1}$ choices of $a_2, \ldots, a_s$ with $|a_i| \leq N$, while (since $P$ is a nonzero polynomial) there are $\ll N^{s-2}$ choices of $a_2, \ldots, a_s$ with $|a_i | \leq N$ such that $P(1,a_2,\ldots, a_s) =0$. Thus there must be at least one  integral tuple $(a_2,\ldots, a_s)$ that yields $P(1,a_2,\ldots, a_s)\neq 0$.


We now fix a tuple $(a_2,\ldots,a_s)$ such that $P(1,a_2,\ldots, a_s) \neq 0$ and let 
\begin{equation*}
\tilde{P}(x_1,\ldots, x_s):= P(x_1,x_2+a_2x_1,\ldots, x_s+a_s x_1).
\end{equation*}
Then we see that 
\begin{equation*}
\tilde{P}(x_1,\ldots, x_s)= P(1,a_2,\ldots, a_s) x_1^d +Q(x_1,\ldots, x_s),
\end{equation*}
with some homogeneous form $Q(x_1,\ldots, x_s)$ of degree less than $d$ in $x_1$. Consider also the coordinate transformation $x_1'=x_1$ and $x_i'= x_i-a_ix_1$ for $2\leq i\leq s$. Then $\tilde{P}(\bfx')= P(\bfx)$, so that there is some constant $C_2$ depending only on the $a_i$ for $2\leq i\leq s$ such that
\begin{multline*}
\# \{\bfx\in \Z^s: \max_i |x_i|\leq N,\ |P(\bfx)|\leq M\} \\
 \leq \#\{ \bfx'\in \Z^s: \max_i|x_i'|\leq C_2N,\ |\tilde{P}(\bfx')|\leq M\}.
\end{multline*}
 Hence we have reduced the lemma to considering the polynomial $\tilde{P}(\bfx),$ which has the property that the term $x_1^d$ appears with non-vanishing coefficient.\par
Now we fix some integer choices of $x_2,\ldots, x_s$  of absolute value at most $C_2N$. We apply Lemma \ref{lemPoly1} to the resulting polynomial $\tilde{P}(x_1,x_2,\ldots, x_s)$ in the variable $x_1$ and obtain
\begin{equation*}
\#\{x_1\in\Z: |\tilde{P}(x_1,x_2,\ldots, x_s)|\leq M\}\ll M^{1/d},
\end{equation*}
with an implied constant only depending on the degree of $\tilde{P}$ and the chosen integers $a_2,\ldots, a_s$. Summing trivially over all possible choices for $x_2,\ldots, x_s$ proves the lemma.
\end{proof}

\subsection{$p$-adic sublevel set estimates}
We also require $p$-adic versions of Lemmas \ref{lemPoly1} and \ref{lemPoly2}.

\begin{lemma}\label{lemPoly3}
Let $P(x)\in \Z[x]$ be a polynomial of degree $d$ with leading coefficient $a_0$. Then 
\begin{equation*}
\#\{1\leq x\leq p^f: p^f|P(x)\}\ll p^{f-f/d},
\end{equation*}
with an implied constant only depending on the leading coefficient $a_0$ and the degree $d$ of the polynomial $P$.
\end{lemma}

\begin{proof}
If $d=1$, the claim is certainly true.
For $d\geq 2$, we recall that the content of a polynomial is the greatest common divisor of its coefficients. If the content of $P$ is relatively prime to $p$ then by Corollary 2 and equation (44) of Stewart \cite{Ste91}, one may estimate the cardinality of the set by $\leq dp^{f-f/d}$. One may reduce to this case in the following way. Let $p^\gamma$ be the highest power of the prime $p$ dividing the content of $P(x)$, and write $P(x)=p^\gamma \hat{P}(x)$. For $f\geq \gamma$ we then have
$$ \#\{1\leq x\leq p^f: p^f|P(x)\} = p^\gamma \#\{ 1\leq x\leq p^{f-\gamma}: p^{f-\gamma}|\hat{P}(x)\}.$$
Now observe that the content of $\hat{P}(x)$ is coprime to $p$ and that $p^\gamma | a_0$.

\end{proof}

\begin{lemma}\label{lemPoly4}
Assume that $P(x_1,\ldots, x_s)$ is a non-zero homogeneous polynomial with integer coefficients and total degree $d$. Then for any prime $p$ and integer $f \geq 0$ we have the bound
\begin{equation*}
\#\{ 1\leq x_i\leq p^f: p^f|P(\bfx)\}\ll p^{sf-f/d},
\end{equation*}
with an implied constant only depending on the polynomial $P(\bfx)$.
\end{lemma}
 

\begin{proof}
We recall the construction of $\tilde{P}(\bfx)$ from the proof of Lemma \ref{lemPoly2}, so that
\begin{equation*}
\begin{split}
\tilde{P}(x_1,\ldots, x_s)&= P(x_1,x_2+a_2x_1,\ldots, x_s+a_sx_1)\\
&= P(1,a_2,\ldots, a_s)x_1^{d}+Q(x_1,\ldots, x_s),
\end{split}
\end{equation*}
with $P(1,a_2,\ldots, a_s)\neq 0$ and $Q(\bfx)$ a homogeneous form of degree less than $d$ in $x_1$. The coordinate transformation $x_1'=x_1$ and $x_i'=x_i-a_ix_1$ for $2\leq i\leq s$ takes a complete set of residues modulo $p^f$ again to a complete set of residues modulo $p^f$. Hence
\begin{equation*}
\#\{ 1\leq x_i\leq p^f: p^f|P(\xbf)\}= \#\{ 1\leq x_i'\leq p^f: p^f| \tilde{P}(\bfx')\}.
\end{equation*}
Now we argue in the very same way as in the proof of Lemma \ref{lemPoly2}. We fix $1\leq x_i'\leq p^f$ for $2\leq i\leq s$ arbitrary and then apply Lemma \ref{lemPoly3} to the polynomial $\tilde{P}(x_1',x_2',\ldots, x_s')$ in $x_1'$. This shows that
\begin{equation*}
\begin{split}
\#\{ &1\leq x_i'\leq p^f: p^f| \tilde{P}(\bfx')\}\\ &\ll \sum_{2\leq i\leq s}\sum_{1\leq x_i'\leq p^f} \#\{ 1\leq x_1'\leq p^f: p^f| \tilde{P}(x_1',x_2',\ldots, x_s')\}\\ &\ll p^{(s-1)f+f-f/d}\\ &\ll p^{sf-f/d},
\end{split}
\end{equation*}
and hence completes the proof of the lemma.
\end{proof}

\section{Oscillatory integrals}
On both the major and the minor arcs we will require an upper bound for oscillatory integrals of the following generic form: for any quadratic form $\calQ$ in $n$ variables, define
\begin{equation*}
I(\calQ;\bflam)=\int_{\R^n} e\left(\calQ (\bfu)-\bflam\cdot\bfu\right)w(\bfu)\d\bfu,
\end{equation*}
where $\bflam \in \R^n$ and $w(\bfu)$ is a smooth weight function supported on $[-1,1]^n$, with uniformly bounded derivatives of all orders.

For later reference, we quote Lemma 3.1 of \cite{HBPierce}, a consequence of integration by parts and the second derivative test.

\begin{lemma}\label{lem15.0}
Let $\calQ$ be a quadratic form in $n$ variables with eigenvalues $\rho_1,\ldots, \rho_n$. If $|\bflam|\geq 4 \Vert \calQ\Vert$ then for all $M \geq 1$,
\begin{equation*}
I(\calQ;\bflam)\ll_{M,w}|\bflam|^{-M}.
\end{equation*}
 Moreover one has the upper bound
\begin{equation*}
|I(\calQ;\bflam)|\ll_w \prod_{i=1}^n \min \left(1,\frac{1}{|\rho_i|^{1/2}}\right).
\end{equation*}
\end{lemma}

We will apply this in our specific setting to deduce the following two lemmas, which are analogous to Lemma 3.3 and Lemma 3.4 in \cite{HBPierce}; here we recall the definition of the notation $\{ \u{\phi}\}$ from Section \ref{sec_notation}.

\begin{lemma}\label{lem15.1}
Let $\phi^*=\max_{1\leq i\leq 3}\{\phi_i\}$ and assume that the system $\bQ$ satisfies Condition \ref{cond2}. Then 
\begin{equation*}
\int_{\{\u{\phi}\}}|I(\bnu\cdot\bQ;\bflam)|\d\bnu \ll \min \{\prod_{i=1}^3 \phi_i, (\phi^*)^{3-k/2}\}.
\end{equation*}
\end{lemma}

\begin{lemma}\label{lem15.2}
Let $\phi^*=\max_{1\leq i\leq 3}\{\phi_i\}$ and set 
\[ \Fcal_i(\bfx_1,\bfx_2)=Q_i(\bfx_1)-Q_i(\bfx_2) \qquad \text{ for $1\leq i\leq 3$},\]
 where the system $\u{Q}$ satisfies Condition \ref{cond2}. Then 
\begin{equation*}
\int_{\{\u{\phi}\}}|I(\bnu\cdot\u{\Fcal};\bflam)|\d\bnu\ll \min\{ (\phi^*)^3, (\phi^*)^{3-k}(1+|\log \phi^*|)\}.
\end{equation*}

\end{lemma}

Each of these lemmas relies on the following observation:  the rank drop condition in Condition \ref{cond3} (which is implied by Condition \ref{cond2}) provides a lower bound for all but one of the eigenvalues of any particular linear combination of the forms $Q_1,Q_2,Q_3$. We record this formally:
\begin{lemma}\label{lem15}
Let $Q_1,Q_2, Q_3$ be quadratic forms such that for all $\u{\nu} \in \R^3$ with $\u{\nu} \neq \u{0},$
\begin{equation*}
\rank (\u{\nu}\cdot \u{Q})\geq k-1.
\end{equation*}
 Let $\nu^*= \max_i \{|\nu_i|\}$ and let $\rho_1,\ldots, \rho_k$ be the eigenvalues associated to $\u{\nu} \cdot \u{Q}$, ordered in a way such that $|\rho_1|\leq \ldots \leq |\rho_k|$. Then one has
\begin{equation*}
|\rho_{2}|\gg \nu^*,\quad \mbox{ and } \quad |\rho_k|\ll \nu^*.
\end{equation*}
\end{lemma}
Since the proof of this lemma is identical to the proof of Lemma 2.4 in \cite{HBPierce}, we omit it here and turn immediately to proving Lemmas \ref{lem15.1} and \ref{lem15.2}, for which we must proceed differently from \cite{HBPierce}, since we may no longer necessarily factor the form $F_{\u{Q}}$ defined in (\ref{det_F_dfn}) into linear factors.

\begin{proof}[Proof of Lemma \ref{lem15.1}]
We note that the conclusion of the lemma is trivial for $\phi^*\leq 1$, and assume from now on that $\phi^* \geq 1$.
Applying Lemma \ref{lem15.0} to the form $\u{\nu} \cdot \u{Q}$, followed by Lemma \ref{lem15}, we have 
\begin{equation*}
|I(\bnu\cdot\bQ;\bflam)|\ll \min \left\{1,\left(\frac{1}{\nu^*}\right)^{\frac{k-1}{2}}\right\} \min \left\{1,\frac{1}{|\rho_1|^{1/2}}\right\}.
\end{equation*}
Recall the definition of $F_{\u{Q}}$ in (\ref{det_F_dfn}) and note that
\begin{equation*}
F_{\u{Q}}(\bnu)= \det\left(\sum_{i=1}^3 \nu_iQ_i\right)= \rho_1\cdots \rho_k.
\end{equation*}
Hence, again by Lemma \ref{lem15} we have
\begin{equation*}
|\rho_1|\gg \frac{|F_{\u{Q}}(\bnu)|}{(\nu^*)^{k-1}}.
\end{equation*}
We thus obtain
\begin{equation}\label{eqn15.5}
\int_{\{\bphi\}}|I((\bnu\cdot\bQ);\bflam)|\d\bnu \ll (\phi^*)^{-\frac{k-1}{2}}\int_{\{\bphi\}} \min\left\{1,\left(\frac{(\phi^*)^{k-1}}{|F_{\u{Q}}(\bnu)|}\right)^{1/2}\right\} \d\bnu.
\end{equation}
Next we recall that by our assumption of Condition \ref{cond2}, we have $\Disc (F_{\u{Q}}(\bnu))\neq 0$, so that $F_{\u{Q}}(\bnu)$ is nonsingular outside of the origin, and so
\begin{equation*}
\nabla F_{\u{Q}}(\bnu)\neq 0,\quad \mbox{ for } \bnu\neq 0.
\end{equation*}
Furthermore, the set $|\bnu|_\infty =1$ is compact and hence there is some constant $c$ with
\begin{equation*}
\min_{|\bnu|_\infty=1}|\nabla F_{\u{Q}}(\bnu)| \geq c >0.
\end{equation*}
We can now partition affine $3$-space into three measurable sets $\Ome_i$ (in fact closed cones), such that  for each $1\leq i\leq 3$,
\begin{equation*}
\min_{\bstack{|\bnu|_\infty =1}{\bnu\in \Ome_i}}\left|\frac{\partial F_{\u{Q}}}{\partial \nu_i}(\u{\nu})\right|\geq c',
\end{equation*}
for some positive constant $c'$. Hence, on $\Ome_i$ we may use the homogeneity of the polynomial $F_{\u{Q}}(\bnu)$ and its derivatives to rescale to $|\u{\nu}|_\infty=\nu^*$, obtaining the lower bound
\begin{equation}\label{pF1}
\left|\frac{\partial F_{\u{Q}}}{\partial \nu_i}(\u{\nu})\right|\geq c' (\nu^*)^{k-1}.
\end{equation}

We now consider the contribution to (\ref{eqn15.5}) of $\{\bphi\}\cap \Ome_i$ for each $1\leq i\leq 3$; for simplicity of notation, we now consider $i=1$. For fixed $\nu_2$ and $\nu_3$ the function $F_{\u{Q}}(\bnu)$ is a polynomial of degree at most $k$ in $\nu_1$ and hence we can cover the set of $\nu_1\in \{\phi_1\}$, $\bnu\in \Ome_1$ with at most $k$ intervals on which $F_{\u{Q}}(\u{\nu})$ is monotone as a function of $\nu_1$ (recalling of course that $\nu_2,\nu_3$ are fixed) and $\frac{\partial}{\partial \nu_1} F(\bnu)$ is bounded below as in (\ref{pF1}). Restricting our attention to one of these intervals, say $I$, a variable transformation $(u_1,u_2,u_3) = (F_{\u{Q}}(\u{\nu}), \nu_2, \nu_3)$ on this interval leads to
\begin{align*}
\int_{I \intersect \{\bphi\}\cap \Ome_1} &\min\left\{1,\left(\frac{(\phi^*)^{k-1}}{|F_{\u{Q}}(\bnu)|}\right)^{1/2}\right\} \d\bnu \\ \ll &\int_{\{\phi_3\}}\int_{\{\phi_2\}}\int_{0}^{C(\phi^*)^k} \left|\frac{\partial F_{\u{Q}}}{\partial \nu_1}(\bnu)\right|^{-1} \min \left\{ 1,\left(\frac{(\phi^*)^{k-1}}{|u_1|}\right)^{1/2}\right\} \d u_1 \d u_2 \d u_3,
\end{align*}
where $C$ is some constant and the region of integration is  implicitly further restricted to $\Ome_1$. We can now estimate this, using (\ref{pF1}), by
\begin{align*}
\ll &(\phi^*)^{-(k-1)}(\phi^*)^2 (\phi^*)^{k-1} +\int_{\{\phi_3\}}\int_{\{\phi_2\}} \int_{(\phi^*)^{k-1}}^{C(\phi^*)^k} (\phi^*)^{-(k-1)} \frac{(\phi^*)^{\frac{k-1}{2}}}{|u_1|^{1/2}}\d u_1\d u_2\d u_3 \\
\ll & (\phi^*)^2 +(\phi^*)^{2-\frac{k-1}{2}}\int_{(\phi^*)^{k-1}}^{C(\phi^*)^k} \frac{1}{|u_1|^{1/2}}\d u_1 \\\ll & (\phi^*)^2 +(\phi^*)^{3-1/2}\ll (\phi^*)^{3-1/2}.
\end{align*}
Similarly, the contributions of the other intervals, and the contributions of $\Omega_2$ and $\Omega_3$, are dominated by $(\phi^{*})^{3-1/2}$. In combination with equation (\ref{eqn15.5}), this proves the lemma.
\end{proof}

\begin{proof}[Proof of Lemma \ref{lem15.2}]
The proof of Lemma \ref{lem15.2} is almost identical to the proof of Lemma \ref{lem15.1}. First note that 
\[\sum_{i=1}^3 \nu_i\Fcal_i(\bfx_1,\bfx_2)\]
 has eigenvalues $\pm \rho_1,\ldots, \pm \rho_k$ where $\rho_i$ are the eigenvalues of $\bnu\cdot \bQ$. Hence for $\phi^*\geq 1$ we can estimate
\begin{align*}
\int_{\{\bphi\}} |I(\bnu\cdot\u{\Fcal};\bflam)|\d\bnu &\ll \int_{\{\bphi\}} \prod_{i=1}^k \min \left\{1, \frac{1}{|\rho_i|}\right\}\d\bnu \\ 
& \ll (\phi^*)^{-(k-1)} \int_{\{\bphi\}} \min \left\{1,\frac{(\phi^*)^{k-1}}{|F_{\u{Q}}(\bnu)|}\right\}\d\bnu.
\end{align*}
The same analysis as in the proof of Lemma \ref{lem15.1} leads to the bound
\begin{align*}
&\ll (\phi^*)^{-2(k-1)}(\phi^*)^2 \int_{0}^{(\phi^*)^k}\min\left\{1,\frac{(\phi^*)^{k-1}}{|u_1|}\right\}\d u_1 \\ &\ll (\phi^*)^{-2k+4} \left((\phi^*)^{k-1}+(1+|\log \phi^*|)(\phi^*)^{k-1}\right) \ll (\phi^*)^{3-k}(1+|\log \phi^*|).
\end{align*}
\end{proof}

\subsection{The Singular Integral}
We define the singular integral by
\begin{equation}\label{sing_int_dfn}
J_w(\bmu)=\int_{\R^3}I_w(\btet) e\left(- \btet \cdot  \bmu \right) \d\btet,
\end{equation}
where
\begin{equation}\label{I_dfn}
 I_w(\btet)= \int_{\R^k} e(\btet\cdot\bQ(\bfx))w(\bfx)\d\bfx.
\end{equation}
 For now we merely assume the weight function is smooth and compactly supported in $[-1,1]^k$ with bounded derivatives of all orders.
Next we define for any positive real number $R$ and  $\bmu\in\R^3$ the truncated singular integral
\begin{equation}\label{sing_int_dfn_R}
J_w (\bmu;R)=\int_{[-R,R]^3}  I_w(\u{\theta}) e(-\u{\theta} \cdot \u{\mu}) \d\btet.
\end{equation}
As soon as the number of variables $k$ is large enough, the limit $ \lim_{R \maps \infty} J_w(\bmu;R)$ exists and the singular integral is indeed absolutely convergent.

\begin{proposition}\label{prop_sing_int_conv}
Assume that $k>6$. Then the singular integral $J_w(\bmu)$ is absolutely convergent, and bounded uniformly in $\bmu$. More precisely, we have
\begin{equation*}
|J_w(\bmu)-J_w(\bmu;R)|\ll R^{3-k/2}(\log R)^2.
\end{equation*}

\end{proposition}

\begin{proof}

Lemma \ref{lem15.1} implies that we have the bound
\begin{equation*}
\int_{\{\bphi\}}|I_w(\btet)|\d\btet \ll \min \left\{ \prod_{i=1}^3 \phi_i, (\phi^*)^{3-k/2}\right\}.
\end{equation*}
In order to prove the proposition we need to estimate the integral 
\begin{equation*}
\int_{\phi^*>R} |I_w(\btet)|\d \btet.
\end{equation*}
For this we use a dyadic subdivision of each of the coordinates $\tet_i$, $1\leq i\leq 3$. 
This leads us to the bound
\begin{equation*}
\int_{\phi^*>R} |I_w(\btet)|\d \btet \ll \sum_{2^n >R} \sum_{\substack{m\in \Z\\ m\leq n}}\sum_{\substack{\ell \in \Z\\ \ell \leq m}} \min\left\{2^{\ell+m+n}, 2^{n(3-k/2)}\right\}.
\end{equation*}
We split the right hand side into two sums of the form
\begin{equation*}
\Sigma_1= \sum_{2^n>R} \;\; \sum_{\substack{\ell+m\geq n(2-k/2)\\ \ell \leq m\leq n}}2^{n(3-k/2)},
\end{equation*}
and
\begin{equation*}
\Sigma_2= \sum_{2^n>R} \;\; \sum_{\substack{\ell+m<n(2-k/2)\\ \ell \leq m\leq n}}2^{\ell+m+n},
\end{equation*}
so that
\begin{equation*}
\int_{\phi^*>R} |I_w(\btet)|\d \btet\ll \Sigma_1+\Sigma_2.
\end{equation*}
Now we can easily estimate the sums individually. Indeed,
\begin{equation*}
\begin{split}
\Sigma_2&= \sum_{2^n>R} \sum_{m\leq n} 2^{m+n} \sum_{\substack{\ell \leq m\\ \ell < n(2-k/2)-m}}2^\ell \\
&\ll \sum_{2^n>R} \sum_{\substack{m\leq n\\ n(2-k/2)<2m}} 2^{m+n+n(2-k/2)-m} + \sum_{2^n>R} \sum_{\substack{m\leq n\\ 2m\leq n(2-k/2)}} 2^{2m+n}\\
&\ll \sum_{2^n>R}2^{n(3-k/2)}kn + \sum_{2^n>R}2^{n+n(2-k/2)}\ll R^{(3-k/2)}\log R.
\end{split}
\end{equation*}
Similarly we estimate the contribution of the first sum as
\begin{equation*}
\begin{split}
\Sigma_1= &\sum_{2^n>R}2^{n(3-k/2)}\sum_{\substack{\ell \leq m\leq n\\ n(2-k/2)\leq \ell +m}}1 \ll \sum_{2^n>R}2^{n(3-k/2)} (kn)^2\\ & \ll R^{3-k/2}(\log R)^2.
\end{split}
\end{equation*}
\end{proof}

 In order to give lower bounds on the singular integral as in Theorem \ref{thm_SJ}, it is useful to have the following alternative interpretation:

\begin{proposition}\label{prop15.4}
Let $k >6$. Then one has
\begin{equation*}
J_w(\bmu)= \lim_{\eps \rightarrow 0} \eps^{-3} \int_{\max_i |Q_i(\bfx)-\mu_i|\leq \eps}w(\bfx) \prod_{i=1}^3 \left(1-\frac{|Q_i(\bfx)-\mu_i|}{\eps}\right)\d\bfx.
\end{equation*}
\end{proposition}

\begin{proof}
The proof is a standard way to rewrite the singular integral as the measure of a bounded piece of a manifold. For completeness we give a short proof here. We start by introducing the kernel
\begin{equation}\label{eqn15.4}
K_\eps(\btet)= \prod_{i=1}^3 \left(\frac{\sin (\pi \eps \tet_i)}{\pi\eps\tet_i}\right)^2,
\end{equation}
and note that we have $K_\eps(\btet)= 1+O(\eps^{1/2})$ for $\max_i |\tet_i|\leq \eps^{-1/2}$ and $|K_\eps(\btet)|\ll 1$ for all $\btet$. Now we claim that
\begin{equation}\label{eqn16.4}
J_w(\bmu)= \lim_{\eps \rightarrow 0} \int_{\R^3} K_\eps(\btet) I_w(\btet) e(-\btet\cdot\bmu)\d\btet.
\end{equation}
To justify the claim we need to estimate the following integrals; here we will use the notation $\theta^* = \max_i|\theta_i|$. For the region of small $\btet$ we note that
\begin{equation*}
\int_{\theta^*<\eps^{-1/2}}|K_\eps(\btet)-1||I_w(\btet)|\d\btet \ll \eps^{1/2}\int_{\theta^*<\eps^{-1/2}}|I_w(\btet)|\d\btet \ll \eps^{1/2},
\end{equation*}
where we have used Proposition \ref{prop_sing_int_conv} for $k>6$. On the other hand we estimate the contribution of large $\btet$ by
\begin{equation*}
\begin{split}
\int_{\theta^*>\eps^{-1/2}}|K_\eps(\btet)-1||I_w(\btet)|\d\btet &\ll \int_{\theta^*>\eps^{-1/2}} |I_w(\btet)|\d\btet \\ &\ll (\eps^{-1/2})^{3-k/2}(1+\log^2(\eps^{-1/2})),
\end{split}
\end{equation*}
again using Proposition \ref{prop_sing_int_conv}.
This establishes equation (\ref{eqn16.4}). Now the proposition follows from noting that
\begin{equation*}
\begin{split}
\int_{\R^3}K_\eps(\btet)e(\btet\cdot\blam)\d\btet &= \prod_{i=1}^3 \int_\R\left(\frac{\sin \pi \eps \tet_i}{\pi \eps\tet_i}\right)^2 e(\lam_i\tet_i)\d\tet_i \\
&= \left\{\begin{array}{cc} \eps^{-3}\prod_{i=1}^3 \left(1-\frac{|\lam_i|}{\eps}\right) & \mbox{ if } \max_i|\lam_i|\leq \eps\\ 0 & \mbox{ otherwise.} \end{array}\right.
\end{split}
\end{equation*}
\end{proof}

\section{Exponential sums: major arcs}

We now turn to considering the exponential sums we will encounter on the major arcs. 
Define
\begin{equation*}
S_q(\ba)=\sum_{\bfx \modd{q}} e_q(\ba\cdot\bQ(\bfx)),
\end{equation*}
and
\begin{equation*}
S_q(\ba,\bn)=S_q(\ba)e_q(-\ba\cdot\bn).
\end{equation*}
In our major arc analysis we will be concerned with exponential sums of the form
\begin{equation}\label{Tnq_dfn}
T(\bn;q)=\sum_{\substack{\u{a} \modd{q}\\ (\ba,q)=1}}S_q(\ba;\bn).
\end{equation}
 We note that $T(\bn;q)$ is a multiplicative function, so that for $(q_1,q_2)=1$,
\begin{equation*}
T(\bn;q_1q_2)=T(\bn;q_1)T(\bn;q_2).
\end{equation*}
This reduces the study of $T(\bn;q)$ to the case of $q$ being a prime power. 

Our first bound for $T(\bn;p^e)$, valid for any prime power, will follow from bounds on each individual summand $S_q(\ba,\bn)$:
\begin{prop}\label{lemT1}
Let $p$ be a prime and $e\geq 1$. Then there exists a constant $A_p$ such that 
\begin{equation*}
T(\bn;p^e)\leq A_p p^{e(3+k/2)}.
\end{equation*}
Precisely, we may take $A_p = 4 C_p^{1/2}c_p$ where $C_p$ and $c_p$ are as in Lemma \ref{lemZ} and Lemma \ref{lemF}. In particular, $A_p=A$ may be taken independent of $p$ unless $p|2\Disc(F_{\u{Q}}(\u{x}))$.
\end{prop}
Note that this bound corresponds to obtaining square-root cancellation in $k$ of the variables, but no cancellation from the sum over $\u{a}$ in $T(\u{n};p^e)$. We also remark that this bound is of a typical form, where the constant is dependent on $p$ only for the finitely many bad primes $p$.

We will refine these bounds in the case that $p$ is good, and especially if $p$ is Type I:
\begin{proposition}\label{lemT4}
Assume that $p$ is a good prime. Then 
\begin{equation*}
T(\bn;p)=O\left(p^{\frac{k+4}{2}}\right).
\end{equation*}
If we assume additionally that $p$ is of Type I then we have the stronger bounds
\begin{equation*}
T(\bn;p)=O\left( p^{\frac{k+3}{2}}\right),
\end{equation*}
and for $e \geq 2$,
\begin{equation*}
T(\bn;p^e)=0.
\end{equation*}
The implied constants depend only on the quadratic forms but not on $\bn,$ $p$ or $e$.
\end{proposition}
We note that when $p$ is of Type I we have obtained square-root cancellation in all the variables in $T(\u{n};p).$

\subsection{Proof of Proposition \ref{lemT1}}
We will prove an upper bound for $S_q(\ba,\bn)$ where $q$ is a prime power; we will then use this to deduce Proposition \ref{lemT1}.
We start by considering
\begin{align*}
|S_q(\ba)|^2=\sum_{\bfx,\bfy\modd{q}}e_q(\ba\cdot\bQ(\bfx)-\ba\cdot\bQ(\bfy));
\end{align*}
substituting $\bfx=\bfy+\bfz$, we obtain
\begin{align*}
|S_q(\ba)|^2&=\sum_{\bfy,\bfz  \modd{q}} e_q(\ba \cdot \bQ(\bfz)+2\bfz^t(\ba\cdot\bQ)  \bfy)\\
&=q^k \sum_{\substack{\bfz \modd{q}\\ q|2\bfz^t(\ba\cdot\bQ)}} e_q(\ba\cdot\bQ(\bfz)).
\end{align*}
Upon defining the counting function
\begin{equation}\label{Zaq_dfn}
Z(\ba,q)=\#\{\bfz \modd{q}:\ q|2\bfz^t(\ba\cdot\bQ)\},
\end{equation}
we have the upper bound
\begin{equation}\label{eqn5.0}
|S_q(\ba)|^2\leq q^k Z(\ba,q).
\end{equation}
Note that so far this argument holds for any modulus $q$, not necessarily a prime power.

We will bound $Z(\u{a},p^e)$ by using the Smith normal form and the determinant form $F_{\u{Q}}(\u{x}) = \det (\u{x} \cdot \u{Q})$; for this we will require a  maneuver from working modulo $p^e$ to  residue classes modulo $p^{ek}$. What follows differs significantly from the treatment in \cite{HBPierce}, where a simpler treatment was possible. 

We record the following facts about $Z(\u{a},p^e)$.

\begin{lemma}\label{lemZ}
Let $p$ be a prime, $e\geq1$ and $\ba\in (\Z/p^e\Z)^3$. Then, for any $\bb\in (\Z/p^{ke}\Z)^3$ with $\bb\equiv \ba$ modulo $p^e$, we have the bound
\begin{equation*}
Z(\ba,p^e)\leq C_p \gcd (F_{\u{Q}}(\bb),p^{ke}),
\end{equation*}
where $C_p=1$ for $p\neq 2$ and $C_2=2^k$.

Moreover, upon averaging over all $\u{a}$ modulo $p^{e}$ with $(\u{a},p)=1$, we have for $\kappa=1/2$ or $\kappa=1$, 
\[ \sum_{\bstack{ \u{a} \modd{p^{e}}}{(\u{a},p)=1}} Z(\u{a},p^e)^\kappa \leq C_p^\kappa c_p c_\kappa(e) p^{3e},\]
where $C_p$ is as above, $c_p$ (from Lemma \ref{lemF}) is independent of $p$ for $p\ndiv 2 \Disc(F_{\u{Q}})$, and $c_\kappa(e)=4$ if $\kappa=1/2$ and $c_\kappa(e) = ke+1$ if $\kappa=1$. 
\end{lemma}

Note that this last statement serves the role of proving, roughly speaking, that on average over $\u{a} \modd{p^e}$ with $(\u{a},p)=1$,  $Z(\u{a},p^e)$ is $O(1)$.
For the moment we assume this lemma and proceed with the proof of Proposition \ref{lemT1}.
We recall that
\begin{equation*}
T(\bn;p^e)= \sum_{\substack{1\leq \ba\leq p^e\\ (\ba,p)=1}}S_{p^e}(\ba;\bn),
\end{equation*}
so that taking absolute values and applying equation (\ref{eqn5.0}) gives
\begin{equation*}
|T(\bn;p^e)|\leq \sum_{\substack{1\leq \ba\leq p^e\\ (\ba,p)=1}}p^{ke/2}Z(\ba,p^e)^{1/2}.
\end{equation*}
An application of Lemma \ref{lemZ} with $\kappa=1/2$ thus shows that 
\begin{equation*}
|T(\bn;p^e)| \leq  4C_p^{1/2}c_p p^{3e +ke/2},
	\end{equation*}
	as desired.

\begin{proof}[Proof of Lemma \ref{lemZ}]
The first part of the lemma is a consequence of Proposition 7 in \cite{Loxton}.
For the second part of the lemma, we argue as follows. After averaging over all $\bb$ modulo $p^{ke}$ that reduce to each $\u{a}$ modulo $p^e$,  the first part of the lemma implies
\beq\label{ZF_sum}
\sum_{\bstack{ \u{a} \modd{p^{e}}}{(\u{a},p)=1}} Z(\u{a},p^e)^\kappa  \leq C_p^\kappa p^{3e(1-k)} \sum_{\substack{\u{b} \modd{p^{ke}}\\ (\bb,p)=1}} \gcd(F_{\u{Q}}(\bb),p^{ke})^{\kappa}.
\eeq
Next we note that
\begin{equation*}
\begin{split}
 \sum_{\bstack{ \u{b} \modd{p^{ke}}}{(\u{b},p)=1}}\gcd(F_{\u{Q}}(\bb),p^{ke})^{\kappa} 
&\leq \sum_{f=0}^{ke}p^{\kappa f}\#\{\u{b} \modd{p^{ke}}: (\bb,p)=1,\ p^f|F_{\u{Q}}(\bb)\}\\
&= \sum_{f=0}^{ke}p^{\kappa f} p^{3(ke-f)}\#\{\u{b} \modd{p^f}: (\bb,p)=1,\ p^f|F_{\u{Q}}(\bb)\}.
\end{split}
\end{equation*}
We now use Lemma \ref{lemF} to understand how often $F_{\u{Q}}(\u{b})$ is divisible by the prime power $p^f$ as $\u{b}$ varies over appropriate residue classes; we obtain
\begin{equation*}
\begin{split}
 \sum_{\bstack{ \u{b} \modd{p^{ke}}}{(\u{b},p)=1}}\gcd(F_{\u{Q}}(\bb),p^{ke})^{\kappa} 
 &\leq  \sum_{f=0}^{ke}p^{\kappa f} p^{3(ke-f)}(c_p p^{2f})\\
 & =  c_p p^{3ke} \sum_{f=0}^{ke} p^{(\kappa-1)f} \\
 & =  c_p c_\kappa(e) p^{3ke},
\end{split}
\end{equation*}
where $c_p$ is as in Lemma \ref{lemF} and $c_\kappa(e)$ is as defined in the statement of the present lemma. We apply this in (\ref{ZF_sum}) to conclude that 
\[\sum_{\bstack{ \u{a} \modd{p^{e}}}{(\u{a},p)=1}} Z(\u{a},p^e)^\kappa \leq C_p^{\kappa} p^{3e(1-k)}  c_p c_\kappa(e) p^{3ke} = C_p^{\kappa}c_p c_\kappa(e)p^{3e},\]
as claimed.

\end{proof}

\subsection{Refined bounds for good primes}
In order to obtain better bounds on $T(\bn;p^e)$ for good primes, we first give an alternative description of $T(\bn;p^e)$  in terms of counting functions related to the system of congruences $\bQ(\bfx)\equiv \bn$. More precisely we define
\begin{equation*}
N(\bn;q)=\#\{\bfx \modd{q}: \bQ(\bfx)\equiv \bn \modd{q}\}.
\end{equation*}

\begin{lemma}\label{lemT2}
Let $p$ be a prime and $e\geq 1$. Then 
\begin{equation*}
T(\bn;p^e)=p^{3e} N(\bn;p^e)-p^{k+3(e-1)}N(\bn;p^{e-1}).
\end{equation*}
\end{lemma}

\begin{proof}
First we observe that
\begin{equation*}
\sum_{1\leq \ba\leq p^e}S_{p^e}(\ba;\bn)= p^{3e} N(\bn;p^e).
\end{equation*}
Hence we can rewrite the exponential sum $T(\bn;p^e)$ as
\begin{equation*}
\begin{split}
T(\bn;p^e)&= \sum_{1\leq \ba \leq p^e} S_{p^e}(\ba;\bn)- \sum_{1\leq \ba\leq p^{e-1}}S_{p^{e-1}}(p\ba;\bn)\\ &= p^{3e}N(\bn;p^e)-p^{k+3(e-1)}N(\bn;p^{e-1}).
\end{split}
\end{equation*}
\end{proof}

Analogous to \cite{HBPierce}, we will use explicit counts for $N(\bn;p^e)$ to give upper bounds on $T(\bn;p^e)$. The main ingredient for this is an application of Deligne's estimates to  $N(\bn;p)$.

\begin{lemma}\label{lemN1}
For $p$ a good prime,
\begin{equation*}
N(\bn;p)=p^{k-3}+O\left(p^{\frac{k-2}{2}}\right).
\end{equation*}
For $p$ a prime of Type I,
\begin{equation*}
N(\bn;p)=p^{k-3}+O\left(p^{\frac{k-3}{2}}\right).
\end{equation*}
\end{lemma}

Before proving this, we record a brief lemma.
\begin{lemma}\label{lemma_pTI}
If $p$ is of Type I, then the projective variety $V$ defined over $\F_p$ by 
\[ \u{Q}(\xbf) = \u{n}x_0^2
\]
is smooth.
If $p$ is of Type II, then $V$ has singular locus of dimension at most $0$ over $\F_p$. 
\end{lemma}
Briefly, the proof proceeds as follows: points on the  intersection of $V$ with the hyperplane $x_0=0$ are nonsingular since $p$ is good, and points contained in the open set $x_0\neq 0$ are nonsingular since $p$ is assumed to be of Type I, so that $p\ndiv H_{\u{Q}}(\u{n})$ and hence $\rank J_{\u{Q}}(\xbf) =3$ over $\F_p$, for every point $\xbf$ such that $\u{Q}(\xbf)=\u{n}$. 
If $p$ is of Type II, then we claim that the singular locus of the projective variety $V$ has dimension at most $0$. If it had positive dimension, then the intersection of the singular locus with the hyperplane $x_0=0$ would be non-empty, which is impossible since $p$ is a good prime and the resulting intersection is smooth.

\begin{proof}[Proof of Lemma \ref{lemN1}]
We projectivize the problem and rewrite the counting function $N(\bn;p)$ as
\begin{equation*}
N(\bn;p)=(p-1)^{-1}(N^{(1)}(\bn;p)- N^{(2)}(\bn;p)),
\end{equation*}
with counting functions of the form
\begin{equation*}
N^{(1)}(\bn;p)= \#\{(x_0,\bfx)\in \F_p^{k+1}: \bQ(\bfx)=\bn x_0^2\},
\end{equation*}
and
\begin{equation*}
N^{(2)}(\bn;p)= \#\{\bfx\in \F_p^k: \bQ(\bfx)=\u{0}\}.
\end{equation*}
If $p$ is good, then the projective variety $\bQ(\bfx)=\u{0}$ is smooth over $\F_p$, and hence Deligne's bound (in the form of Lemma \ref{lemma_Hooley} with $s=-1$) delivers an asymptotic of the form
\begin{equation*}
N^{(2)}(\bn;p)= p^{k-3}+O\left( p^{\frac{k-2}{2}}\right).
\end{equation*}
In order to give asymptotics for the counting function $N^{(1)}(\bn;p)$ we need to distinguish between primes $p$ of Type I and Type II. If $p$ is of Type I, then the projective variety given by $\bQ(\bfx)=\bn x_0^2$ is smooth, by Lemma \ref{lemma_pTI}. Hence Deligne's bound (Lemma \ref{lemma_Hooley}) delivers in this case
\begin{equation*}
N^{(1)}(\bn;p)= p^{k-2}+O\left(p^{\frac{k-1}{2}}\right).
\end{equation*}
If $p$ is of Type II, then we recall from Lemma \ref{lemma_pTI} that the singular locus of the projective variety $\bQ(\bfx)=\bn x_0^2$ has dimension at most $0$, so we apply Hooley's extension of the Deligne bound (Lemma \ref{lemma_Hooley} with $s \leq 0$) to obtain
\begin{equation*}
 N^{(1)}(\bn;p)= p^{k-2}+O\left(p^{\frac{k}{2}}\right).
\end{equation*}

We now assemble these results. In the case of $p$ being of Type I we  compute
\begin{equation*}
N(\bn;p)= (p-1)^{-1}(p^{k-2}-p^{k-3})+O\left(p^{\frac{k-3}{2}}\right) = p^{k-3}+O\left(p^{\frac{k-3}{2}}\right).
\end{equation*}
Similarly, we obtain for $p$ of Type II an expression for $N(\bn;p)$ of the form
\begin{equation*}
N(\bn;p)= p^{k-3} +O\left(p^{\frac{k-2}{2}}\right).
\end{equation*}
\end{proof}

With Lemma \ref{lemN1} in hand, we may quickly  deduce the upper bounds on $T(\bn;p^e)$ for good primes $p$ given in Proposition \ref{lemT4}.
By Lemma \ref{lemT2}, we have
\begin{equation*}
T(\bn;p)= p^3N(\bn;p)-p^k,
\end{equation*}
and hence the estimates for $T(\bn;p)$ follow from Lemma \ref{lemN1}. If $e\geq 2$ and $p$ is of Type I, then one has the recursion
\begin{equation*}
N(\bn;p^e)=p^{k-3}N(\bn;p^{e-1}),
\end{equation*}
which is a consequence of Hensel's lemma. Together with Lemma \ref{lemT2} this establishes $T(\bn;p^e)=0$ for $e\geq 2$ and $p$ of Type I.

\subsection{Lower bounds for local densities}
In order to provide lower bounds for the singular series we need to understand the local densities 
\begin{equation}\label{loc_dens}
\sig_p(\bn)=\sum_{e=0}^\infty p^{-ek}T(\bn;p^e).
\end{equation}
Note that by Proposition \ref{lemT1} we have $T(\bn;p^e)\ll p^{e(3+k/2)}$, and hence $\sig_p(\bn)$ is absolutely convergent for $k>6$.\par
The goal of this section is to provide a lower bound on $\sig_p(\bn)$ in terms of $\bn$, for any prime $p$.
We will prove:
\begin{proposition}\label{prop5.2}
There exists a positive real number $\alptil$ such that for any prime $p$, if there exists a solution $\xbf_0 \in \Z_p^{k}$ to $\u{Q}(\xbf) \con \u{n}$ in $\Z_p$, then
\begin{equation*}
\sig_p(\bn)\geq \ome_p |H_{\u{Q}}(\bn)|_p^{2\alptil (k-3)},
\end{equation*}
for some constant $\ome_p$ depending on $p$ and the system of quadratic forms $\bQ$.
\end{proposition}

In order to prove Proposition \ref{prop5.2} we will require a lemma on the singularity of the Jacobian matrix at a local solution to $\u{Q}(\xbf) = \u{n}$. We recall that the Jacobian matrix is given by 
\[ J_{\u{Q}}(\bfx) = \left( \begin{array}{c}
	\nabla Q_1 (\bfx)\\
	\nabla Q_2(\bfx)\\
	 \nabla Q_3(\bfx)\end{array} \right).\]
Given any three distinct column indices $1 \leq i,j,\ell \leq k$ we will let  $\Del_{ij\ell}(\bfx)$ denote the determinant of the corresponding $3\times 3$ minor of $J_{\u{Q}}(\xbf)$.
The key point is  that for any solution $\xbf_0$ to the system of equations $\bQ(\bfx)=\bn$ over $\Z_p$ with $H_{\u{Q}}(\bn)\neq 0$, the corresponding Jacobian matrix $J_{\u{Q}}(\xbf_0)$ cannot be too singular, in the sense that there is some $3\times 3$ minor whose determinant is divisible only by a bounded power of $p$ depending on $H_{\u{Q}}(\bn)$. In particular, the constant $\alptil$ is provided by the following key lemma:
\begin{lemma}\label{lem5.6}
There is a positive real number $\alptil$ such that the following holds.
For each prime $p$ there is a positive real number $b_p$ such that if $\bfx_0\in \Z_p^k$ is a solution to $\bQ(\bfx_0)=\bn$, then
\begin{equation*}
|H_{\u{Q}}(\bn)|_p^\alptil\leq b_p \max_{i,j,\ell}|\Del_{ij\ell}(\bfx_0)|_p.
\end{equation*}
One can take $\alptil= 3^{3+k}$.
\end{lemma}
We note that an analogue to this lemma appeared in \cite[Lemma 4.8]{HBPierce}; however that proof depended on simultaneously diagonalizing the system of two quadratic forms. In our setting we cannot diagonalize the three forms simultaneously, and we approach the proof of Lemma \ref{lem5.6} from an entirely new perspective.

\begin{proof}[Proof of Lemma \ref{lem5.6}]

We apply the Nullstellensatz to the system of polynomials  in the variables $\bfx,\bn$ comprising the polynomials $Q_i(\bfx)-n_i$ for $1\leq i\leq 3$ and all the $3\times 3$ minors $\Del_{ij\ell}(\bfx)$. If they have a common zero over $\overline{\Q}$, then by Proposition \ref{prop3.4}, one has $H_{\u{Q}}(\bn)=0$. Hence there is some power $\alptil\in \N$ and some polynomials $g_i(\bfx,\bn)$ and $h_{ij\ell}(\bfx,\bn)$ such that
\begin{equation}\label{eqn5.6}
H_{\u{Q}}(\bn)^\alptil= \sum_{i=1}^3 g_i(\bfx,\bn) (Q_i(\bfx)-n_i) + \sum_{1 \leq i,j,\ell\leq k} h_{ij\ell}(\bfx,\bn) \Del_{ij\ell}(\bfx).
\end{equation}
Since we applied the Nullstellensatz over $\overline{\Q}$, the polynomials $g_i$ and $h_{ij\ell}$ may have coefficients in $\overline{\Q}$. Yet these polynomials contain only a finite number of coefficients, so that they all lie in some finite extension of $\Q$. Taking the trace of this extension down to $\Q$ in the above equation shows that we may assume that $g_i(\bfx,\bn),h_{ij\ell}(\bfx,\bn)\in \Q[\bfx,\bn]$.\par
Now assume that  $\bn\in \Z^3$ is fixed, and there exists $\bfx_0\in \Z_p^k$ with $\bQ(\bfx_0)=\bn$, so that the first set of terms on the right hand side of (\ref{eqn5.6}) vanishes for $\xbf =\xbf_0$. Then taking $p$-adic absolute values in equation (\ref{eqn5.6}) shows that 
\begin{equation*}
|H_{\u{Q}}(\bn)|_p^{\alptil}\leq b_p \max_{i,j,\ell} |\Del_{ij\ell}(\bfx_0)|_p,
\end{equation*}
where one can take $b_p$ to be the maximum of the $p$-adic absolute values of any coefficient appearing in any of the polynomials $h_{ij\ell}(\bfx,\bn)$.

Furthermore, one can ensure $\alptil \leq 3^{3+k}$ by an explicit version of Hilbert's Nullstellensatz; see Theorem 1.5 together with Remark 1.6 in Koll\'ar's work \cite{Kollar}. That work applies to homogeneous polynomials, but we may easily homogenize $Q_i(\bfx)-n_i$ to $Q_i(\bfx)-t^2n_i$ and apply the Nullstellensatz to $H_{\u{Q}}(t^2\bn)$. 
After setting $t=1$ in the resulting equation, we obtain the desired result.
\end{proof}

\begin{proof}[Proof of Proposition \ref{prop5.2}]
We return to proving Proposition \ref{prop5.2} on lower bounds for the local densities $\sig_p(\u{n})$. 
We consider a truncated piece of the series defining $\sig_p(\bn)$, setting
\begin{equation*}
\sig_p(\bn;E):= \sum_{e=0}^E p^{-ek} T(\bn;p^e),
\end{equation*}
so that $\sig_p(\bn)= \lim_{E\rightarrow \infty} \sig_p(\bn;E)$. 
Our first step is to interpret $\sig_p(\bn;E)$ as the number of local solutions to the system of equations $\bQ(\bfx)=\bn$.
 For this we use Lemma \ref{lemT2} to rewrite $\sig_p(\bn;E)$ as
\begin{equation*}
\sig_p(\bn;E)= T(\bn;p^0) +\sum_{e=1}^E p^{-ek} \left(p^{3e}N(\bn;p^e)-p^{k+3(e-1)}N(\bn;p^{e-1})\right).
\end{equation*}
Noting that the sum on the right hand side is a telescoping sum, we obtain
\begin{equation}\label{apply_tel}
\sig_p(\bn;E)= p^{E(3-k)} N(\bn;p^E).
\end{equation}
In order to show that the limit of $\sig_p(\bn;E)$ is bounded below for $E\rightarrow \infty$, we need to prove a lower bound for $N(\bn;p^E)$. Let $\bfx_0\in\Z_p^k$ be a given solution to $\bQ(\bfx_0)=\bn$, which is assumed to exist in the hypothesis of the proposition. Then Lemma \ref{lem5.6} implies that 
\begin{equation}\label{eqn5.7}
|H_{\u{Q}}(\bn)|^{\alptil}_p\leq b_p \max_{i,j,\ell}|\Del_{ij\ell}(\bfx_0)|_p.
\end{equation}
For simplicity of notation assume that the maximum on the right hand side is attained for $(i,j,\ell)=(1,2,3)$. Since the statement of the proposition we intend to prove is trivial for $H_{\u{Q}}(\bn)=0$, we may assume that there is some $\nu\in \Z$ with 
\beq\label{Del_nu}
|\Del_{123}(\bfx_0)|_p= p^{-\nu}.
\eeq
We now give a lower bound on $N(\bn;p^E)$ for $E\geq 2\nu+1$. For this we write $\bfw=(w_1,w_2,w_3)$ and $\bfy= (y_4,\ldots, y_k)$. We choose $\bfy$ to be any element modulo $p^E$ with the restriction that $\bfy\equiv \mathbf{0} $ modulo $p^{2\nu+1}$; there are $p^{(k-3)(E-2\nu-1)}$ such choices. We claim that for any such choice of $\bfy$ the system of equations
\begin{equation}\label{eqn5.8}
\bQ(\bfx_0+(\bfw,\bfy))\con \bn \modd{p^E}
\end{equation}
has an integral solution $\wbf$.
Once we prove we can do this for any of our choices for $\bfy$, we will obtain the lower bound
\begin{equation*}
N(\bn;p^E)\geq p^{(k-3)(E-2\nu-1)}.
\end{equation*}
This translates into a lower bound for the density $\sig_p(\bn)$ via (\ref{apply_tel}), namely
\begin{equation*}
\sig_p(\bn)=\lim_{E\rightarrow \infty} \sig_p(\bn;E) \geq p^{-(k-3)(2\nu+1)} .
\end{equation*}
As written, this lower bound still depends on $\nu$, and we use equation (\ref{eqn5.7}) to rephrase the lower bound as
\begin{equation*}
\sig_p(\bn) \geq p^{-(k-3)} \left(b_p^{-1}|H_{\u{Q}}(\bn)|_p^\alptil\right)^{2(k-3)} \geq \ome_p |H_{\u{Q}}(\bn)|_p^{2\alptil (k-3)},
\end{equation*}
with 
\begin{equation*}
\ome_p= p^{-(k-3)}b_p^{-2(k-3)},
\end{equation*}
thus completing the proof of Proposition \ref{prop5.2}, pending the claim on (\ref{eqn5.8}).

Our claim on the solubility of (\ref{eqn5.8}) is a consequence of Hensel's lemma. 
Temporarily defining $F_i(\wbf) = Q_i(\xbf_0 + (\wbf,\ybf)) - Q_i(\xbf_0)$ for $i=1,2,3$, we seek an integral solution $\wbf$ to the system 
\[ F_1(\wbf) = F_2(\wbf) = F_3(\wbf) \con 0 \modd{p^E}.\]
For any choice of $\ybf$ modulo $p^E$ that reduces to $\mathbf{0} \modd{p^{2\nu+1}}$, and for $\wbf \con \mathbf{0} \modd{p^{2\nu+1}}$, we have $\xbf_0 + (\wbf,\ybf) \con \xbf_0 \modd{p^{2\nu+1}}$, so that  (\ref{eqn5.8}) certainly holds as a congruence modulo $p^{2\nu+1}$. Moreover, for such $\wbf,\ybf$, the first $3\times 3$ minor $\Del_{123}(\xbf_0)$ of the Jacobian $J_{\u{Q}}(\xbf_0)$ has $\Del_{123}(\xbf_0) \con \Del_{123}(\xbf_0 + (\wbf,\ybf)) \modd{p^{2\nu+1}}$, and so in particular $|\Del_{123}(\xbf_0 + (\wbf,\ybf))|_p= p^{-\nu} $,  by (\ref{Del_nu}). 
These facts indicate that $p^{2\nu+1} | F_i(\mathbf{0})$ for $i=1,2,3$ while
\[ \det \left( \begin{array}{c} \nabla F_1(\mathbf{0}) \\ \nabla F_2(\mathbf{0})  \\ \nabla F_3(\mathbf{0})  \end{array} \right) \not\con 0 \modd{p^{\nu+1}},\]
so that Hensel's lemma indicates the existence of a solution $\wbf$ such that (\ref{eqn5.8}) holds modulo $p^E$. 
\end{proof}

\section{Exponential sums: minor arcs}
For $\bfl = (\bfl_1,\bfl_2)\in\Z^k\times \Z^k$ and a positive integer $q$, we introduce the exponential sum
\begin{equation*}
S(\bfl;q)=\sum_{\substack{\ba \modd{q}\\ (\ba,q)=1}}\sum_{\substack{\bfr_1 \modd{q}\\ \bfr_2 \modd{q}}}e_q(\ba\cdot\bQ(\bfr_1)-\ba\cdot\bQ(\bfr_2))e_q(\bfr_1\cdot \bfl_1+\bfr_2\cdot \bfl_2).
\end{equation*}
We note that $S(\bfl;q)$ is a multiplicative function in $q$, so that it is sufficient to understand $S(\bfl;p^e)$ for all prime powers $p^e$. The goal of this section is to give pointwise upper bounds as well as average upper bounds for these sums. We start with pointwise upper bounds:
\begin{lemma}\label{lemS1}
For any prime $p$,
\begin{equation*}
|S(\bfl;p^e)|\ll c'_p(ke+1) p^{e(k+3)},
\end{equation*}
where $c'_p$ is independent of $e$; in particular $c_p' = C_p c_p$ where $C_p,c_p$ are as in Lemma \ref{lemZ}, and in particular are independent of $p$ for all primes $p \ndiv 2 \Disc(F_{\u{Q}})$. Consequently, for any composite modulus $q$ we have
\begin{equation*}
S(\bfl;q)\ll_\eps q^{k+3+\eps},
\end{equation*}
for any $\eps >0$, where the implied constant is independent of $q$.
\end{lemma}

In the case of good primes we will refine this estimate for $S(\bfl;p^e)$. We first introduce a $(k+1) \times (k+1)$ matrix defined by
\begin{equation*}
M(\u{x};\b{x},\b{y}):=
\left(\begin{array}{c|c} \rule[-3mm]{0mm}{1mm}\u{x} \cdot \u{Q} & \b{y}\\ \hline
\rule[3mm]{0mm}{1mm}\b{x}^t & 0\end{array}\right),
\end{equation*}
where $\bfx,\bfy\in\Z^k$. Furthermore, we set $\bfl_3:=\bfl_1+\bfl_2$ and $\bfl_4:=\bfl_1-\bfl_2$. With this notation, our refined estimate for good primes is as follows:

\begin{lemma}\label{lemS2}
For any prime $p \ndiv 2 \Disc(F_{\u{Q}})$,
\begin{align*}
S(\bfl;p^e)\ll &(ke+1)p^{e(k+3)} (p^{-1}  +p^{-3}f(\lbf;p)+p^{-2}g(\lbf;p)),
\end{align*}
where 
\begin{eqnarray*}
f(\lbf;p) & = & \#\{\bb \modd{p}:\ p|\det M(\bb;\bfl_3,\bfl_4)\} \\
g(\lbf;p) & = & \#\{\bb \modd{p}: \ \rank \left(M(\bb;\bfl_3,\bfl_4)\right)\leq k-1\}.
\end{eqnarray*}
Here $\rank \left(M(\bb;\bfl_3,\bfl_4)\right)$ is the rank of the matrix over the finite field $\F_p$.
\end{lemma}
We note that the first term improves on Lemma \ref{lemS1} by a factor of $p^{-1}$; in order for the second and third terms to constitute an improvement on Lemma \ref{lemS1}, we will later show (roughly speaking) that $f(\lbf;p)$ and $g(\lbf;p)$ contribute at most $p^2$ and $p$, respectively, on average over $\lbf$ (see Proposition \ref{propSav}).

\begin{proof}[Proof of Lemma \ref{lemS1}]
The second part of Lemma \ref{lemS1} follows quickly from the first part and the multiplicativity of $S(\bfl;q)$ as a function in $q$. More precisely, there is some positive constant $C$ such that
\begin{equation*}
S(\bfl;q)\ll C^{\ome (q)} d(q)^k q^{k+3}\ll_\eps q^{k+3+\eps}.
\end{equation*}
Note that the implied constant (resulting from the application of Lemma \ref{lemF}) may be taken independent of the prime factors, since only a finite number of primes $p | 2\Disc (F_{\u{Q}})$, and the corresponding factors $c_p$ may be included in a universal constant.

For the first part of the lemma we proceed similarly to the proof of Proposition \ref{lemT1}. First we rewrite $S(\bfl;q)$ via the variable substitution $\bfr_1=\bfr_2+\bfh$ and $\bfl_3=\bfl_1+\bfl_2$ as
\begin{equation*}
S(\bfl;q)= \sum_{\substack{\u{a} \modd{q}\\ (\ba,q)=1}}\sum_{\bfr_2 \modd{q}}\sum_{\bfh \modd{q}} e_q(\ba\cdot\bQ(\bfh)+2\bfh^t(\ba\cdot\bQ)\bfr_2+\bfl_3 \cdot \bfr_2+\bfl_1 \cdot \bfh).
\end{equation*}
Evaluating the linear sum over $\bfr_2$ modulo $q$ leads to
\begin{equation*}
S(\bfl;q)\ll q^k \sum_{\substack{\ba \modd{q}\\ (\ba,q)=1}} \#\calS (\bfl_3,\ba;q),
\end{equation*}
with $\calS (\bfl_3,\ba;q)$ defined by
\begin{equation*}
\calS(\bfl_3,\ba;q)=\{\bfh \modd{q}:\ q|2\bfh^t(\ba\cdot\bQ)+\bfl_3^t\}.
\end{equation*}
Note that this set is either empty or a coset of $\calS(\mathbf{0},\ba;q);$ we recall from our previous notation (\ref{Zaq_dfn}) that $\# \calS(\mathbf{0},\ba;q)=Z(\ba,q)$. Hence we may further estimate $S(\bfl;q)$ by
\begin{equation*}
S(\bfl;q)\ll q^k \sum_{\substack{\ba \modd{q}\\ (\ba,q)=1}} Z(\ba,q).
\end{equation*}
Now suppose $q=p^e$ is a prime power. We apply Lemma \ref{lemZ} with $\kappa=1$  to conclude that 
\begin{equation*}
\begin{split}
S(\bfl;p^e)&\ll p^{ek} . C_pc_p (ke+1)p^{3e} ,\end{split}
\end{equation*}
in which $c_p$ and $C_p$ are as in Lemma \ref{lemZ}, which is sufficient.
We recall that $c_p$ and $C_p$ are independent of $p$ for all $p \ndiv 2 \Disc(F_{\u{Q}})$. 
\end{proof}

\subsection{Refined bounds for good primes}

Now let $p$ be a good prime and consider $q=p^e$. Note that in particular we have $p\neq 2$ and hence  the transformation $\bfx=\bfu+\bfv$, $\bfy=\bfu-\bfv$ is invertible modulo $p$. We apply this to the exponential sum $S(\bfl;p^e)$ and obtain
\begin{align*}
S(\bfl;p^e)&=\sum_{\substack{\ba \modd{p^e}\\ (\ba,p)=1}} \sum_{\bfu,\bfv \modd{p^e}}e_{p^e}(4\bfu^t(\ba\cdot\bQ)\bfv+\bfl_3^t\bfu+\bfl_4^t\bfv) \\
&= p^{ek} \sum_{\substack{\ba \modd{p^e}\\ (\ba,p)=1}}\sum_{\substack{\bfu \modd{p^e}\\ p^e|4\bfu^t(\ba\cdot\bQ)+\bfl_4^t}}e_{p^e}(\bfl_3^t\bfu).
\end{align*}
Since $p \neq 2$, the factor of $4$ may be eliminated by a change of variables in $\u{a}$.
We introduce an additional summation and obtain 
\begin{equation*}
\phi(p^e)S(\bfl;p^e)=p^{ek}\sum_{\substack{ r \modd{p^e} \\(r,p)=1}}\sum_{\substack{\ba \modd{p^e}\\ (\ba,p)=1}} \sum_{\substack{\bfu \modd{p^e}\\ p^e|r^{-1}\bfu^t (r\ba\cdot\bQ)+\bfl_4^t}}e_{p^e}(\lbf_3^t\bfu).
\end{equation*}
Next we substitute $\u{b}$ for $r\ba$  and $\bfw$ for $r^{-1}\bfu$. This leads to
\begin{equation*}
\phi(p^e)S(\bfl;p^e)=p^{ek}\sum_{\substack{ r \modd{p^e} \\(r,p)=1}}\sum_{\substack{\bb \modd{p^e}\\ (\bb,p)=1}} \sum_{\substack{\bfw \modd{p^e}\\ p^e|\bfw^t (\bb\cdot\bQ)+\bfl_4^t}}e_{p^e}(r\bfl_3^t\bfw).
\end{equation*}
Now we evaluate the summation over $r$, to conclude that
\begin{equation}\label{SN1N2}
S(\bfl;p^e)=\phi(p^e)^{-1}p^{ek}(p^eN_2(p^e)-p^{e-1}N_1(p^e)),
\end{equation}
with counting functions of the form
\begin{equation*}
N_1(p^e)=\#\{\bb,\bfw \modd{p^e}: (\bb,p)=1,\ p^e|(\bfw^t(\bb\cdot\bQ)+\bfl_4^t),\ p^{e-1}|\bfl_3^t\bfw\},
\end{equation*}
and
\begin{equation*}
N_2(p^e)=\#\{\bb,\bfw \modd{p^e}: (\bb,p)=1,\ p^e|(\bfw^t(\bb\cdot\bQ)+\bfl_4^t),\ p^e|\bfl_3^t\bfw\}.
\end{equation*}
We first estimate the contribution to $S(\bfl;p^e)$ arising from $N_1(p^e)$. Since $N_1(p^e)$ is multiplied by the smaller factor $p^{e-1}$, we can afford to ignore  the additional condition that $p^{e-1}|\bfl_3^t\bfw$ within $N_1(p^e)$. Furthermore, we note that for fixed $\bb$ the set of solutions $\bfw$ modulo $p^e$ such that
\begin{equation*}
(\bb\cdot\bQ)\bfw+\bfl_4\equiv \mathbf{0}\modd{p^e},
\end{equation*}
is either empty or is a coset of the set of $\bfw$ that are solutions to the homogeneous system $(\bb\cdot\bQ)\bfw\equiv 0 \modd{p^e}$. Hence we see that
\begin{equation*}
N_1(p^e)\leq \sum_{\substack{\bb \modd{p^e}\\ (\bb,p)=1}}Z_0(\bb,p^e),
\end{equation*}
where
\beq\label{Z0_dfn}
Z_0(\bb,p^e)=\#\{\wbf \modd{p^e}:\ p^e|\wbf^t(\bb\cdot\bQ)\}.
\eeq
We note that since $p$ is good, $p$ must be odd, so that in fact $Z_0(\ba,p^e) = Z(\ba,p^e)$, 
where $Z(\bb,p^e)$ is defined by (\ref{Zaq_dfn}), as usual. 
Thus we may apply Lemma \ref{lemZ} with $\kappa=1$ (and $C_p=1$ since $p \neq 2$), so that  
\[ N_1(p^e) \leq c_p(ke+1)p^{3e},\]
where we may take $c_p$ to be independent of $p$ since $p$ is good (so that $p \ndiv 2\Disc(F_{\u{Q}})$).
Hence the contribution of $N_1(p^e)$ to $S(\bfl;p^e)$ in (\ref{SN1N2}) is bounded by
\begin{equation*}
\ll (ke+1) p^{e(k+3)-1},
\end{equation*}
with an implied constant independent of $p$.
This is sufficient for Lemma \ref{lemS2}.

We next consider the contribution from $N_2(p^e)$. For this we first fix a coefficient vector $\bb$ modulo $p^e$ and consider the set of solutions $\bfw$ modulo $p^e$ to the system of congruences
\begin{equation}\label{eqn6.10}
\begin{split}
(\bb\cdot\bQ)\bfw+\bfl_4&\equiv 0\modd{p^e} \\
\bfl_3^t\bfw&\equiv 0 \modd{p^e}.
\end{split}
\end{equation}
If the set of solutions in $\bfw$ is non-empty, then there is in particular a solution to this system modulo $p$; in other words the vector $(\wbf,1)$ satisfies
\[ 
M(\u{b}; \lbf_3,\lbf_4) \left(\begin{array}{c}  \wbf \\1  \end{array}\right) =\mathbf{0}
\]
in $\F_p$.
Hence 
\begin{equation}\label{remark_det}
p|\det M(\bb;\bfl_3,\bfl_4).
\end{equation}
By Lemma \ref{geolem1}, since $p$ is good and $(\u{b},p)=1$, the rank of $\bb\cdot\bQ$ over $\F_p$ is at least $k-1$ and hence we need to distinguish two cases; we will let $N_3(p^e)$ denote the contribution to $N_2(p^e)$ when the rank of $M(\bb;\bfl_3,\bfl_4)$ is $k$ over the field $\F_p$, and $N_4(p^e)$ denote the contribution to $N_2(p^e)$ when the rank of $M(\bb;\bfl_3,\bfl_4)$ is $k-1$ over the field $\F_p$.

First consider $N_3(p^e)$, in which case the rank of $M(\bb;\bfl_3,\bfl_4)$ over $\F_p$ is $k$. Then we claim that the system (\ref{eqn6.10}) has at most one solution $\bfw$ modulo $p^e$.
 Indeed, if it has any solution modulo $p$, then by the previous observation, the column vector $(\bfl_4,0)$ lies in the span of the columns of the $(k+1) \times k$ matrix 
\beq\label{part_matrix}
\left(\begin{array}{c} \rule[-3mm]{0mm}{1mm}\u{b} \cdot \u{Q} \\ \hline
\rule[3mm]{0mm}{1mm}\lbf_3^t \end{array}\right).
\eeq
 Hence the assumption that the rank of $M(\bb;\bfl_3,\bfl_4)$ over $\F_p$ is $k$ implies that the rank of the matrix (\ref{part_matrix}) is still $k$. Note that the set of solutions to (\ref{eqn6.10}) is, if non-empty, a coset of the homogeneous system
\begin{equation}\label{eqn6.11}
\begin{split}
(\bb\cdot\bQ)\bfw \equiv 0 \modd{p^e} \\
\quad \lbf_3^t\bfw \equiv 0 \modd{p^e}.
\end{split}
\end{equation}
In the case  that (\ref{part_matrix}) has rank $k$ in $\F_p$, this has exactly one solution in $\F_p$ and hence modulo $p^e$, as we claimed. Hence we can estimate the contribution to $N_2(p^e)$ from $N_3(p^e)$ as being at most
\[\#\{ \bb \modd{p^e}:\ (\bb,p)=1,\ p|\det M(\bb;\bfl_3,\bfl_4)\},\]
so certainly at most 
\[\#\{ \bb \modd{p^e}:\ p|\det M(\bb;\bfl_3,\bfl_4)\}= p^{3e-3} \#\{ \bb \modd{p}:\ p|\det M(\bb;\bfl_3,\bfl_4)\}.\]
This is a sufficient bound for $N_3(p^e)$.

Now consider the contribution of $N_4(p^e)$, in which case $M(\bb;\bfl_3,\bfl_4)$ has rank $k-1$ over $\F_p$.
For a fixed $\bb$ with $\rank \left( M(\bb;\bfl_3,\bfl_4)\right) \leq k-1$, the set of solutions to (\ref{eqn6.10}) is again, if non-empty, a coset of the homogeneous system (\ref{eqn6.11}). For a fixed coefficient vector $\bb$ we note that (\ref{eqn6.11}) has at most $Z_0(\bb,p^e)$ solutions in $\bfw$, where we recall the definition of $Z_0(\bb,p^e)$ in (\ref{Z0_dfn}).
We now call upon the following lemma, a modification of Lemma \ref{lemZ}, which is again a consequence of Proposition 7 in \cite{Loxton}:
\begin{lemma}\label{lemma_Z0}
If $p$ is a good prime, then $$Z_0(\bb;p^e)\leq \gcd (F_{\u{Q}}(\bb),p^{e}).$$
\end{lemma}
For its proof we note that if $p$ is a good prime, we may apply Lemma \ref{geolem1} to conclude that $\u{b} \cdot \u{Q}$ has rank at least $k-1$ over $\F_p$. Hence $p$ divides at most one of the eigenvalues of $\u{b}\cdot \u{Q}$ and the determinant of $\u{b}\cdot \u{Q}$ is divisible at most by $p^e$.\par
Now we see that $N_4(p^e)$ is bounded above by
\begin{equation*}
 \sum_{f=1}^e p^f \#\{ \bb \modd{p^e}: (\bb,p)=1,\ p^f|F_{\u{Q}}(\bb),\ \rank \left( M(\bb;\bfl_3,\bfl_4)\right) \leq k-1\}.
\end{equation*}
Trivially, we may re-write this as 
\[ \sum_{f=1}^e p^f p^{3(e-f)} \#\{ \bb \modd{p^f}: (\bb,p)=1,\ p^f|F_{\u{Q}}(\bb),\ \rank \left( M(\bb;\bfl_3,\bfl_4)\right) \leq k-1\}.\]

Now consider a fixed vector $\ba \modd{p}$ with $(\ba,p)=1$ and $p|F_{\u{Q}}(\ba)$. Since $p$ is good, i.e. $p\ndiv 2 \Disc (F_{\u{Q}})$, the point $\ba$ on $F_{\u{Q}}(\ba)=0$ is nonsingular over $\F_p$, so for at least one index $i=1,2,3$ we have $p \ndiv \partial/\partial x_i F_{\u{Q}}(\u{a})$. Supposing without loss of generality that $i=1$, we may fix any choices of $b_2,b_3 \modd{p^{f}}$ that reduce to $a_2,a_3$ modulo $p$ and then apply Hensel's lemma (see e.g. Proposition 5.20 in \cite{Greenberg}) to obtain a unique choice of $b_1 \modd{p^f}$ that reduces to $a_1 \modd{p}$; in conclusion there are at most $p^{2(f-1)}$ choices of $\bb\modd{p^f}$ such that $p^f|F_{\u{Q}}(\bb)$ and $\bb\equiv \ba\modd{p}$. 
Hence we may bound $N_4(p^e)$ by 
\begin{equation*}
\begin{split}
&\sum_{f=1}^e p^f p^{3(e-f)}p^{2(f-1)} \#\{ \u{a} \modd{p}: (\u{a},p)=1,\ p|F_{\u{Q}}(\u{a}),\ \rank \left( M(\u{a};\bfl_3,\bfl_4)\right) \leq k-1\}\\
&\leq ep^{3e-2} \#\{ \u{a} \modd{p}: (\u{a},p)=1,\ \rank \left( M(\u{a};\bfl_3,\bfl_4)\right) \leq k-1\}\\
&\leq ep^{3e-2} \#\{ \u{a} \modd{p}: \rank \left( M(\u{a};\bfl_3,\bfl_4)\right) \leq k-1\}.
\end{split}
\end{equation*}

Assembling the contributions from $N_3(p^e)$ and $N_4(p^e)$, we see that the contribution from $N_2(p^e)$ to $S(\lbf;p^e)$ in (\ref{SN1N2}) is at most 
\[ p^{e(k+3)}p^{-3}f(\lbf;p)+ ep^{e(k+3)}p^{-2}g(\lbf;p),\]
in the notation Lemma \ref{lemS2}. This suffices to complete the proof.

\subsection{Results on average: geometric preliminaries}

We now use the upper bounds on $S(\bfl;p^e)$ from Lemma \ref{lemS1} and Lemma \ref{lemS2} to produce non-trivial upper bounds for $S(\bfl;q)$ on average over $q$ and $\bfl$. Our main goal, which we will prove in Section \ref{sec_Sav}, is the following proposition:

\begin{prop}\label{propSav}
For $k \geq 4$ and any integers $Q,L\geq1$ we have
\begin{equation*}
\sum_{q\leq Q}\sum_{|\bfl|\leq L} |S(\bfl;q)|\ll_\eps Q^{k+3+\eps} L^{2k+\eps}+Q^{k+4+\eps}L^{k},
\end{equation*}
for any $\ep>0$.
\end{prop}

We note that this is similar in appearance to Lemma 5.6 of \cite{HBPierce}, but our approach is different, and the second term in the above result is sharper than the analogous term in \cite{HBPierce}, due to greater control we attain on the dimension of ``bad'' $\lbf$, as discussed below.

In preparation for proving this, we introduce two schemes associated to each given vector $\bfl = (\lbf_1,\lbf_2) \in \Z^{2k}$; we recall that $\lbf_3 = \lbf_1+\lbf_2$, $\lbf_4 = \lbf_1- \lbf_2$. For fixed $\bfl\in \Z^{2k}$ let $\calX_\bfl\subset \A_\Z^3$ be given by
\begin{equation*}
\det M(\u{x};\bfl_3,\bfl_4)=0.
\end{equation*}
We set $X_\lbf= \calX_\lbf \times_\Z\Q$ and $X_{\bfl,p}= \calX_\lbf \times_\Z \Z/p\Z$. Given any $(k+1) \times (k+1)$ matrix $A$, write $\Del_{ij} A$ for the $i,j$-th minor of $A$. For fixed $\lbf \in \Z^{2k}$, we define $\calY_\bfl \subset \A_\Z^3$ by the system of equations
\begin{equation*}
\Del_{ij} M(\u{x};\bfl_3,\bfl_4)=0,\quad 1\leq i,j\leq k+1,
\end{equation*}
and define $Y_\bfl$ and $Y_{\bfl,p}$ similarly as for $\calX_\lbf$.
Note that \emph{a priori} we only know that $\dim X_\bfl \leq 3$, but we do know that $\dim Y_\bfl \leq 2$ for all $\bfl\in \Z^{2k}$, since the minor $\Del_{k+1,k+1} M(\u{x};\bfl_3,\bfl_4)$ equals the form $F_{\u{Q}}(\u{x})$, which is a non-zero form in $\u{x}$.

Now we split the set of all integer vectors $\bfl\in \Z^{2k}$ into good and bad vectors $\bfl$. The idea is that for $\bfl$ that are generic (``good'' in some sense), the variety $X_\bfl$ should be of dimension $2$ and $Y_\bfl$ is expected to be at most one-dimensional. 
\begin{dfn}
We say $\bfl\in \Z^{2k}$ is \emph{bad} if either 
\[ \dim X_\bfl=3 \quad \text{ or } \quad \dim Y_\bfl=2;\]
otherwise we say that $\lbf$ is \emph{good}.
 Let $\calL^b\subset \Z^{2k}$ be the set of bad $\bfl$ and let $\calL^g= \Z^{2k}\setminus \calL^b$ be its complement.
\end{dfn}
\par

We will see that the set of bad $\bfl$ is  sparse, having in some sense half the dimension of the space of all $\lbf$:
\begin{lemma}\label{lemL}
\begin{equation*}
\#\{ |\bfl|\leq L: \bfl\in \calL^b \}\ll L^k.
\end{equation*} 
\end{lemma}

Next we show that for a fixed good $\bfl$, for almost all primes $p$ the varieties $X_{\bfl,p}$ and $Y_{\bfl,p}$ over the finite field $\F_p$ have the same dimension as $X_\bfl$ and $Y_\bfl$.
\begin{lemma}\label{lemreduction}
There are nonempty finite index sets $I$ and $J$ and non-zero polynomials $P_i(\bfl)$ for each $i\in I$ and $R_j(\bfl)$ for each $j\in J$ with the following properties. If $\bfl\in \calL^g$, then there is at least one $i\in I$ and at least one $j\in J$ such that $P_i(\bfl)\neq 0$ and $R_j(\bfl)\neq 0$. Furthermore, one has $\dim X_{\bfl,p} = 2$ if there exists at least one $i \in I$ such that $p\ndiv P_i(\bfl)$, and $\dim Y_{\bfl,p}\leq 1$ if there exists at least one $j \in J$ such that $p\ndiv R_j(\bfl)$.
\end{lemma}

Once we have proved these two lemmas, we will turn in Section \ref{sec_Sav} to proving Proposition \ref{propSav}.
In order to prove Lemma \ref{lemL},  we need the following auxiliary result. 
\begin{lemma}\label{HilfslemmaS}
There are $k$ vectors $\bb^{(i)}\in \overline{\Q}^3$ with $\rank (\bb^{(i)}\cdot\bQ) =k-1$ and such that the nullvectors $\bfe^{(i)}$ of $\bb^{(i)}\cdot\bQ$ are linearly independent for $1 \leq i \leq k$.
\end{lemma}

\begin{proof}
We consider the linear combination $\u{x} \cdot \u{Q}$; Lemma \ref{geolem0} shows that there is an invertible change of variables in  $\u{x}$ that allows us to replace the quadratic forms $Q_i$ by quadratic forms $Q_i'$ such that the intersection of $Q_1'$ and $Q_2'$ is smooth.
By Lemma \ref{geoprop7} the quadratic forms $Q_1'$ and $Q_2'$ can furthermore be simultaneously diagonalized, say $Q_1'= \diag(c_i)$ and $Q_2'= \diag(d_i)$. The points $[c_i:d_i]$, seen as elements in the projective line, are all distinct, because the intersection of $Q_1'$ and $Q_2'$ is smooth. (This is proved explicitly in Proposition 2.1 of \cite{HBPierce}.) 

We will find $k$ vectors $\u{b}^{(i)}$ appropriate to this new system $\u{Q}' = \{Q_1',Q_2',Q_3'\}$ with $Q_1'$ and $Q_2'$ in the diagonal form above, and then by an invertible change of variables we will obtain $k$ vectors suitable to the original system. For the new system of forms, we choose the vectors $\bb^{(i)}$ by setting, for each $1 \leq i \leq k$, the third coordinate $b_3^{(i)}=0$ and the first two coordinates $b_1^{(i)}=d_i$ and $b_2^{(i)}=-c_i$. It is now clear by construction that $\rank (\bb^{(i)}\cdot\bQ') =k-1$. Moreover, for each $1 \leq i \leq k$ the nullvector for $(\bb^{(i)} \cdot \u{Q}')$ is simply the $i$-th unit vector, and thus the $k$ nullvectors are linearly independent. The invertible transformations required to return to the original system of forms preserves this linear independence.
\end{proof}

We may now prove Lemma \ref{lemL}.
\begin{proof}[Proof of Lemma \ref{lemL}]
Assume that $\bfl\in \Z^{2k}$ is bad. Then we need to distinguish two cases. First assume that $\dim X_\lbf=3$, so that
$\det M(\u{x};\bfl_3,\bfl_4)$
is identically zero as a polynomial in $\u{x}$. In particular we then have $\det M(\bb^{(i)};\bfl_3,\bfl_4)=0$ for all $1\leq i\leq k$ and $\u{b}^{(i)}$ given in Lemma \ref{HilfslemmaS}. Suppose $i$ is fixed; if $\bfl_3^t$ is linearly independent of the row vectors of the rank $k-1$ matrix $\bb^{(i)}\cdot \bQ$, then the $(k+1) \times k$ matrix 
\[
\left(\begin{array}{c} \rule[-3mm]{0mm}{1mm}\u{b}^{(i)} \cdot \u{Q} \\ \hline
\rule[3mm]{0mm}{1mm}\lbf_3^t \end{array}\right)
\]
 has rank $k$. Then $\bfl_4$ must be in the span of the column vectors of $\bb^{(i)}\cdot \bQ$, since otherwise $M(\bb^{(i)};\bfl_3,\bfl_4)$ would have rank $k+1$, contrary to the assumption that $\det M(\bb^{(i)};\bfl_3,\bfl_4)=0$. Hence $\lbf_4$ must be orthogonal to the nullspace of $\u{b}^{(i)} \cdot \u{Q}$. Similarly, if $\bfl_4$ is not contained in the span of the column vectors of $\bb^{(i)}\cdot\bQ$, then $\bfl_3^t$ must lie in the span of the  row vectors of $\bb^{(i)}\cdot\bQ$; by symmetry $\lbf_3$ lies in the span of the column vectors of $\u{b}^{(i)} \cdot \u{Q}$, so that $\lbf_3$ is orthogonal to the nullspace of $\u{b}^{(i)} \cdot \u{Q}$. In any case, for each $1\leq i\leq k$ we either have
\begin{equation*}
\bfl_3\cdot\bfe^{(i)}=0,\quad \mbox{ or } \quad \bfl_4\cdot\bfe^{(i)}=0,
\end{equation*}
where $\bfe^{(i)}$ is the nullvector of $\bb^{(i)} \cdot \bQ$ provided by Lemma \ref{HilfslemmaS}. By the linear independence of the $k$ nullvectors $\bfe^{(i)}$, this restricts the set of $\bfl$ such that $\dim X_\bfl=3$ to a
 collection of $2^k$ subspaces of dimension $k$ in $\Z^{2k}$. Hence their contribution to the quantity in Lemma \ref{lemL} is bounded by 
\begin{equation*}
\#\{|\bfl|\leq L: \dim X_\bfl =3\}\ll L^{k}.
\end{equation*}

Next we consider the collection of $\bfl\in \Z^{2k}$ that are labelled bad because $\dim Y_\bfl =2$. Fix one such $\lbf$ and recall that $Y_\bfl$ is defined by
\begin{equation*}
\Del_{ij}M(\u{x};\bfl_3,\bfl_4)=0,\quad 1\leq i,j\leq k+1,
\end{equation*}
and we have $\Del_{k+1,k+1}M(\u{x};\bfl_3,\bfl_4)=F_{\u{Q}}(\u{x})$. Recall that we assume that $\u{Q}$ satisfies Condition \ref{cond2}, so that the curve in the projective plane given by $F_{\u{Q}}(\u{x})=0$ is smooth and hence irreducible. We claim that if $\dim Y_\bfl=2$, then $\u{x}$ solves all the equations 
\begin{equation}\label{minors_vanish}
\Del_{ij} M(\u{x};\bfl_3,\bfl_4)=0, \forall 1\leq i,j\leq k+1,
\end{equation}
as soon as $F_{\u{Q}}(\u{x})=0$. Otherwise the irreducible projective curve $F_{\u{Q}}(\u{x})=0$ would intersect some other curve $\Del_{ij} M(\u{x};\bfl_3,\bfl_4)$ (for $i,j$ not both equal to $k+1$) in a finite number of points in $\P^2$, and then the dimension of $Y_\bfl$ as an affine variety could be at most one.  (We note for this argument that for all $i,j$, the polynomials $\Del_{ij}M(\u{x};\bfl_3,\bfl_4)$ are homogeneous in $\u{x}$.)

In particular, taking the tuples $\bb^{(i)}$ for $1\leq i\leq k$ as in Lemma \ref{HilfslemmaS}, by construction we have $F_{\u{Q}}(\u{b}^{(i)}) = \det (\u{b}^{(i)} \cdot \u{Q})=0$ and hence by the above argument, for each $\u{x} = \u{b}^{(i)}$, all the $k \times k$ minors in (\ref{minors_vanish}) vanish as well. Thus $M(\bb^{(i)};\bfl_3,\bfl_4)$ has rank at most $k-1$ for $1\leq i\leq k$, so certainly it has vanishing determinant for each $\u{b}^{(i)}$. 
 Then we may argue exactly as before to conclude that $\lbf$ must lie in a collection of $2^k$ subspaces of dimension $k$ in $\Z^{2k}$, so that
\begin{equation*}
\#\{ |\bfl|\leq L: \dim Y_\bfl=2\}\ll L^k,
\end{equation*}
which completes the proof of Lemma \ref{lemL}.
\end{proof}

\begin{proof}[Proof of Lemma \ref{lemreduction}]

  We let the $P_i$ be the polynomials in $\bfl$ that are the coefficients of each monomial in the $\u{x}$ of $\det M(\u{x};\bfl_3,\bfl_4)$.
   Then for $\bfl\in \Z^{2k}$, if $P_i(\bfl)=0$ for all $i$, then $\det M(\u{x};\bfl_3,\bfl_4)$ is the zero polynomial in $\u{x}$ over $\Q$ so that $\dim X_{\bfl}=3$ and $\lbf$ is bad.  If $\bfl\in \Z^{2k}$, and the $P_i(\bfl)$ do not all vanish modulo $p$, then $\det M(\u{x};\bfl_3,\bfl_4)$ does not vanish identically as a polynomial in $\u{x}$ with coefficients in $\F_p$, and so $\dim X_{\bfl,p} = 2$.

 To construct the polynomials $R_j$, we first define polynomials $\tilde{R}_j$ that are the polynomials in $\bfl$ that are the coefficients of each monomial in the $\u{x}$ of each of the minors
   \begin{equation*}
\Del_{ij} M(\u{x};\bfl_3,\bfl_4)=0,\quad 1\leq i,j\leq k+1, (i,j)\ne(k+1,k+1).
\end{equation*}
Then we set $R_j (\lbf)= \Disc (F_{\u{Q}}) \tilde{R}_j (\lbf) $.
  Then for $\bfl\in \Z^{2k}$, if $R_j(\bfl)=0$ for all $j$, then $\dim Y_{\bfl}\geq 2$ and $\lbf$ is bad.  If $\bfl\in \Z^{2k}$, and the $R_j(\bfl)$ do not all vanish modulo $p$, then first, since this implies $p\nmid \Disc (F_{\u{Q}})$, we have that $F_{\u{Q}}=0$ defines a irreducible projective plane curve of degree $k$ over $\F_p$.
 In addition, for this given $\bfl$, at least one of the other minors
    $\Del_{ij} M(\u{x};\bfl_3,\bfl_4)$ for $(i,j)\ne(k+1,k+1)$ as a polynomial in the $\u{x}$ is non-zero and degree $k-1$.  Thus the common intersection of $\Del_{ij} M(\u{x};\bfl_3,\bfl_4)=0$ from this non-zero degree $k-1$ minor and the irreducible degree $k$ equation $F_{\u{Q}}=0$ is dimension $1$, and so $\dim Y_{\bfl,p} \leq 1$.

 Note if the $P_i$ were all $0$, then all $\bfl \in \Z^{2k}$ would be bad, which contradicts Lemma~\ref{lemL}, and similarly if all the $R_j$ were $0$.

\end{proof}

\subsection{Proof of Proposition \ref{propSav}}\label{sec_Sav}
We now return to the proof of the main result we will apply to exponential sums on the minor arcs, Proposition \ref{propSav}.
First we recall from Lemma \ref{lemS1} the upper bound
\begin{equation*}
S(\bfl;q)\ll_\eps q^{k+3+\eps}. 
\end{equation*}
With this combined with Lemma \ref{lemL}, we  quickly obtain an upper bound for the average restricted to bad $\lbf$:
\begin{equation}\label{eqnSav0}
\sum_{q\leq Q} \sum_{\substack{|\bfl|\leq L\\ \bfl\in \calL^b}} |S(\bfl;q)| \ll Q^{k+4+\eps} \# \{ |\bfl|\leq L: \bfl\in \calL^b\} \ll_\eps Q^{k+4+\eps}L^{k}.
\end{equation}
This is sufficient for the conclusion of the proposition.\par
Next we note that Lemma \ref{lemS1} in connection with Lemma \ref{lemS2} implies that there exists a constant $C$ such that
\begin{equation}\label{eqnSav1}
|S(\bfl;p^e)|\leq C (ke+1) p^{e(k+3)} (p^{-1}+p^{-3}f(\bfl;p)+p^{-2}g(\bfl;p)),
\end{equation}
where we recall the functions $f(\lbf;p)$ and $g(\lbf;p)$ from Lemma \ref{lemS2}. Note that since finitely many primes are bad, we have been able to choose $C$ sufficiently large so that this bound holds uniformly for all primes $p$.
Moreover we can trivially bound $(ke+1)\leq (e+1)^k$ so that when we combine the factors $(ke+1)$ resulting from all the prime powers $p^e|q$, this will contribute at most $d(q)^k \ll q^\ep$.

As we will now want to combine conditions for several primes simultaneously, we re-write the functions $f$ and $g$ in the equivalent form, for any square-free integer $q$, as
\begin{eqnarray*}
f(\bfl;q)&=& \#\{ \bb \modd{q}:\ q| \det M(\bb;\bfl_3,\bfl_4)\}, \\
g(\bfl;q)&=&\#\{\bb \modd{q}:\ q| \Del_{ij} M(\bb;\bfl_3,\bfl_4),\ \forall {1\leq i,j\leq k+1}\}.
\end{eqnarray*}
Note that the functions $f(\bfl;q)$ and $g(\bfl;q)$ are multiplicative in $q$.

Given an integer $q$, we write $\kap (q)= \prod_{p|q} p$.  Then we have for any integer $q$ the bound
\begin{eqnarray}
|S(\bfl;q)|&\leq & C^{\omega(q)} d(q)^{k} q^{k+3}  \prod_{p|q} (p^{-1}+p^{-3}f(\bfl;p)+p^{-2}g(\bfl;p)) \nonumber\\ 
&\leq &C^{\omega(q)} d(q)^{k} q^{k+3} \sideset{}{'}\sum_{q_1q_2q_3=q} \frac{1}{\kap(q_1)}\frac{f(\bfl;\kap(q_2))}{\kap(q_2)^3}\frac{g(\bfl;\kap(q_3))}{\kap(q_3)^2},\label{eqnSav2}
\end{eqnarray}
in which the sum is over all factorizations of $q$ into three factors that are pairwise relatively prime.
We now turn to the contribution of the good $\bfl$ and fix some vector $\bfl\in \calL^g$. Let the polynomials $P_i(\bfl)$, $i\in I$ and $R_j(\bfl)$, $j\in J$ be given as in Lemma \ref{lemreduction}. Assume that for our vector $\bfl$ under consideration, we have specified indices $i$ and $j$ (guaranteed by Lemma \ref{lemreduction}) so that the  values $P_i(\bfl)$ and $R_j(\bfl)$ are nonzero. Using equation (\ref{eqnSav2}) we bound the sum
\begin{align*}
\sum_{q\leq Q} |S(\bfl;q)| &\ll_\eps Q^{k+3+\eps} \sum_{q\leq Q} \sideset{}{'}\sum_{q_1q_2q_3=q} \frac{1}{\kap(q_1)}\frac{f(\bfl;\kap(q_2))}{\kap(q_2)^3}\frac{g(\bfl;\kap(q_3))}{\kap(q_3)^2}\\
&\ll_\eps Q^{k+3+\eps}  \Sig_1(\bfl)\Sig_2(\bfl)\Sig_3(\bfl),
\end{align*}
with sums of the form
\begin{eqnarray*}
\Sig_1(\bfl)&=& \sum_{q\leq Q} \frac{1}{\kap(q)},\\
\Sig_2(\bfl)&=&\sum_{q\leq Q} \frac{f(\bfl;\kap(q))}{\kap(q)^3},\\
\Sig_3(\bfl)&=&\sum_{q\leq Q} \frac{g(\bfl;\kap(q))}{\kap(q)^2}.
\end{eqnarray*}
We now separately estimate the sums $\Sig_i(\bfl)$ for $1\leq i\leq 3$. We start with $\Sig_1(\bfl)$, which we treat via Rankin's trick, as in Section 5.5 of \cite{HBPierce}; for convenience we briefly recall the method. Note that
\begin{equation*}
\Sig_1(\bfl) \leq Q^{\eps} \sum_{q=1}^\infty \frac{1}{q^\eps \kap (q)} \leq Q^\eps \prod_{p} (1+p^{-1-\eps}+p^{-1-2\eps}+\cdots ).
\end{equation*}
We set 
\beq\label{c_dfn}
c=c(\eps)= 1+2^{-\eps}+4^{-\eps}+\cdots
\eeq
 and rewrite the upper bound as
\begin{equation*}
\Sig_1(\bfl)\leq Q^\eps \prod_p (1+p^{-1-\eps}c) \leq Q^\eps \zet (1+\eps)^c\ll_\eps Q^\eps.
\end{equation*}
We next bound the factor $\Sig_2(\bfl)$. Recall that we have assumed that $P_i(\bfl)\neq 0$. First we note the trivial bound $f(\bfl;p)\leq p^3$ which holds for any prime $p$ by definition of the counting function $f(\bfl;p)$. If $p\nmid P_i(\bfl)$, then by Lemma \ref{lemreduction} we know that $\dim X_{\bfl,p}=2$. Hence Lemma \ref{geolem4} implies for such primes $p$ the upper bound
\begin{equation*}
f(\bfl;p)\ll p^2,
\end{equation*}
with an implied constant only depending on $k$. We conclude by multiplicativity of the function $f(\bfl;q)$ that 
\begin{equation*}
f(\bfl;\kap(q))\ll \left(\prod_{p|(P_i(\bfl),\kap(q))} p\right) C^{\omega(q)} (\kap(q))^2,
\end{equation*}
for some positive constant $C$. We insert this into the definition of $\Sig_2(\bfl)$ and obtain the upper bound
\begin{align*}
\Sig_2(\bfl)&\ll_\eps Q^\eps \sum_{q\leq Q} \left(\prod_{p|(P_i(\bfl),\kap(q))} p\right)\frac{1}{\kap(q)} \\
&\ll_\eps Q^{2\eps} \sum_{q=1}^\infty \frac{1}{q^\eps} \frac{1}{\kap(q)} \prod_{p|(P_i(\bfl),\kap(q))} p.
\end{align*}
Recalling the definition of $c=c(\eps)$ in (\ref{c_dfn}), we similarly proceed in bounding $\Sig_2(\bfl)$ by
\begin{align*}
\Sig_2(\bfl)&\ll_\eps Q^{2\eps} \prod_{p|P_i(\bfl)} (1+p^{-\eps}+p^{-2\eps}+\cdots) \prod_{p\nmid P_i(\bfl)} (1+p^{-1-\eps}+p^{-1-2\eps}+\cdots ) \\
&\ll_\eps Q^{2\eps} c^{w(P_i(\bfl))}\zet (1+\eps)^c \ll_\eps Q^{2\eps} L^\eps.
\end{align*}

The treatment of $\Sig_3(\bfl)$ is very similar. We recall that the $k\times k$ minor $M_{ij}(\bb;\bfl_3,\bfl_4)$ with $(i,j) = (k+1,k+1)$ is just the function $F_{\u{Q}}(\bb)$, hence we get as a first upper bound
\begin{equation*}
g(\bfl;p)\leq \#\{ \bb\modd{p}:\ p|F_{\u{Q}}(\bb)\} \leq c p^2,
\end{equation*}
by Lemma \ref{lemF}, where $c$ depends only on the finitely many bad primes and hence may now be regarded as a universal constant. Furthermore, if $p\nmid R_j(\bfl)$, then the variety $Y_{\bfl,p}$ is at most one dimensional by Lemma \ref{lemreduction}. Hence again Lemma \ref{geolem4} implies that $g(\bfl;p)\ll p$ with an implied constant only depending on $k$. Together these estimates imply that
\begin{equation*}
g(\bfl;\kap(q))\ll C^{\omega(q)} \kap(q) \prod_{p|(R_j(\bfl),\kap(q))} p,
\end{equation*}
for some positive constant $C$. Hence we can bound the sum $\Sig_3(\bfl)$ by
\begin{equation*}
\Sig_3(\bfl)\ll_\eps Q^\eps \sum_{q\leq Q} \frac{1}{\kap(q)} \left( \prod_{p|(R_j(\bfl),\kap(q))}p\right) \ll_\eps Q^{2\eps} L^\eps,
\end{equation*}
by a similar argument to that given for $\Sig_2(\lbf)$.\par
We now assemble the estimates for the sums $\Sig_i(\bfl)$ in order to bound the contribution of the good vectors $\bfl\in \calL^g$. We easily obtain that
\begin{align*}
\sum_{\substack{|\bfl|\leq L\\ \bfl\in \calL^g}} \sum_{q\leq Q} |S(\bfl;q)| \ll_\eps Q^{k+3+\eps} \sum_{\substack{|\bfl|\leq L\\ \bfl\in \calL^g}} \Sig_1(\bfl)\Sig_2(\bfl)\Sig_3(\bfl) \ll_\eps Q^{k+3+\eps} L^{2k+\eps}.
\end{align*}
In combination with the estimate (\ref{eqnSav0}), this completes the proof of  Proposition \ref{propSav}.

\section{The Singular Series}
We define the singular series to be
\begin{equation}\label{sing_ser_dfn}
\grS(\bn)= \sum_{q=1}^\infty \frac{1}{q^k} T(\bn;q),
\end{equation}
recalling the definition of $T(\u{n};q)$ in (\ref{Tnq_dfn}).
We first show that $\grS (\bn)$ is absolutely convergent for $\bn$ with $H_{\u{Q}}(\bn)\neq 0$. 

\begin{proposition}\label{prop_sing_ser_conv}
Assume that $k>6$ and $H_{\u{Q}}(\bn)\neq 0$. Then $\grS(\bn)$ is absolutely convergent; more precisely we have 
\begin{equation*}
\sum_{q\geq R} \frac{1}{q^k} |T(\bn;q)|\ll_\eps |\bn|^\eps R^{-\nu},
\end{equation*}
for any $0 < \nu < k/2-3$, with an implied constant only depending on $\eps$ and the system of quadratic forms $\bQ$.
\end{proposition}

\begin{proof}
Let $R\geq 1$ and $\nu$ be as assumed. Then the tail of the sum admits the bound
\begin{equation*}
\sum_{q\geq R} \frac{1}{q^k} |T(\bn;q)| \leq R^{-\nu} \sum_{q=1}^\infty q^{\nu-k} |T(\bn;q)|.
\end{equation*}
Note that the right hand side factorizes over all finite primes, and hence we have
\begin{equation}\label{T_prod}
\sum_{q\geq R} \frac{1}{q^k} |T(\bn;q)| \leq R^{-\nu}\prod_p \psi_p,
\end{equation}
with 
\begin{equation*}
\psi_p=1+\sum_{e=1}^\infty p^{e(\nu-k)}|T(\bn;p^e)|.
\end{equation*}
If $p$ is a good prime of Type I, then we use Proposition \ref{lemT4} to note that $T(\u{n};p^e)=0$ for $e \geq2$, so that we may bound the local factor $\psi_p$ by
\begin{equation}\label{psi_T1}
\psi_p=1+O(p^{\nu-k}p^{(k+3)/2}).
\end{equation}
If $p$ is a good prime of Type II or a bad prime, then we use the weaker bounds  provided by Proposition \ref{lemT1} to obtain
\begin{eqnarray*}
\psi_p&=&1+O\left(\sum_{e=1}^\infty p^{e(\nu-k)}p^{e(3+k/2)}\right) \\
&=& 1+O\left(\sum_{e=1}^\infty p^{e(\nu +3-k/2)}\right)= 1+O\left(p^{\nu+3-k/2}\right).
\end{eqnarray*}
Here we have used the assumption that $\nu+3 - k/2<0$, and we note that the implied constant may be taken to be independent of $p$ (since it depends only on the finitely many bad primes).\par

Thus we see that the product in (\ref{T_prod}) is absolutely convergent as soon as $k/2>\nu+3$ (in which case $\nu +3/2-k/2 <-1$, so that (\ref{psi_T1}) gives sufficient decay for the infinitely many Type I primes). More explicitly, we require an upper bound for the product, in terms of $\u{n}$. Note that if $\nu +3/2-k/2 <-1$,  the product of $\psi_p$ over all primes of Type I is bounded by $O(1)$ independently of $\bn$. The product of $\psi_p$ over the finitely many bad primes is $O(1)$, independently of $\u{n}$, since the notion of being ``bad'' is entirely independent of $\u{n}$. On the other hand, the Type II primes must divide $H_{\u{Q}}(\u{n})$, so that their contribution is 
\begin{equation*}
\prod_{p \; \text{Type II}} \psi_p \leq C^{\omega(H_{\u{Q}}(\bn))} \ll_\eps |\bn|_\infty^\eps,
\end{equation*}
for some constant $C$. 
\end{proof}

Having proved convergence, we are next interested in lower bounds for $\grS(\bn)$:

\begin{proposition}
Let $k>6$ and suppose $H_{\u{Q}}(\bn)\neq 0$, and let $\alptil$ be given as in Lemma \ref{lem5.6}. Assume additionally that for every prime $p$ the system of equations $\bQ (\bfx)=\bn$ is soluble in $\Z_p$. Then there exists $p_0>0$ such that
\begin{equation*}
\grS(\bn)\gg_\eps |\bn|_\infty^{-\eps} \prod_{p\leq p_0} |H_{\u{Q}}(\bn)|_p^{2\alptil(k-3)},
\end{equation*}
for any $\eps >0$.
\end{proposition}

\begin{proof}
By the multiplicativity of $T(\bn;q)$ we can factorize the singular series as
\begin{equation*}
\grS(\bn)= \prod_p \sig_p(\bn),
\end{equation*}
with local densities 
\[ \sig_p(\bn)=1 + \sum_{e=1}^\infty p^{-ek}T(\bn;p^e).
\]
For any good prime $p$ (including Type II), applying Proposition \ref{lemT4} for $e=1$ and Proposition \ref{lemT1} for $e \geq 2$ implies that for $k >6$,
\begin{equation*}
\sig_p(\bn)=1+O\left(p^{-k/2+2}+\sum_{e=2}^\infty p^{e(3-k/2)} \right)= 1+O\left(p^{-k/2+2}+p^{6-k}\right).
\end{equation*}
Hence as long as $k>6$, there is some positive real number $B$ with the property that for all good primes,
\begin{equation*}
\sig_p(\bn)\geq 1-Bp^{-1}.
\end{equation*}
In addition, for all good primes of Type I we apply Proposition \ref{lemT4} again to note that
\begin{equation*}
\sig_p(\bn)=1+O\left(p^{\frac{3-k}{2}}\right) \geq 1-Bp^{\frac{3-k}{2}},
\end{equation*}
if we enlarge $B$ suitably. We may then fix some threshold $p_0=p_0(B)$ such that there exists some fixed $0< \del < 1/2$ so that for $p \geq p_0$ of Type I,
\begin{equation*}
\sig_p(\bn)\geq 1-p^{\frac{3-k+\del}{2}}.
\end{equation*}
For Type II primes $p \geq p_0$ we have $\sig_p(\u{n}) \geq B_0^{-1}$ for a fixed constant $B_0>1$ dependent only on $B$.
Then if $k > 6$, the contribution of all good primes $p\geq p_0$ to the product $\prod_p \sig_p(\bn)$ is bounded below by
\begin{equation*}
 \prod_{p \; \text{Type I}} \left(1-p^{\frac{3-k+\del}{2}}\right) \prod_{p | H_{\u{Q}}(\u{n})} B_0^{-1} \gg (d(H_{\u{Q}}(\u{n})))^{-1} \gg_\eps |\bn|_\infty^{-\eps},
\end{equation*}
where we have again used the fact that if $p$ is of Type II then $p| H_{\u{Q}}(\u{n})$.
We may also assume that $p_0$ is chosen sufficiently large such that all bad primes are at most of size $p_0$. For $p\leq p_0$ we use the bound from Proposition \ref{prop5.2} and obtain
\begin{equation*}
\sig_p(\bn)\geq \ome_p |H_{\u{Q}}(\bn)|_p^{2\alptil (k-3)}.
\end{equation*}
In total we can bound the singular series below by
\begin{equation*}
\grS(\bn)\gg_\eps |\bn|_\infty^{-\eps}\prod_{p\leq p_0} |H_{\u{Q}}(\bn)|_p^{2\alptil (k-3)}.
\end{equation*}
\end{proof}

\section{The Singular integral}
The goal of this section is to prove the results in Theorem \ref{thm_SJ} regarding the singular integral $J_w(\bmu)$. Note that the fact that $J_w(\bmu)\ll 1$  uniformly in $\bmu$, already follows from Proposition \ref{prop_sing_int_conv}. Hence it remains to prove the explicit lower bound (\ref{JggH}) for $J_w(\bmu)$. For this we recall Lemma 9.1 of \cite{HBPierce}, which immediately generalizes to the situation of three forms (we omit the proof, which may be found in \cite{HBPierce}):

\begin{lemma}\label{lem8.1}
There is a real constant $\Lamm$ depending only on the system of quadratic forms $\bQ$, such that the following holds. If $1/2 \leq |\bmu|\leq 1$  is such that $\bQ(\bfx)=\bmu$ has a real solution $\bfx\in \R^k$,  then there is also a real solution $\bfx_0$ to this system with $|\bfx_0|\leq \Lamm$. 
\end{lemma}

If we are now given a solution to the system $\bQ(\bfx)=\bmu$ with $\bfx$ bounded as in Lemma \ref{lem8.1}, then we can deduce a lower bound on the maximal determinant of a $3\times 3$ minor $\Del_{i j \ell}(\xbf)$ of the Jacobian $J_{\u{Q}}(\xbf)$ at the point $\bfx$; this is the real analogue to Lemma \ref{lem5.6}.

\begin{lemma}\label{lem8.2}
Let $\bfx\in \R^k$ be such that $|\bfx|\leq \Lam$ for some real $\Lam >0$ and assume that $\bQ(\bfx)=\bmu$. Let $\alptil$ be given as in Lemma \ref{lem5.6}. Then we have
\begin{equation*}
\max_{i,j,\ell}|\Del_{ij\ell}(\bfx)|\gg |H_{\u{Q}}(\bmu)|^{\alptil},
\end{equation*}
with an implied constant only depending on $\Lam$ and the system $\u{Q}$.
\end{lemma}

\begin{proof}
The proof is essentially the same as that of Lemma \ref{lem5.6}. Equation (\ref{eqn5.6}) shows that
\begin{equation*}
H_{\u{Q}}(\bmu)^\alptil= \sum_{i=1}^3 g_i(\bfx,\bmu) (Q_i(\bfx)-\mu_i) + \sum_{i,j,\ell} h_{ij\ell}(\bfx,\bmu) \Del_{ij\ell}(\bfx).
\end{equation*}
If $|\bfx|\leq \Lam$ then one also has $\bmu\ll_\Lam 1$, so that $h_{ij\ell}(\xbf,\u{\mu}) \ll_{\Lambda} 1$. Using the fact that by assumption the terms $Q_i(\bfx) - \mu_i$ all vanish, we may conclude that
\begin{equation*}
|H_{\u{Q}}(\bmu)|^{\alptil} \ll \sum_{i,j,\ell} |\Del_{ij\ell}(\bfx)|,
\end{equation*}
and the lemma follows.
\end{proof}

To prove (\ref{JggH}), assume now that we are given some $\bmu$ with $1/2\leq |\bmu|\leq 1$ for which the system $\bQ(\bfx)=\bmu$ has a real solution. By Lemma \ref{lem8.1} we may assume that the size of this solution is bounded by $\Lamm$, i.e. there is some $|\bfx_0|\leq \Lamm$ with $\bQ(\bfx_0)=\bmu$. 
Furthermore, we may assume that $M:=\max_{i,j,\ell} |\Del_{ij\ell}(\bfx_0)|> 0$ since otherwise (\ref{JggH}) is trivial, by Lemma \ref{lem8.2}.\par
We will next explicitly construct a region in which we may apply the implicit function theorem in order to invert the singular integral.

\begin{lemma}\label{lem8.3} 
Given a real vector $\bfx_0\in \R^k$ with $|\bfx_0|\leq \Lamm$ and $\bQ(\bfx_0)=\bmu$, assume that $M:=\max_{i,j,\ell}|\Del_{ij\ell}(\bfx_0)|>0$. Let $f: \R^k\rightarrow \R^k$ be given by 
\[ f : \bfx\mapsto (Q_1(\bfx), Q_2(\bfx), Q_3(\bfx),x_4,\ldots, x_k).\] 

Then there are open subsets $V$ and $W$ with $\bfx_0\in V\subset \R^k$ and $f(\bfx_0)\in W\subset \R^k$ such that $f$ is a bijection from $V$ to $W$ and has a differentiable inverse $f^{-1}$ on $W$, with $\det((f^{-1})') \geq M^{-1}$ on $W$. Furthermore, one may choose 
\[W= \{\bfy\in \R^k: |f(\bfx_0)-\bfy|\leq AM^2\}\]
 for some positive constant $A$, which only depends on $\Lam$ and the system of quadratic forms $\bQ$.
\end{lemma}

We momentarily defer the proof of Lemma \ref{lem8.3}, and first complete the proof of (\ref{JggH}) in Theorem \ref{thm_SJ}. 
By Proposition \ref{prop15.4} we can express the singular integral $J_w(\bmu)$ as the limit
\begin{equation*}
J_w(\bmu)=\lim_{\eps\rightarrow 0}\eps^{-3}\int_{\max_i|Q_i(\bfx)-\mu_i|\leq \eps} w(\bfx) \prod_{i=1}^3\left(1-\frac{|Q_i(\bfx)-\mu_i|}{\eps}\right)\d\bfx .
\end{equation*}
We define the auxiliary integral
\begin{equation*}
 J_w^{(\ep)}(\bmu)=\int_{\max_i|Q_i(\bfx)-\mu_i|<\frac{1}{2}\eps}w(\bfx)\d\bfx;
\end{equation*}
then 
\beq\label{JJ_bd}
J_w(\bmu) \geq \frac{1}{(2\ep)^3} J^{(\ep)}_w(\u{\mu}),
\eeq
 so that it suffices to bound the auxiliary integral from below.

In Theorem \ref{thm_SJ} we choose the constant $C$ governing the support of $w$ so that $C>2\Lam$; then we can choose the set $V$ in Lemma \ref{lem8.3} sufficiently small so that $w(\bfx)\gg 1$ on $V$. Let $\chi_{\ep/2}$ be the characteristic function of the interval $(-\frac{1}{2}\eps,\frac{1}{2}\eps)$. Then we may bound our auxiliary integral below by
\begin{equation*}
J_w^{(\ep)}(\bmu)
	\geq \int_{V} w(\bfx)\prod_{i=1}^3 \chi_{\ep/2}(Q_i(\bfx)-\mu_i)\d\bfx 
	\gg \int_{V} \prod_{i=1}^3 \chi_{\ep/2}(Q_i(\bfx)-\mu_i)\d\bfx.
\end{equation*}
We apply the change of variables $f:V\rightarrow W$ described in Lemma \ref{lem8.3}, and obtain
\begin{equation*}
\begin{split}
J_w^{(\ep)}(\bmu)&\gg \int_W |\det ((f^{-1})')| \prod_{i=1}^3 \chi_{\ep/2}(y_i-\mu_i)\d\bfy \\ &\gg  \int_W M^{-1} \prod_{i=1}^3 \chi_{\ep/2}(y_i-\mu_i)\d\bfy.
\end{split}
\end{equation*}
We recall the choice of $W$ in Lemma \ref{lem8.3} and obtain for $\eps$ sufficiently small a lower bound of the form
\begin{equation*}
J_w^{(\ep)}(\bmu)\gg M^{-1}\eps^3 M^{2(k-3)}.
\end{equation*}
We now deduce from (\ref{JJ_bd}) and Lemma \ref{lem8.2} that 
\begin{equation*}
J_w(\bmu)\gg M^{2(k-3)-1}\gg |H_{\u{Q}}(\bmu)|^{\alptil(2(k-3)-1)},
\end{equation*}
where $\alptil$ is given as in Lemma \ref{lem5.6}. This completes the proof of Theorem \ref{thm_SJ}.

\begin{proof}[Proof of Lemma \ref{lem8.3}]
For simplicity of notation we assume that 
\begin{equation*}
|\Del_{123}(\bfx_0)|= M :=   \max_{i,j,\ell}|\Del_{ij\ell}(\bfx_0)| ;
\end{equation*}
recall that we assume $M>0$.
In this case we let $f: \R^k\rightarrow \R^k$ be given by $\bfx\mapsto (\bQ(\bfx),x_4,\ldots, x_k)$. We will explicitly find a small open neighbourhood of $\xbf_0$ in which the implicit function theorem is applicable, following the proof of Theorem 2.11 in \cite{Spivak}.
Let $U$ be the closed ball given by 
\begin{equation*}
U=\{|\bfx-\bfx_0|\leq C_1M\},
\end{equation*}
for a sufficiently small constant $C_1$. (All constants $C_i$ to follow only depend on $\Lamm$ as in Lemma \ref{lem8.1} and the system of quadratic forms $\bQ$.) Let $C_3>0$ be fixed. We claim that for $C_1$ sufficiently small, there is a positive real number $C_2$ such that the following three properties hold for all $\bfx,\bfx_1,\bfx_2\in U$.
\begin{enumerate}
\item $|f(\bfx)-f(\bfx_0)|\geq C_2M|\bfx-\bfx_0|$,
\item $|\frac{\partial}{\partial x_j}f_i(\bfx)-\frac{\partial}{\partial x_j}f_i(\bfx_0)|\leq C_3$, for all $1\leq i,j\leq k$,
\item $|\Del_{123}(\bfx)|\geq \frac{1}{2}M$.
\end{enumerate}
Since $M\ll 1$ the second observation is clear if $C_1$ is chosen sufficiently small. Similarly, since $M=\Del_{123}(\bfx_0)$, observation (3) follows from the boundedness of all derivatives of $\Del_{123}(\xbf)$, for $|\bfx-\bfx_0|\leq C_1 M$ and $C_1$ sufficiently small. 

To prove observation (1) we note that 
\begin{equation}\label{eqn8.5}
f(\bfx)-f(\bfx_0)=J(f)(\bfx_0)(\bfx-\bfx_0)+O(|\bfx-\bfx_0|^2),
\end{equation}
where $J(f)$ is the $k \times k$ Jacobian matrix of $f$, so in particular $\det J(f)(\xbf) \gg M$ for all $\xbf \in U$, by (3).
In general, if $A$ is an invertible $k\times k$ matrix with bounded entries, and $\bfh\in \R^k$ then we have $|\bfh|\leq \Vert A^{-1}\Vert |A\bfh|$, and hence $|A\bfh|\gg (\det A)|\bfh|$. We apply this to the Jacobian $J(f)(\bfx_0)$ and obtain $|J(f)(\bfx_0)(\bfx-\bfx_0)|\gg M |\bfx-\bfx_0|$. Together with equation (\ref{eqn8.5}) we obtain 
\begin{equation*}
|f(\bfx)-f(\bfx_0)|\geq C_2 M|\bfx-\bfx_0|,
\end{equation*}
for $C_2$ and $C_1$ sufficiently small, proving (1).
In fact, this same argument shows that for any for $\bfx_1,\bfx_2\in U$, as long as $C_1$ is sufficiently small,
\begin{equation}\label{eqn8.6}
|f(\bfx_1)-f(\bfx_2)|\geq C_2 M|\bfx_1-\bfx_2|.
\end{equation}
Note that for $\xbf$ on the boundary of $U$ we have
\begin{equation*}
|f(\bfx)-f(\bfx_0)|\geq C_2M|\bfx-\bfx_0|= C_1C_2M^2.
\end{equation*}
We set $d:=C_1C_2M^2$ and define 
$$W=\{\bfy\in\R^k: |\bfy-f(\bfx_0)|<\frac{1}{2}d\}.
$$
In particular, if $\bfy\in W$ and $\bfx$ is on the boundary of $U$ then we have
\begin{equation}\label{eqn8.7}
|\bfy-f(\bfx_0)|< \frac{1}{2}d\leq |\bfy-f(\bfx)|.
\end{equation}
We claim that for all $\bfy\in W$ there is exactly one $\bfx$ in the interior of $U$ such that $f(\bfx)=\bfy$. For this we consider the function $g:U\rightarrow \R$ given by
\begin{equation*}
g(\bfx)=|\bfy-f(\bfx)|^2=\sum_{i=1}^k (y_i-f_i(\bfx))^2.
\end{equation*}
Since this is a continuous function on $U$, it takes its minimum on $U$ and equation (\ref{eqn8.7}) shows that the minimum does not occur on the boundary of $U$. Hence the minimum is achieved in the interior and this implies (see for example Theorem 2.6 in \cite{Spivak}) that there is some $\bfx\in{\rm int} (U)$ such that $\frac{\partial}{\partial x_j}g(\bfx)=0$ for all $1 \leq j \leq k$. That is, for each $1 \leq j \leq k$,
\begin{equation*}
\sum_{i=1}^k 2(y_i-f_i(\bfx))\frac{\partial}{\partial x_j}f_i(\bfx)=0.
\end{equation*}
Since the Jacobian matrix $J(f)(\bfx)$ is invertible on $U$ we get $y_i=f_i(\bfx)$ for all $1\leq i\leq k$. The uniqueness of $\bfx$ follows from equation (\ref{eqn8.6}).\par
We now set $V$ to be the intersection of $\mathrm{int}(U)$ and the pre-image of $W$ under $f$. Then the function $f:V\rightarrow W$ has an inverse $f^{-1}:W\rightarrow V$ which is continuous by (\ref{eqn8.6}). Similarly as in Theorem 2.11 in \cite{Spivak} one sees that $f^{-1}$ is differentiable and $(f^{-1})'(\bfy)=[f'(f^{-1}(\bfy))]^{-1}$ for all $\bfy\in W$, 
 so that $|\det ((f^{-1})')| \geq M^{-1}$ on $W$.
This completes the proof of Lemma \ref{lem8.3}.

\end{proof}

\section{The Circle method}

\subsection{Division into major and minor arcs}
Let $\Del>0$ be a small real number to be chosen later. For any integers $1 \leq a_1,a_2,a_3 \leq q \leq B^\Del$ we define the box 
\[ I(\u{a};q)  = \prod_{j=1}^3 \left[\frac{a_j}{q} - \frac{B^\Del}{B^2}, \frac{a_j}{q}
 + \frac{B^\Del}{B^2}\right]  \subseteq [0,1]^3.\]
If the tuples $(\u{a};q)$ and $(\u{a}';q')$ are not identical, the corresponding boxes are disjoint, provided that $B$ is large enough and
\[  \qquad \Del< 2/3.\]
Indeed if $(\u{a};q) \neq (\u{a}';q')$ 
then for $j=1,2,$ or 3 we have
\[ \left| \frac{a_j}{q} - \frac{a_j'}{q'} \right| \geq \frac{1}{qq'}
\geq \frac{1}{B^{2\Del}} > 2\frac{B^\Del}{B^2},\] 
for all $B$ sufficiently large.  We now define the major arcs to be \[ \Mf(\Delta) = \Union_{1 \leq q \leq B^\Del} \Union_{\substack{1 \leq a_1,
   a_2, a_3 \leq q\\ (\u{a},q)=1}} I(\u{a};q)\] and we take the minor arcs to be the complement of the major arcs in $[0,1]^3$,
\[ \mf(\Delta) = [0,1]^3 \setminus \Mf (\Delta).\]

In order to obtain an alternative characterization of the minor arcs we will employ 
a 3-dimensional Dirichlet approximation with a parameter $S \geq 1$ (see
Hardy and Wright \cite[Theorem 200]{HardyWright}, for example).
Given any tuple $\u{\al} \in [0,1]^3,$ there exist $1 \leq q \leq S$ and $1 \leq
a_1, a_2, a_3 \leq q$ with $(\u{a},q)=1$ such that
 \[ \left| \al_j - \frac{a_j}{q} \right| \leq \frac{1}{ qS^{1/3}}, \quad \text{for $j=1,2,3$}.\]
 Our ultimate strategy is to bound the integral over the minor arcs from above by a sum of integrals over the collection of relevant arcs provided by the Dirichlet approximation; the fact that we are not controlling for overlap of the intervals is acceptable, since we only seek an upper bound.
 
Given $\u{\al}=( \al_1,\al_2, \al_3)$ and approximations $\al_j = a_j/q + \theta_j$ of the above type for $j=1,2,3$, then if
$\u{\al} \in \mf(\Delta)$, at least one of the four inequalities \beq\label{maj_cond}
q \leq B^\Delta, \qquad |\theta_1| \leq B^{-2 + \Delta}, \qquad
|\theta_2| \leq B^{-2 + \Delta}, \qquad
|\theta_3| \leq B^{-2 + \Delta},
\eeq
must fail to hold. 
We may now further decompose the minor arcs as 
\beq\label{mindecomp}
\mf(\Delta)\subseteq \mf_0(\Delta)\cup\mf_1(\Delta)\cup\mf_2(\Delta)\cup\mf_3(\Delta),
\eeq
where we set
\[\mf_0(\Delta)=\bigcup_{B^{\Delta}\leq q\leq S}
\bigcup_{\substack{1 \leq a_1,a_2,a_3 \leq q\\ (\u{a},q)=1}}\mf_0(\Delta,q,\u{a})\]
and for $j=1,2,3$, we set
\[\mf_j(\Delta)=\bigcup_{1\leq q\leq S}
\bigcup_{\substack{1 \leq a_1,a_2,a_3 \leq q\\ (\u{a},q)=1}}\mf_j(\Delta,q,\u{a}),\]
with
\[\mf_0(\Delta,q,\u{a})=
\{\u{\al}: |\al_j-a_j/q|\leq (qS^{1/3})^{-1}\mbox{ for }j=1,2,3\},\]
and for $j=1,2,3$,
\[\mf_j(\Delta,q,\u{a})=\mf_0(\Delta,q,\u{a})
\cap
\{\u{\al}:\, |\al_j-a_j/q|\geq B^{-2+\Delta}\}.\]
  Given $q$ and $\u{a}$ we will write
\[\al_j=a_j/q+\theta_j \qquad \text{$j=1,2,3$}.\]
For our application we shall choose \[S=B^{3/2},\]
which is essentially optimal.

\subsection{The major arcs: the singular integral and singular series}
We recall that our ultimate interest is in the representation function
\[ \Rcal_B(\u{n}) =  \sum_{\bstack{\xbf \in \Z^k}{\u{Q}(\xbf) = \u{n}}} w_B(\xbf),\]
where $w_B(\xbf) := w (\xbf/B)$.
Upon defining 
\[S(\u{\al}) = \sum_{\x \in \Z^k} e(\u{\al} \cdot \u{Q}(\xbf) )w_B(x),\]
we may express
\[ \Rcal_B(\u{n}) = \int_{[0,1]^3} S(\u{\al}) e(-\u{\al} \cdot
\u{n}) d\u{\al} .\] 
We now record the standard argument used to isolate the main term on the major arcs:

\begin{prop}\label{prop_major_int}
Suppose that $k >6$ is an integer and $|\u{n}|\ll B^2$ with $H_{\u{Q}}(\u{n})\not=0$. Then for any fixed positive $\Delta \leq 1/10$
we have
\[ \int_{\Mf(\Delta)} S(\u{\al}) e(-\u{\al} \cdot \u{n})d\u{\al}= \Sf(\u{n}) J_w(B^{-2}\u{n}) B^{k-6} + E(\u{n}),\] with
\beq\label{eqn_E}
E(\u{n}) \ll B^{k-6 - \Del/4},
\eeq
where $\Sf(\u{n})$ and $J_w(\u{\mu})$ are given by
(\ref{sing_ser_dfn}) and (\ref{sing_int_dfn}) respectively.
\end{prop}

The first step is the following result:
\begin{lemma}\label{xxx}
If $\u{\al}$ belongs to a major arc $I(\u{a};q)$ and
$\u{\theta} = \u{\al} - \u{a}/q$, then \beq\label{S_2terms}
S(\u{\al}) = q^{-k} B^k S_q(\u{a}) I_w(B^2 \u{\theta})+O(B^{k-1+2\Delta}), \eeq
with $I_w(\u{\phi})$ given by (\ref{I_dfn}).
\end{lemma}
The proof, which we omit, is standard. Lemma 5.1 in \cite{Bir62} is a classical source;  Lemma 7.2 of \cite{HBPierce} proves the analogous result for 2 quadratic forms, and that proof is easily modified to the case of three forms.

Next we recall that by definition of the major arcs,
\[
\int_{\Mf(\Delta)} S(\u{\al}) e(-\u{\al} \cdot \u{n}) d\u{\al}
= \sum_{1 \leq q \leq B^\Del} \sum_{\substack{1 \leq a_1,a_2, a_3 \leq
           q\\ (\u{a},q)=1}} \int_{I(\u{a};q)} S(\u{\al}) e(-\u{\al} \cdot \u{n}) d\u{\al}.
\]
We note that
measure of the total collection of major arcs is
 \[\ll B^{\Delta} \cdot B^{3\Delta} \cdot (B^{-2 + \Delta})^3
\ll B^{-6+ 7\Delta}.\] Thus Lemma \ref{xxx} immediately implies: 
\[
\int_{\Mf(\Delta)} S(\u{\al}) e(-\u{\al} \cdot \u{n}) d\u{\al} = 
B^{k-6}J_w(B^{-2}\u{n};B^\Del)\sum_{q\leq B^\Delta}q^{-k}T(\u{n};q)
+ O (B^{k-7 +9\Delta})
\]
with $J_w(\u{\mu};R)$ given by (\ref{sing_int_dfn_R}).

Finally, we apply the results of Propositions \ref{prop_sing_int_conv}
and \ref{prop_sing_ser_conv} (with integral $k >6$, so we may certainly take $\nu =1/3 < k/2-3$) to the truncated singular integral and
singular series in order to pass to the limit on the right hand
side. We obtain
\begin{eqnarray*}
\int_{\Mf(\Delta)} S(\u{\al}) e(-\u{\al} \cdot \u{n})d\u{\al}
&=& \Sf(\u{n}) J_w(B^{-2}\u{n}) B^{k-6}\\ &&\hspace{5mm}\mbox{}+O(B^{k-6 - \Del/3+\varepsilon})+O(B^{k-7 + 9\Del}),
\end{eqnarray*}
as long as $H_{\u{Q}}(\u{n})\not=0$.
Proposition \ref{prop_major_int} then follows, upon restricting $\Del \leq 1/10.$

\section{Proof of main theorem}
We now formulate the mean-square argument that allows us to prove our foundational result, Theorem \ref{thm_meansquare}. The basic structure of the mean-square method is standard. The key input to the method is Proposition \ref{prop_minor_arcs}, which controls the minor arcs: this proposition is the culmination of all the estimates for exponential sums and oscillatory integrals we have established thus far, and this proposition controls the number of variables ultimately required in our main theorems.

\subsection{The mean square argument}
 Proposition \ref{prop_major_int} establishes that for $N = B^2$,
\begin{multline*}
  \sum_{\bstack{|\u{n}|_\infty \leq N}{H_{\u{Q}}(\u{n}) \neq 0}} |\Rcal_B(\u{n} )- \Sf(\u{n}) J_w(B^{-2}\u{n}) B^{k-6}|^2 
  \\
 	=   \sum_{\bstack{|\u{n}|_\infty \leq N}{H_{\u{Q}}(\u{n}) \neq 0}}\left| \int_{\mathfrak{m}(\Del)} S(\u{\al}) e(-\u{\al} \cdot \u{n} )d\u{\al}\right|^2 + O(B^{2k-6 -\Del/2}),
	\end{multline*}
for any positive $\Del \leq 1/10$.
Temporarily set $f(\u{\al}) = S(\u{\al})\chi_{\mathfrak{m}(\Del)}(\u{\al})$; then by positivity, the contribution of the minor arcs may be bounded above by
\begin{eqnarray*}
\ll \sum_{\u{n} \in \Z^3}\left| \int_{\mathfrak{m}(\Del)} S(\u{\al}) e(-\u{\al} \cdot \u{n} )d\u{\al}\right|^2 
 	& = &  \sum_{\u{n} \in \Z^3}\left| \int_{\R^3} f(\u{\al}) e(-\u{\al} \cdot \u{n} )d\u{\al}\right|^2  \\
	& = &\int_{\R^3} |f(\u{\al})|^2 d \u{\al} \\
	& = &  \int_{\mathfrak{m}(\Del)} |S(\u{\al})|^2 d \u{\al} ,
\end{eqnarray*}
where in the second equality we have applied Parseval's identity to $f$.
Thus in total 
\beq\label{RS_final}
  \sum_{\bstack{|\u{n}|_\infty \leq N}{H_{\u{Q}}(\u{n}) \neq 0}} |\Rcal_B(\u{n} )- \Sf(\u{n}) J_w(B^{-2}\u{n}) B^{k-6}|^2 
\ll \int_{\mathfrak{m}(\Del)} |S(\u{\al})|^2 d\u{\al} + O(B^{2k-6 -\Del/2}).
\eeq
The keystone of the paper is the following bound for the minor arcs integral:
\begin{prop}\label{prop_minor_arcs}
For any $k  \geq 10$, any $\ep>0$, and any $\Del\in(0,1/13)$, we have
\beq\label{min_int}
\int_{\mf(\Delta)} |S(\u{\al})|^2d\u{\al} \ll
B^{2k-6- 6\Del + \ep}.
\eeq
\end{prop}
Combining this with (\ref{RS_final}), with the choice $\Delta=1/14$, then establishes Theorem \ref{thm_meansquare}.

\subsection{Control of the minor arcs}
We now turn to proving Proposition \ref{prop_minor_arcs}. 
We will bound the minor arcs integral (\ref{min_int}) from above by using a collection of integrals over dyadic pieces, defined for any real $L \geq 1$ and $\u{\phi} \in \R^3_{\geq 0}$ by
\beq\label{Sig_dfn}
\Sig (L, \u{\phi}) = \sum_{L \leq q < 2L}
\sum_{\substack{1 \leq a_1,a_2, a_3 \leq q\\ (\u{a},q)=1}}
\int_{\{\u{\phi} \}} |S(\u{a}/q + \u{\theta})|^2
d\u{\theta}.
\eeq

The key upper bound for these dyadic pieces $\Sig(L,\u{\phi})$, which will immediately imply Proposition \ref{prop_minor_arcs}, is as follows:
\begin{prop}\label{prop_sig_bd}
For any $k \geq 10$, any $\ep>0$, and any $\Del\in(0,1/13)$, 
\beq\label{Sig_R}
 \Sig (L, \phi)  \ll B^{2k-6-6\Del+\ep}
 \eeq
for $L\ll S=B^{3/2}$ and $\u{\phi}$ with $\phi^* = \max_i\{\phi_i\} \ll L^{-1}S^{-1/3}$,
unless 
\beq\label{maj_cond'}
L \leq \frac{1}{2}B^\Delta, \qquad \text{and} \quad \phi^* \leq \frac{1}{2}B^{-2 + \Delta}.
\eeq
 \end{prop}
 We note that for $L \leq q < 2L$ and $\u{\theta}$ in the dyadic range of the integral, the two conditions in (\ref{maj_cond'}) encode the four conditions (\ref{maj_cond}) that  identify the major arcs.

Before proving Proposition \ref{prop_sig_bd} we show how it implies
Proposition \ref{prop_minor_arcs}.  We begin by applying the trivial bound 
$|S(\u{\al})|\ll B^k$ for the portion of the integral (\ref{min_int}) over the region in which at least one of $|\theta_1|,|\theta_2|,|\theta_3|$ is $\ll B^{-2k}$. Recalling that  $S = B^{3/2}$, the total measure  contributed by  such points is at most
\[\ll \sum_{q\leq S}q^3(qS^{1/3})^{-2}B^{-2k}\ll S^{4/3}B^{-2k}=B^{2-2k}.\]
We thus see  that the integral of $|S(\al)|^2 \ll B^{2k}$ over this set contributes $O(B^2)$, which is
acceptable for Proposition \ref{prop_minor_arcs}.

There remain tuples $\u{\al},q,\u{a}$ for which $\u{\theta} = \u{a}/q - \u{\al}$ satisfies
\[1\leq q\leq S,\;\;\;B^{-2k}\leq  \theta_1,\theta_2,\theta_3 \leq (q S^{1/3})^{-1}\]
and at least one of
the inequalities
\[q\geq B^{\Delta},\;\;\; |\u{\theta}|_\infty\geq B^{2-\Delta}\]
holds. We may cover these with $O((\log B)^4)$ dyadic
contributions of the form $\Sig(L,\u{\phi})$.
We thus have
\[ \int_{\mathfrak{m}(\Del)} |S(\u{\al})|^2 d\u{\al} \ll
B^2+(\log B)^4 \sup \Sig(L,\u{\phi}),\] 
where the supremum is taken over all dyadic parameters $L,\u{\phi}$ with $0<L< S$
and $B^{-k} \leq \phi_i \leq (LS^{1/3})^{-1}$ for $i=1,2,3$ such that not
both conditions (\ref{maj_cond'}) hold. 
An application of
Proposition \ref{prop_sig_bd} then clearly suffices for
Proposition \ref{prop_minor_arcs}.

We now turn to the proof of Proposition \ref{prop_sig_bd}.
We begin with an application of Poisson summation, which gives the following:
\begin{lemma}
\beq\label{Sig_SI}
\Sig (L, \u{\phi}) = B^{2k}\sum_{L \leq q < 2L}
\frac{1}{q^{2k}}	\sum_{\lbf \in \Z^{2k}}S(\lbf;q)
\Ical_{\{\u{\phi}\}} (\lbf;q),
\eeq
where
\[S(\lbf;q) = \sum_{\substack{\u{a} \modd{q}\\ (\u{a},q)=1}}
\sum_{\substack{\rbf_1 \modd{q}\\ \rbf_2 \modd{q}}} e_q( \u{a} \cdot
\u{Q} (\rbf_1)- \u{a} \cdot \u{Q}(\rbf_2))e_q(\rbf \cdot \lbf)\]
and 
\[\Ical_{\{\u{\phi}\}} (\lbf;q) = \int_{\{\u{\phi}\}}
J_w(B^2\u{\theta},B\lbf/q)d\u{\theta}=
B^{-6}\int_{\{B^2\u{\phi}\}} J_w(\u{\nu},B\lbf/q)d\u{\nu}.\]
\end{lemma}
To see this, we need only expand the integrand in (\ref{Sig_dfn}) and write $\xbf_j = \lbf_j q +
\rbf_j$, where $\lbf_j \in \Z^{k}$ and $\rbf_j \in (\Z/q\Z)^{k}$ for
$j=1,2$ in order to obtain
\[|S(\u{a}/q + \u{\theta})|^2=
\sum_{\substack{\rbf_1 \modd{q}\\ \rbf_2\modd{q}}} e_q(\u{a}\cdot\u{Q}(\rbf_1)-\u{a}\cdot\u{Q}(\rbf_2))
\Sigma(\u{\theta},\rbf_1,\rbf_2,q),\]
where we have temporarily set
\[\Sigma(\u{\theta},\rbf_1,\rbf_2,q)=
\sum_{\lbf_1, \lbf_2 \in \Z^k} e(\u{\theta} \cdot
\u{Q} (\lbf_1q + \rbf_1) - \u{\theta} \cdot \u{Q} (\lbf_2q +
\rbf_2)) w_B(\lbf_1q+\rbf_1)w_B(\lbf_2q+\rbf_2).\]
We use Poisson summation to transform this into the sum
\[\Sigma(\u{\theta},\rbf_1,\rbf_2,q)=\left(\frac{B}{q} \right)^{2k}
\sum_{\lbf \in \Z^{2k}}e_q(\rbf \cdot \lbf)J_w(B^2\u{\theta},B\lbf/q),\]
where $J_w(\u{\nu},\lambf)$ is given by
\[ J_w(\u{\nu},\lambf)  = \int_{\R^{2k}} e( \unu \cdot \u{Q}(\ubf_1)
- \unu \cdot \u{Q}(\ubf_2))  w(\ubf_1) w(\ubf_2) e(-\ubf \cdot \lambf)
d\ubf.\] 
We may conclude that 
\[\sum_{\substack{\u{a} \modd{q}\\ (\u{a},q)=1}}
|S(\u{a}/q+ \u{\theta})|^2=B^{2k}
\sum_{\lbf \in \Z^{2k}}q^{-2k}S(\lbf;q)J_w(B^2\u{\theta},B\lbf/q),\]
and the lemma follows after integrating over the region $\{ \u{\phi} \}$.

In order to bound $\Sig(L,\u{\phi})$, we will now bound $I_{\{\u{\phi}\}}(\lbf;q)$ individually, and then apply bounds we have proved for averages of $S(\lbf;q)$. 
We first note that Lemmas \ref{lem15.0} and \ref{lem15.2} immediately imply the following: 
\begin{lemma}
For any $\lbf \in \Z^{2k}$, 
\[ \int_{\{B^2\u{\phi}\}} J_w(\u{\nu},B\lbf/q) d \u{\nu} 
	\ll \min ((B^2 \phi^*)^3, (B^2\phi^*)^{3-k}) \log B.\]
If $|\lbf | \gg LB\phi^*$, then for any $M \geq 1$,
\[ \int_{\{B^2\u{\phi}\}} J_w(\u{\nu},B\lbf/q) d \u{\nu} 
	\ll_M (B^2 \phi^*)^3 \left( \frac{B|\lbf|}{q} \right)^{-M}.\]
\end{lemma}

Once $L,\{ \u{\phi} \}$ are fixed, we define the threshold $L_0 = L B^{-1+\ep} (1 + B^2 \phi^*)$. Upon recalling the trivial bound $S(\lbf;q) \ll q^{2k+3}$, we may take care of the tail of the sum in (\ref{Sig_SI}) with the estimate 
\begin{eqnarray*}
B^{2k}\sum_{L \leq q < 2L}
\frac{1}{q^{2k}} \left|	\sum_{|\lbf| \geq L_0}S(\lbf;q)
\Ical_{\{ \u{\phi} \}} (\lbf;q) \right|
	&\ll& B^{2k-6} L^4 (B^2 \phi^*)^3 \left( \frac{L}{B} \right)^M \sum_{|\lbf| \geq L_0} |\lbf|^{-M} \\
		&\ll& B^{2k} L^4 \left( \frac{L}{B} \right)^M L_0^{2k+1-M} \\
	&\ll& B^{2k} L^4 \left( \frac{L}{B} \right)^M (LB^{-1+\ep})^{2k+1-M} \\
	&\ll &B^{2k} L^4 (LB^{-1+\ep})^{2k+1} B^{-\ep M},
	\end{eqnarray*}
valid for $M >2k+1$, and with implied constant depending on $\ep$ and $M$.
 Upon choosing $M$ sufficiently large with respect to $\ep$, this tail contributes $O(1)$ to $\Sig (L, \{ \u{\phi} \})$. 

On the other hand, we recall from Lemma \ref{lemS1} that for any $q$ and $\lbf$, we have the bound $S(\lbf;q) \ll q^{k+3 +\ep}$. Thus in particular the contribution of the $\lbf=\mathbf{0}$ term to $\Sig(L, \u{\phi} )$ is at most 
\begin{eqnarray*}
 && \ll B^{2k-6} (B^2 \phi^*)^3 \min (1, (B^2 \phi^*)^{-k}) (\log B) \sum_{L \leq q \leq 2L} q^{-2k} q^{k+3+\ep} \\
 && 
  \ll B^{2k-6} (B^2 \phi^*)^3 \min (1, (B^2 \phi^*)^{-k}) (\log B) L^{-k+4+\ep}.
  \end{eqnarray*}
If $L$ is such that $L \leq \frac{1}{2} B^\Del$,  we must have $\phi^* \geq  \frac{1}{2} B^{-2+\Del}$; then we may conclude that the above contribution is 
  \[ \ll B^{2k-6} (B^2 \phi^*)^{3-k} (\log B) L^{-k+4+\ep} \ll B^{2k-6} B^{\Del(3-k)} (\log B),
  \]
  which is $o(B^{2k-6})$ for $k >3$ and suffices for Proposition \ref{prop_sig_bd} if $k \geq 9$.
Alternatively, if $L \geq \frac{1}{2} B^\Del$ then the contribution is 
  \[ \ll B^{2k-6} L^{-k+4+\ep} \ll B^{2k-6} B^{\Del(4-k +\ep)},
  \]
which is $o(B^{2k-6})$ for $k>4$ and suffices for Proposition \ref{prop_sig_bd} if $k \geq 10$.
 
 Finally, we turn to the most important range, with $1 \leq |\lbf| \leq L_0$, which contributes to $\Sig(L, \u{\phi} )$ at most
 \[ B^{2k-6} \min( (B^2 \phi^*)^3, (B^2 \phi^*)^{3-k}) (\log B) \sum_{L \leq q < 2L} q^{-2k} \sum_{1 \leq |\lbf| \leq L_0} |S(\lbf;q)|.\]
 Upon applying the average bounds of Proposition \ref{propSav}, this is bounded above by
 \[ B^{2k-6+5\ep} \min( (B^2 \phi^*)^3, (B^2 \phi^*)^{3-k}) L^{-2k} \{ L^{k+3} L_0^{2k} + L^{k+4}L_0^k \},\]
 where we have used the fact that $LL_0 \ll B^4 $ so that $(LL_0)^\ep \ll B^{4\ep}$.
 In the case where $B^2 \phi^* \geq 1$, we have $L_0 \ll LB^{1+\ep} \phi^*$, and we compute this contribution as 
 \begin{eqnarray*}
 && \ll B^{2k-6+3k\ep} (B^2 \phi^*)^{3-k} L^{-2k} \{ L^{k+3} (LB\phi^*)^{2k} + L^{k+4}(LB\phi^*)^k \} \\
  && \ll B^{2k-6+3k\ep} \{ L^{k+3}B^6( \phi^*)^{k+3}+ L^{4}B^{6-k}(\phi^*)^3 \}.
 \end{eqnarray*}
 Now we recall that by the Diophantine approximation, we only consider 
 \[ \phi^* \ll L^{-1} S^{-1/3} \ll L^{-1} B^{-1/2},\]
 so that (upon recalling $L \leq S = B^{3/2}$) the contribution above is at most
   \begin{eqnarray*}
&& \ll B^{2k-6+3k\ep} \{ B^{-(k-9)/2}+ B^{6-k} \} \ll B^{2k-6+3k\ep} \cdot B^{-(k-9)/2},
 \end{eqnarray*}
which is sufficient for Proposition \ref{prop_sig_bd} if $k\geq 10$ and $6\Del < 1/2$.

On the other hand if $B^2 \phi^* \leq 1$, we have $L_0 \ll LB^{-1+\ep}$, so that (upon recalling $L \leq S = B^{3/2}$)
 the contribution is 
 \begin{eqnarray*}
 && \ll  B^{2k-6+3k\ep}  (B^2 \phi^*)^3 L^{-2k} \{ L^{k+3} (LB^{-1})^{2k} + L^{k+4}(LB^{-1})^k \} \\
 && \ll B^{2k-6+3k\ep}  \{ L^{k+3} B^{-2k} + L^{4}B^{-k} \} \\
  && \ll B^{2k-6+3k\ep}  \{ B^{-(k-9)/2} + B^{6-k} \} \ll B^{2k-6+3k\ep} \cdot B^{-(k-9)/2}, 
 \end{eqnarray*}
 which is again sufficient for Proposition \ref{prop_sig_bd} if $k\geq 10$ and $6\Del < 1/2$. This completes the proof of Proposition \ref{prop_sig_bd} and hence of Theorem \ref{thm_meansquare}.

\section{Proofs of the remaining main theorems}\label{sec_final_proofs}
We have proved Theorems \ref{thm_meansquare} and \ref{thm_SJ}. 
To derive Theorem \ref{thm_asymp}, we need only define $\Ecal_\varpi(N)$ to be the set of $\u{n} \in [-N,N]^3$ such that either $H_{\u{Q}}(\u{n})=0$ or the difference 
\[ \Rcal_B(\u{n}) - J_w(B^{-2}\u{n})\mathfrak{S}(\u{n}) B^{k-6}\]
fails to be $O(B^{k-6 - \varpi}).$ We deduce from Lemma \ref{lemPoly2} that at most $O(N^2)$ tuples $\u{n}$ satisfy the first criterion, and see by Theorem \ref{thm_meansquare} that at most $O(N^{3 - (1/56 - \varpi)})$ tuples $\u{n}$ satisfy the second criterion. This provides the upper bound on the size of the exceptional set, as claimed in Theorem \ref{thm_asymp}.

We next note that the derivation of Theorem \ref{thm_primes} from Theorem \ref{thm_exceptions} is completely analogous to the derivation of Theorem 1.4 from Theorem 1.3 in \cite{HBPierce}, thus we omit the proof here. 

We turn to proving Theorem \ref{thm_exceptions}. For this we first note that by a dyadic division it is enough to prove Theorem \ref{thm_exceptions} for tuples $\u{n} \in \Z^3$ with $N/2\leq |\u{n}|_\infty < N$. Furthermore we may choose our weight function $w$ in such a way that the assumptions in Theorem \ref{thm_SJ} are satisfied and set $B=N^{1/2}$ for the rest of this section.\par
Now we divide the tuples $\bn$ with $N/2\leq  |\u{n}|_\infty< N$ which are counted by $E(N)$ into three cases, according to  real parameters $a>0,b>0$ to be chosen later. In Case I, we assume that 
\begin{equation*}
J_w (B^{-2}\bn)\geq B^{-b},\mbox{ and } \grS(\bn)\geq B^{-a}.
\end{equation*}
In Case II we assume that $H_{\u{Q}}(\bn)=0$ or $J_w (B^{-2}\bn)<B^{-b}$. The remaining Case III may be characterized by the property that $H_{\u{Q}}(\bn)\neq 0$ and $\grS (\bn)<B^{-a}$. For $1 \leq i \leq 3$ we let $E_i(N)$ denote the number of $N/2\leq |\u{n}|_\infty < N$ counted in Case I, II
and III respectively.
We will prove:
\begin{lemma}\label{lemma_E123}
With the parameters $a,b$ as above, and with the real constants $c = 1/28$ and $\al, \be$ as given in Theorem \ref{thm_SJ},
\begin{equation*}
\begin{split}
E_1(N)&\ll N^2+N^{3+a+b-c/2},\\
E_2(N)&\ll N^2+N^{3-b/(2\bet \deg(H_{\u{Q}}))},\\
E_3(N)&\ll N^{3-1/(\deg(H_{\u{Q}}))+\eps}+N^{3-a/(2\alp \deg (H_{\u{Q}}))+\eps}.
\end{split}
\end{equation*}
\end{lemma}

Assuming Lemma \ref{lemma_E123} (whose proof we defer for the moment), we may quickly derive Theorem \ref{thm_exceptions}. We see that in order for $E_1(N)$ to be sufficiently small to prove Theorem \ref{thm_exceptions} with some $\varpi>0$, we must at least choose $a,b$ so that $a+b < c/2$. 
We will also assume we choose $a$ so that $a\leq 2\alp$, at which point $E_3(N) \ll N^{3-a/(2\alp \deg (H_{\u{Q}}))+\eps}.$
 We also note that $3-1/\deg(H_{\u{Q}})\geq 2$ as soon as $\deg (H_{\u{Q}})\geq 1$, so that we expect the upper bound for $E_3(N)$ to be at least as big as those for $E_1(N)$ and $E_2(N)$. 
In balancing the three terms, we aim to choose $a,b$ so that 
\begin{equation}\label{3terms}
 \frac{c}{2}-a-b   = \frac{b}{2\bet\deg (H_{\u{Q}})}=\frac{a}{2\alp \deg(H_{\u{Q}})} .
\end{equation}
Equating the second and third terms implies that $b=\frac{\bet}{\alp}a$; subsequently equating the first and third terms gives
\begin{equation}\label{a_dfn}
a= \frac{c}{2} \left(\frac{1}{2\alp \deg (H_{\u{Q}})}+1+\frac{\bet}{\alp}\right)^{-1}.
\end{equation}
We note that with this choice, we have 
\[ a+b = \frac{c}{2} \left(\frac{1}{2\alp \deg (H_{\u{Q}})}+1+\frac{\bet}{\alp}\right)^{-1} \left(1+\frac{\be}{\al} \right) < \frac{c}{2},\]
as required. In addition, upon defining $a$ by (\ref{a_dfn}), we certainly have $a < 1/56 < 1 \leq 2\alp$ (recalling that $\al \geq 1$ from the remark following Theorem \ref{thm_SJ}).
We may conclude that after choosing $a,b$ as above, Theorem \ref{thm_exceptions} holds for 
\[ \varpi = \frac{c}{2} - a-b  = \frac{c}{2} \cdot \frac{1}{2(\al+\be) \deg(H_{\u{Q}}) +1}.\]
 In particular, to make this independent of the particular system $\u{Q}$, we note that the degree of $H_{\u{Q}}(\u{x})$ as a polynomial in $\u{x}$ is always bounded above by the degree $\deg(H)$ of $H(\Abf,\u{x})$ as a polynomial in all the variables. Thus (also recalling the allowable choices of $\al,\be$ from Theorem \ref{thm_SJ}), we may make the smaller, but universal, choice of 
 \[  \varpi = \frac{c}{2} \cdot \frac{1}{2(\al+\be) \deg(H) +1} = \frac{1}{2^4.7.3^{3+k}(4k-13) \deg(H)+1}.\]

\subsection{Proof of Lemma \ref{lemma_E123}}
We recall that by Lemma \ref{lemPoly2},
\begin{equation}\label{H_bound_set}
\#\{\max_i|n_i|\leq N: H_{\u{Q}}(\bn)=0\}\ll N^2,
\end{equation}
since $H_{\u{Q}}(\bn)$ is non-zero polynomial in $\bn$.
We recall that $E(N)$ counts those $\u{n}$ such that $\u{Q}(\xbf) = \u{n}$ has no integral solution, so $\Rcal_B(\u{n})=0$.
 Hence if $E_1'(N)$ is the number of those $\bn$  in Case I with the additional property that $H_{\u{Q}}(\bn)\neq 0$, then we see from Theorem \ref{thm_meansquare} that
\begin{align*}
E_1'(N) B^{-2a}B^{-2b} B^{2k-12} &\leq \sum_{\substack{|\u{n}|_\infty \leq N,\ \text{case I}\\ H_{\u{Q}}(\bn)\neq 0}} |\Rcal_B(\bn)-J_w (B^{-2}\bn)\grS(\bn)B^{k-6}|^2 \\ &\ll B^{2k-6-c},
\end{align*}
with $c=1/28$. Hence we obtain the upper bound
\begin{equation*}
E_1(N)\ll N^2+B^{6+2a+2b-c}\ll N^2+N^{3+a+b-c/2}.
\end{equation*}

Next we consider the contribution of $E_2(N)$. Since we already have a satisfactory bound for the $\bn\in \Z^3$ with $H_{\u{Q}}(\bn)=0$ by (\ref{H_bound_set}), we may assume that $H_{\u{Q}}(\bn)\neq 0$ and $J_w (B^{-2}\bn)<B^{-b}$. Note that for such $\bn$ the assumptions in the second part of Theorem \ref{thm_SJ} are satisfied by $B^{-2}\u{n}$ since $1/2 \leq |B^{-2}\u{n}|_\infty\leq 1$ and $\bQ(\bfx)=B^{-2}\bn$ has a real solution in $\bfx$ as soon as $\bQ(\bfx)=\bn$ is soluble in $\R^k$. Hence we see that for such $\bn$ we have the lower bound 
\begin{equation*}
J_w (B^{-2}\bn)\gg |H_{\u{Q}}(B^{-2}\bn)|^\beta.
\end{equation*}
We recall that the polynomial $H_{\u{Q}}(\bn)$ is homogeneous. Hence we can estimate the contribution of $E_2(N)$ by
\begin{equation}\label{eqnE2}
\begin{split}
E_2(N) &\ll N^2+ \#\{\bn\in\Z^3: |\u{n}|_\infty\leq N,\ |H_{\u{Q}}(B^{-2}\bn)|^\bet\ll B^{-b}\} \\
&\ll N^2 +  \#\{ \bn\in \Z^3: |\u{n}|_\infty \leq N,\ |H_{\u{Q}}(\bn)|\leq C_1 N^{\deg(H_{\u{Q}})- b/(2\bet)}\},
\end{split}
\end{equation}
for some positive constant $C_1$. 
Then the sublevel set estimate of Lemma \ref{lemPoly2} delivers the bound
\begin{equation*}
E_2(N)\ll N^2 + N^{3- b/(2\bet\deg(H_{\u{Q}}))}.
\end{equation*}

We next turn to the contribution of $E_3(N)$. We recall that $E_3(N)$ counts all tuples $\bn\in \Z^3$ with $N/2\leq |\u{n}|_\infty \leq N$ such that $H_{\u{Q}}(\bn)\neq 0$ and $\grS(\bn)<B^{-a}$ and such that $\bQ(\bfx)=\bn$ is locally soluble everywhere. Hence the first part of Theorem \ref{thm_SJ} implies that for $\bn$ counted by $E_3(N)$,
\begin{equation*}
|\grS(\bn)|\gg_\eps |\bn |_\infty^{-\eps} \prod_{p\leq p_0} |H_{\u{Q}}(\bn)|_p^\alp,
\end{equation*}
with some parameter $p_0$ only depending on the system of quadratic forms $\bQ$. Hence we need to estimate the cardinality of the set
\begin{equation*}
\#\{ \bn\in \Z^3: N/2 \leq | \bn |_\infty\leq N,\ H_{\u{Q}}(\bn)\neq 0, \prod_{p\leq p_0} |H_{\u{Q}}(\bn)|_p^\alp \ll B^{-a} N^\eps\}.
\end{equation*}
We can rewrite this as
\begin{multline}\label{eqnE3a}
E_3(N)\ll_\eps \#\bigg\{ \bn\in\Z^3:N/2\leq |\bn|_\infty\leq N,\ H_{\u{Q}}(\bn)\neq 0, \bigg. \\
	\bigg. \prod_{p\leq p_0}|H_{\u{Q}}(\bn)|_p\ll B^{-a/\alp+\eps}\bigg\}.
\end{multline}
If  $\bn$ is contained in this set, then there is a divisor $q$ of $H_{\u{Q}}(\bn)$ with
\begin{equation*}
q\gg B^{a/\alp-\eps},
\end{equation*}
which is composed entirely of primes $p \leq p_0$. Since we also assume that $H_{\u{Q}}(\bn)\neq 0$, we see we also have an upper bound $q\ll N^{\deg (H_{\u{Q}})}$. In particular, for any prime power $p^f|q$, this implies that $f\ll \log N$. Thus in particular there exists some constant $C_2(N)$ with
 $C_2(N) \ll \log N$ such that $q |(\prod_{p \leq p_0}p)^{C_2(N)}$; this shows that the number of possibilities for $q$ is $O(C_2(N)^{p_0})=O((\log N)^{p_0})$. Recalling that $p_0$ depends only on $\u{Q}$, which is an allowable dependence in our implied constants, there is thus some positive constant $C_3=C_3(\eps,\u{Q})$ such that there is a set $\calS$ of at most $C_3N^\eps$ divisors $q$ such that 
\begin{enumerate}
\item for any $\bn$ counted in (\ref{eqnE3a}) the polynomial $H_{\u{Q}}(\bn)$ is divisible by one of these divisors $q$,
\item $q$ is entirely composed of primes $\leq p_0$,
\item $B^{a/\alp-\eps} \ll q\ll N^{\deg(H_{\u{Q}})},$
\item and furthermore we may assume $q\leq p_0 N$.
\end{enumerate} 
All but the last property have been proved; for the last point, if we had some $\bn$ such that $H_{\u{Q}}(\bn)$ is divisible by $q>p_0N$ composed only of primes less than or equal to $p_0$, then $q/p$ also divides $H_{\u{Q}}(\bn)$ for any $p\leq p_0$. We can hence divide $q$ by primes $p \leq p_0$ until we have found some divisor $q'$ of $H_{\u{Q}}(\bn)$ with $q'\leq p_0N$, as claimed. We conclude that we can estimate the contribution of $E_3(N)$ by
\begin{equation}\label{eqnE3b}
E_3(N)\ll N^\eps \max_{q\in \calS} \#\{ \bn\in \Z^3: |\u{n}|_\infty \leq N: q| H_{\u{Q}}(\bn)\}.
\end{equation}
We fix one $q\in \calS$ and consider
\begin{equation}\label{eqnE3c}
\#  \{1\leq |\u{n}|_\infty \leq N: q|H_{\u{Q}}(\bn)\} \ll \left(\frac{N}{q}+1\right)^3r(q),
\end{equation}
with the counting function $r(q)$ defined by
\begin{equation*}
r(q):= \#\{ \u{n} \modd{q}: q|H_{\u{Q}}(\bn)\}.
\end{equation*}
We note that $r(q)$ is a multiplicative function in $q$, so that if $q=\prod p^{f_p}$ is some decomposition into prime powers, then $r(q)=\prod r(p^{f_p})$. In particular it is enough to bound $r(p^f)$ for prime powers. For this we need the $p$-adic version of Lemma \ref{lemPoly2} given in Lemma \ref{lemPoly4}, which provides the upper bound
\begin{equation*}
r(p^f)\ll p^{3f-f/(\deg(H_{\u{Q}}))},
\end{equation*}
for any prime $p$. In combination with the estimates in (\ref{eqnE3b}) and (\ref{eqnE3c}) this leads to the bound
\begin{equation*}
\begin{split}
E_3(N)
&\ll N^\eps\max_{q\in \calS}\left(\frac{N^3}{q^3}+1\right) C^{\omega(q)}q^{3-1/(\deg(H_{\u{Q}}))} \\ 
&\ll N^{2 \eps}\max_{q\in \calS} \{ N^3q^{-1/(\deg(H_{\u{Q}}))} + q^{3-1/(\deg(H_{\u{Q}}))}\}\\ 
&\ll N^{3\eps} \{ N^3 B^{-a/(\alp\deg(H_{\u{Q}}))} +N^{3-1/(\deg(H_{\u{Q}}))}\}.
\end{split}
\end{equation*}
We note that here we have used the bounds in (3) and (4) for $q\in \calS$. Hence we have
\begin{equation*}
E_3(N)\ll N^{3-a/(2\alp \deg(H_{\u{Q}}))+\eps} + N^{3-1/(\deg(H_{\u{Q}}))+\eps},
\end{equation*}
as claimed.

 \section{Appendix}\label{sec_appendix}
 
 \subsection{Forms satisfying Condition 2}
We provide here a specific example of a system of three quadratic forms $Q_1,Q_2,Q_3$ in $10$ variables such that the corresponding form $F_{\u{Q}}(x,y,z) = \det (xQ_1 + yQ_2+zQ_3)$ satisfies $\Disc(F_{\u{Q}}) \neq 0$, and hence Condition \ref{cond2} holds over $\bar{\Q}$. This therefore serves as a natural example of a system for which our main theorems hold.

Let $k=10$. For $i=1,2,3$ we define $Q_i$ to be the integral quadratic form with the following associated matrix $Q_i$: 
\[ Q_1 = \left( \begin{array}{cccc}
1 &&& \\
 & 1&& \\
 && \ddots & \\
 &&&1 
\end{array} \right), \; 
Q_2 = \left( \begin{array}{cccc}
1 &&& \\
 & 2&& \\
 && \ddots & \\
 &&&k 
\end{array} \right), \;
Q_3 = \left( \begin{array}{cccc}
1 &1&& \\
 1& 1&1& \\
 && \ddots &1 \\
 &&1&1 
\end{array} \right).
\]
We may confirm that 
\beq\label{F_disc0_final}
\Disc(F_{\u{Q}}) \neq 0
\eeq
 by checking that $F_{\u{Q}}=0$ is (projectively) nonsingular; that is, (using the notation $\partial_x = \partial/\partial x$) we need to check that  the only point at which 
\beq\label{partials_3}
\partial_x F_{\u{Q}} (x,y,z) = \partial_y F_{\u{Q}} (x,y,z) =\partial_z F_{\u{Q}} (x,y,z) =0
\eeq
 is $(x,y,z)=(0,0,0)$. 
Now let $\Rcal_t(G_1,G_2)$ denote the resultant of two polynomials $G_1,G_2$ as a function of $t$, so that $\Rcal_t(G_1,G_2)=0$ if and only if there is a value of $t$ such that both $G_1(t)=G_2(t)=0$. 
 In particular,
\beq\label{Resultant}
\Rcal_y(\Rcal_x(\partial_x F_{\u{Q}}, \partial_y F_{\u{Q}}), \Rcal_x(\partial_x F_{\u{Q}}, \partial_z F_{\u{Q}}))
\eeq
is a function of $z$ that vanishes if
there are values of $x,y$  such that (\ref{partials_3}) holds.
Resultants are easily computed by computer algebra packages (e.g., {\tt{maple}}), and for example we have
\begin{align*}
\Rcal_x(\partial_x F_{\u{Q}}, \partial_y F_{\u{Q}})=-&122880000y^9\\
 \cdot &\; (2740056028310076978941971660800000000000000y^{72} \\
& +\; 24804900858187350835399766310912000000000000y^{70}z^2\\
&+\dots+29405801953440000z^{72}).
\end{align*}
In particular, we can compute that (\ref{Resultant}) is equal to $cz^{6561}$ for a certain integer $c$, so that only if $z=0$ can there exist values of $x,y$ such that (\ref{partials_3}) holds.

 After specifying $z=0$ one may explicitly compute that $\Rcal_x(\partial_x F_{\u{Q}}, \partial_y F_{\u{Q}})=0$ and $\Rcal_x(\partial_x F_{\u{Q}}, \partial_z F_{\u{Q}})=0$ then require $y=x=0$ as well, as desired. Alternatively, after specifying $z=0$, seeking solutions to (\ref{partials_3}) reduces to seeking $x,y$ such that $\partial_x \det(xQ_1 + yQ_2)= \partial_y \det(xQ_1+yQ_2)=0$. But by \cite[Proposition 2.1]{HBPierce}, we know that $\det(xQ_1+yQ_2)=0$ is (projectively) nonsingular since the ratios of the eigenvalues of $Q_1,Q_2$ are all distinct, so only $x=y=0$ may be solutions.

%

Finally, we observe that (\ref{F_disc0_final}) always fails for a system of three \emph{diagonal} quadratic forms, say $Q_1 = \diag(a_i)$, $Q_2 = \diag(b_i)$, $Q_3 = \diag(c_i)$ in $k \geq 2$ variables. For indeed, in this case 
\[ F_{\u{Q}} (x,y,z)= \prod_{i=1}^k (a_i x +  b_i y + c_iz),\]
and we may for example solve for $(x,y,z)$ in the system resulting from the first two linear factors, 
\[ a_i x + b_i y + c_i z =0, \quad i =1,2.\]
This has rank at most 2, so admits a nontrivial solution $(x_0,y_0,z_0)$, which certainly satisfies $F_{\u{Q}}(x_0,y_0,z_0)=0$ but also visibly satisfies 
\[ \partial_x F_{\u{Q}}(x_0,y_0,z_0) =  \partial_y F_{\u{Q}}(x_0,y_0,z_0)= \partial_z F_{\u{Q}}(x_0,y_0,z_0) =0\]
as long as $k \geq 2$. Thus (\ref{F_disc0_final}) fails.

%
%
%
%
%

\subsection{Construction of discriminants}

Though the construction of discriminants is classical (see e.g., \cite{GKZ}), the understanding of their basic properties over fields of arbitrary characteristic is relatively recent (see e.g., \cite{Benoist}).  For our work, we need to work over the integers so that we have a discriminant we can use for all our fields at once.  So we include the necessary background and constructions here.
 
Let $K$ be a field.  Consider a fixed number of homogeneous forms $f_i$ in $x_0,\dots,x_N$ of specified degrees.  Let
$\A_K^M$ be the parameter space of such forms, with coordinates $a_{i\alpha}$ corresponding to all coefficients of the forms.
Then the Jacobian criterion  for the nonsingularity of the system of equations $\intersect_{i}\{ f_i =0\}$
asks if each $f_i(x_0,\ldots, x_N)$ vanishes as well as certain forms (the maximal minors of the Jacobian matrix) in the $a_{i\alpha}$ and $x_j$ are $0$; let $J_K$ be the ideal of $K[a_{i\alpha},x_j]$ generated by these forms.
  For example, if there are two forms, $J_K$ is generated by 
  \begin{align*}
 & \text{$f_1(\xbf)$ and $f_2(\xbf)$,} \\
 & \partial f_1/\partial x_j ({\bf x}) \cdot \partial f_2/\partial x_k ({\bf x}) -\partial f_1/\partial x_k ({\bf x})\cdot \partial f_2/\partial x_j ({\bf x}) \quad \text{for $0\leq j <k\leq N$.}
  \end{align*} 
  The ideal $J_K$ determines a closed subscheme $S_K\subset \A_K^M \times \P_K^N$, which is the parameter space of all $(\underline{f},{\bf x})$ such that the $\underline{f}$ are forms (of our specified degrees) and ${\bf x}$
  a point of $\P_K^N$ such that $\underline{f}$ fails the Jacobian criterion at ${\bf x}$.

We consider the projection $\pi: \A_K^M \times \P_K^N \rightarrow \A_K^M $.  The set-theoretic image $\pi(S_K)$ is the set of all forms with at least one point failing the Jacobian criterion.  (Note that 
even if the forms are defined over $K$, the point ${\bf x}$ may be defined over a finite extension.)
We have an ideal $D_K:=J_K\cap k[a_{i\alpha}] $ of $k[a_{i\alpha}]$. In general such an elimination of variables describes
the closure of the image.  However, since $\P_K^N$ is projective, the map $\pi$ is closed, and thus $\pi(S_K)$ is closed and consists exactly of the points of the subscheme of $\A_K^M$ given by $D_K$. The variety described by $\sqrt{D_K}$ turns out to be codimension $1$ (as long as the forms are not all linear) \cite[Remarque 1.1]{Benoist} and irreducible \cite[Lemme 3.2]{Benoist}.  Hence we can choose an irreducible polynomial $\Disc_K$ of
$k[a_{i\alpha}]$ such that $(\Disc_K)=\sqrt{D_K}$.

Since the elimination of variables and taking radicals commute with separable field extensions, if $L/K$ is a separable field extension, we may take $\Disc_L=\Disc_K$.  For example, we take $\Disc_{\overline{\Q}}=\Disc_{\Q}$.
By minimally clearing denominators, we may take $\Disc_{\Q}$ to have integral coefficients with greatest common divisor $1$.

In order to compare these discriminants across fields, we will make a similar construction over the integers.  The maximal minors of the Jacobian matrix and the forms $f_i$ all lie in $\Z[a_{i\alpha},x_j]$, so they generate an ideal $J_\Z$  of $\Z[a_{i\alpha},x_j]$. The ideal $J_\Z$ determines a closed subscheme $S_\Z\subset \A_\Z^M \times \P_\Z^N$.  The base change of $S_\Z$ to a field $K$ is clearly $S_K$, as they are defined by the same equations. 
We again consider the projection $\pi: \A_\Z^M \times \P_\Z^N \rightarrow \A_\Z^M.$  We have that the set-theoretic image $\pi(S_\Z)$ is closed and contains $\pi(S_\Q)$.  So in particular, it contains the closure of 
$\pi(S_\Q)$ in $\A_\Z^M$.  The closure of $\pi(S_\Q)$ in $\A_\Z^M$ is described by the ideal
$\sqrt{D_\Q}\cap \Z[a_{i\alpha}]$ in $\Z[a_{i\alpha}]$, which is exactly $(\Disc_{\Q})$.
So in particular, any set of forms over any field $K$ for which $\Disc_{\Q}$ vanishes has a point (over a finite extension) for which the Jacobian criterion fails.  

Thus, it follows that in $K[a_{i\alpha}]$, we have
$(\Disc_K)\subset \sqrt{(\Disc_{\Q})}$ (note that over $K$ the ideal $(\Disc_{\Q})$ is not necessarily prime).  
Thus, in $K[a_{i\alpha}]$, we have $ \Disc_{\overline{\Q}}| \Disc_K^m$ for some integer $m$.  
Since $\Disc_K$ is irreducible, we must have $ \Disc_{\overline{\Q}}=\Disc_K^{m'}$.
 By \cite[Th\'eor\`eme 1.3]{Benoist}, we have that when $K$ has characteristic $0$ or odd characteristic that $\deg (\Disc_K)=\deg( \Disc_{\Q}),$ 
so we may take $\Disc_K=\Disc_{\Q}$ for such $K$. 
(Note this can fail in characteristic $2$.  For example, when $N=1$, in the case of a single quadratic form,
we have $\Disc_{\Q}=b^2-4ac$ and $\Disc_{\Z/2\Z}=b$.)
 We thus  call $\Disc_{\Q}$ the discriminant,
 and also denote it $\Disc$.

 \section{Acknowledgements}
The first two authors gratefully acknowledge the Hausdorff Center for Mathematics for support and an excellent work environment. Pierce is partially supported by NSF DMS-1402121. Schindler is partially supported by NSF DMS-1128155. 
Wood is supported by an 
American Institute of Mathematics Five-Year Fellowship, a Packard Fellowship for Science and Engineering, a Sloan Research Fellowship, and National Science Foundation grant DMS-1301690.

\bibliographystyle{amsbracket}
\providecommand{\bysame}{\leavevmode\hbox to3em{\hrulefill}\thinspace}

\label{endofpaper}

\end{document}